\documentclass[a4paper, 11pt]{article}
\usepackage{geometry}
\usepackage[affil-it]{authblk}
\usepackage{amsfonts,amsmath, mathtools, amssymb, mathrsfs, graphicx, geometry, bm, bbm, enumitem, float, color,ragged2e,booktabs,bbm}
\usepackage{algorithm}
\usepackage{algpseudocode}
\usepackage[mathscr]{euscript}
\usepackage{upgreek}
\usepackage{xcolor}
\usepackage{natbib}
\usepackage[titletoc,title]{appendix}
\usepackage[colorlinks=true, allcolors=blue, hidelinks]{hyperref}
\usepackage{float}
\usepackage[labelformat=simple,labelfont=bf,format=hang,font=footnotesize]{subcaption}

\usepackage{authblk} 
\providecommand{\U}[1]{\protect\rule{.1in}{.1in}}

\newtheorem{theorem}{Theorem}[section]
\newtheorem{assumption}{Assumption}[section]
\newtheorem{definition}{Definition}[section]
\newtheorem{lemma}{Lemma}[section]
\newtheorem{proposition}{Proposition}[section]
\newtheorem{corollary}{Corollary}[section]
\newtheorem{remark}{Remark}[section]
\newtheorem{example}{Example}[section]

\newenvironment{proof}[1][Proof]{\noindent\textbf{#1.} }{
\hfill \rule{0.5em}{0.5em}}
\numberwithin{equation}{section}
\allowdisplaybreaks
\counterwithin{figure}{section}
\counterwithin{table}{section}
\DeclareMathOperator*{\argmax}{arg\,max}
\DeclareMathOperator*{\argmin}{arg\,min}

\DeclareMathOperator*{\dif}{\mathrm{d\!}}
\usepackage{setspace}

\title{Cyber Risk Management and Mitigation via Controlled Stochastic SIS Dynamics: An Optimal Control Approach}
\author[1]{Shize Na}
\author[2]{Zhuo Jin}
\author[1]{Ran Xu\thanks{Corresponding author.\\
E-mail address: \\Shize.Na18@student.xjtlu.edu.cn (S. Na),\\ zhuo.jin@mq.edu.au
(Z. Jin),\\ Ran.Xu@xjtlu.edu.cn (R. Xu),\\ Hailiang.Yang@xjtlu.edu.cn (H. Yang).
}
}
\author[1]{Hailiang Yang}
\affil[1]{Department of Financial and Actuarial Mathematics, Xi'an Jiaotong-Liverpool University, Suzhou, P.R. China, 215123.}
\affil[2]{Department of Actuarial Studies and Business Analytics, Macquarie University, NSW 2113, Australia}

\date{}

\begin{document}

\maketitle
\vspace{-.25in}
\begin{abstract}
In this paper, we formulate cyber risk management and mitigation as a stochastic optimal control 
problem under a stochastic Susceptible-Infected-Susceptible (SIS) epidemic model. To capture the 
dynamics and interplay of management and mitigation strategies, we introduce two stochastic controls: (i) a proactive risk management control to reduce external cyber attacks and internal contagion effects, and (ii) a reactive mitigation control to accelerate system recovery from cyber infection. 
The interplay between these controls is modeled by 
minimizing the expected discounted running costs, which balance proactive management expenses against 
reactive mitigation expenditures. We derive the associated Hamilton-Jacobi-Bellman (HJB) equation and 
characterize the value function as its unique viscosity solution. For numerical solutions, we 
propose a Policy Improvement Algorithm (PIA) and prove its convergence via Backward Stochastic 
Differential Equations (BSDEs). Finally, we present a comprehensive numerical analysis through a benchmark example, suboptimal control analysis, sensitivity analysis, and comparative statics.
\end{abstract}


\onehalfspacing
\section{Introduction}
Cybersecurity has emerged as a critical concern in the industrial, financial, insurance, and 
governmental sectors due to the increasing digitization and interconnectedness of various cyber systems.  Currently,  there is no unanimous definition of cyber risk. Some widely accepted definitions include but not limit to: \textit{Cyber risks are operational risks that may result in potential violation of confidentiality, availability, or integrity of information systems}  \citep{cebula2010taxonomy}; \textit{it is a financial risk associated
with network and computer incidents and leading to the failure of information systems} \citep{bohme2006models,bohme2010modeling}.

In academia, the study of cyber risk has attracted the attention of researchers across 
many fields. For example, researchers in computer science have been aware of the 
importance of security in cyberspace, and have made many contributions concerning cyber risk detection \citep{moore2006inferring,garcia2009anomaly},
security breach prediction \citep{zhan2015predicting,bakdash2018malware}, and computer 
system enhancement \citep{jang2014survey}. 
In the field of business and corporate finance, 
the studies focus on investigating cyber risk under the framework of enterprise risk 
management \citep{stoneburner2002risk,gordon2003framework,ougut2011cyber,pate2018cyber}. In addition, in insurance and actuarial science, researchers are more interested in modeling the cyber risk in terms
of its frequency, severity, and dependence with a range of statistical and stochastic techniques, for example,  \cite{herath2011copula,mukhopadhyay2013cyber,eling2017data,eling2018copula,xu2019cybersecurity,dou2020insurance,malavasi2022cyber} and references therein. For a recent comprehensive cross-disciplinary review on modeling and management of cyber risk can refer to \cite{he2024modeling}.

Traditional cyber risk models (especially in the community of insurance and actuarial science) are adept at capturing statistical properties and temporal 
trends of losses, yet they often fail to account for the propagation of cyber threats. By 
studying how cyber risks spread in a closed cyber system, one can identify critical factors that 
determine the scale of contagion and the magnitude of the economic losses. This deeper 
understanding enhances the design of risk management and mitigation strategies. It is 
noted that the parallels between cyber risk and disease spread make epidemiology a natural 
framework for a similar analysis in cyber risk, given that both phenomena involve 
contagion, interdependent exposure, and dynamic evolution over time. 

Inspired by methods from genetic epidemiology, \cite{gil2014genetic} introduced a statistical framework to assess the susceptibility of individual nodes in a network. Their 
approach treats the services running on a host as the defining risk factor for cyber 
threat exposure, drawing an analogy to genetic penetrance models. This conceptual 
borrowing allows for a structured analysis of how certain network configurations increase 
vulnerability. In addition, \cite{liu2016activetrust} developed an innovative 
compartmental model for malware propagation by adapting epidemiological concepts to 
cybersecurity, where the computers are recognized as heterogeneous nodes with different protection levels in the network. The model categorizes nodes into three distinct states: 
weakly protected susceptible (W-nodes), strongly protected susceptible (S-nodes), and 
infected (I-nodes). The authors discussed the malware-free equilibrium when the ``basic 
reproduction number" (an important metric in epidemiological models) $R_0<1$ and $R_0\ge 1$ respectively. These findings align with earlier work by \cite{mishra2014dynamic}, who 
employed a susceptible-exposed-infectious-susceptible-with-vaccination model to analyze worm propagation. More recently, \cite{fahrenwaldt2018pricing} models the spread of cyber infections with an interacting Markov chain and claims with a marked point process, where the spread process is of a pure Poisson jump process, whereas transitions are described by
the susceptible-infected-susceptible (SIS) epidemic model. And, the dependence among 
different nodes (i.e., firms, computers, or devices) is modeled by an undirected network. 
Through a simulation study, the authors demonstrate that network topology plays a crucial role in determining insurance prices and designing effective risk management strategies.
\cite{xu2019cybersecurity} employs both Markov and non-Markov processes within an SIS 
network model. Their Markov formulation incorporates dual infection pathways, where 
Poisson processes are used to capture both internal network transmission and external threats. This 
formulation yields valuable dynamic upper bounds for infection probabilities and 
stationary probability estimates. The empirical validation in the study reveals the critical influence of recovery rates on insurance premium calculations. Later, \cite{antonio2021pricing} extends the Markov-based SIS model by incorporating network 
clustering coefficients, such that the model can capture the local network clustering that 
inhibits epidemic spread through modified transition probabilities. Using N-intertwined 
mean-field approximation, they derived dynamic infection probability bounds that improve 
premium accuracy when applied to both synthetic and real-world networks. \cite{hillairet2021propagation} describes the spread of a cyber attack at a global level with a susceptible-infected-recovered (SIR) model, and approximates the cumulative number of individuals in each group with a Gaussian process. Furthermore,
\cite{hillairet2022cyber} proposes a multi-group epidemic model to assess the impact of large-scale cyber attacks on insurance portfolios, which captures the interdependencies between actors and can be calibrated with limited data, enabling efficient scenario analysis of cyber events.
For other relevant studies in recent literature, the readers may refer to \cite{he2024modeling} and references therein.

Recent studies have significantly advanced the modeling of cyber risk propagation through 
epidemiological approaches, in particular, using SIS and its extended
models with network-dependent interactions. While these models effectively capture 
contagion dynamics, dependence structures, and loss distributions, they 
remain descriptive rather than prescriptive. To be specific, many studies focus on predicting the spread of contagions and estimating losses, but fail to incorporate active 
intervention or control strategies to manage or mitigate the cyber risk in real-time. Hence, a critical limitation in existing cyber risk modeling with epidemic models is the absence of formal 
control mechanisms to identify optimal risk management and mitigation strategies.
It is noted that 
this gap mirrors the recent developments in biological epidemic control literature, where the later 
works incorporate various optimal stochastic control (such as vaccination) into (non)stochastic SI(R)S models, see, for example, \cite{barnett2023epidemic, federico2024optimal,chen2025optimality,gonzalez2025tractable}.

Cyber risk poses a critical threat to interconnected cyber systems across various industrial, financial, insurance and governmental sectors, yet existing literature lacks rigorous frameworks for dynamic risk management under stochastic contagion. To bridge this gap, we formalize the cyber risk management and 
mitigation problem as a stochastic optimal control problem under a controlled Susceptible-Infected-Susceptible (SIS) model. The objective of this paper is to establish a theoretical foundation for optimizing management and mitigation controls  in cyber defense.  The contributions of our research are summarized as follows:
\begin{itemize}
\item Cyber risk has emerged as a critical threat to modern enterprises, governments, and financial systems, with attacks demonstrating potential cross-sector contagion. While recent literature employs stochastic epidemic models to study cyber risk contagion, no prior work has investigated optimal risk management and mitigation strategies under such a framework. We bridge the gap  by formulating cyber risk contagion dynamics via a controlled stochastic SIS model in a closed cyber system. Two 
distinct control variables are considered: one is a \textit{proactive} control (e.g., firewall upgrades, isolating infected systems) to prevent external attacks and internal contagion.
The other is \textit{reactive} control (e.g., malware removal, credential resets) to mitigate damage post-attack.
This unified framework rigorously captures the trade-off between prevention and mitigation, addressing a key practical challenge under resource constraints in cyber defense.  It is worth emphasizing that,  unlike those studies in epidemic control literature (see, e.g., \cite{barnett2023epidemic}), which typically incorporate a single control variable (e.g., vaccination), our work introduces a dual-control structure into the stochastic SIS framework. This represents a non-trivial extension of the controlled SIS model itself.
\item We characterize the value function (the expected discounted running costs) as the unique viscosity solution to the Hamilton-Jacobi-Bellman (HJB) equation derived from our stochastic control formulation for cyber risk management and mitigation. By applying viscosity theory, we establish the existence and uniqueness of the solution of the HJB equation under mild regularity assumptions. This analytical framework ensures numerical tractability, as the uniqueness property guarantees well-posedness for iterative computational methods via the dynamic programming principle. Moreover, our approach naturally accommodates extensions to more complex settings, including jump-diffusion processes and generalized cost structures, broadening its applicability beyond the current setup.
\item For numerical solutions of such an infinite-time horizon stochastic cyber risk control problem, we develop a Policy Improvement Algorithm (PIA) based on Bellman-Howard policy iteration. The algorithm demonstrates superior computational efficiency compared to other alternative numerical methods, such as Markov chain approximation with value iterations. We rigorously establish the algorithm’s convergence (at an exponential rate) and stability under error perturbations during iteration by mapping the value function at each step to the unique solution of some (iteratively defined) infinite-horizon 
Backward Stochastic Differential Equations (BSDEs). Given the inherent challenge of our underlying stochastic control problem (i.e., an infinite-time horizon problem with unknown boundary conditions), these results constitute a nontrivial extension of the finite-horizon frameworks, such as that of \cite{kerimkulov2020exponential}.  We further emphasize that our algorithm and the BSDE-based convergence and stability analysis can be further applied to those stochastic control problems with a random horizon or optimal stopping problems seamlessly. Finally, while the focus of this paper is on a one-dimensional diffusion process (specifically, the stochastic SIS model), all algorithmic and theoretical results extend directly to multi-dimensional controlled diffusion processes, enabling the applications of our results to more generalized compartmental models, for example, the Susceptible-Infectious-Recovered (SIR) model, the Susceptible-Infectious-Recovered-Vaccinated (SIRV) model, etc.
\end{itemize}

The remainder of the paper is organized as follows. Section \ref{sec:2} presents our controlled stochastic SIS model featuring dual control variables for cyber risk management and mitigation. We reformulate the problem as an optimal stochastic control problem under a (drift-controlled) diffusion model, and derive fundamental properties of the value function. Section \ref{sec:3} establishes our core theoretical contribution, where we prove that the value function can be characterized as the unique viscosity solution to the associated Hamilton-Jacobi-Bellman equation. In Section \ref{sec:4}, we propose a numerical method, namely the Policy Improvement Algorithm, together with convergence results based on the theory of infinite-time horizon BSDEs. We also provide various numerical results on the optimal cyber risk management and mitigation strategies, including comprehensive sensitivity analysis and comparative statics. The conclusion is given in Section \ref{sec:5}. Finally, some technical proofs are collected in the Appendices for completeness.

\section{Stochastic SIS model and mathematical formulation}
\label{sec:2}

We start with a standard Susceptible-Infected-Susceptible(SIS) model, which can be characterized by two state variables, namely 
the number of susceptible nodes (for example, terminal computers or servers in a network) 
$S_t$, and the number of cyber-infected nodes $I_t$.  Assume the total number of 
nodes in the system at time $t$ is $N_t$. Similar to the usual step in studying the 
classical deterministic SIS model in the epidemiological literature, we normalize the 
system size to one such that we replace the susceptible and infected nodes $S_t$ and $I_t$ in the system by $s_t = S_t/N_t$ and $i_t = I_t/N_t$, respectively. 
In addition, we further introduce the stochastic version of the SIS model, see e.g. \cite{tran2021optimal,barnett2023epidemic}, and the dynamics of the state variables are expressed as 
\begin{equation}\label{eq:stochastic-SIS}
    \begin{cases}
       \dif s_t = -(\alpha +\beta i_t)s_t\dif t + \gamma i_t\dif t - \sigma s_ti_t\dif W_t,\\
        \dif i_t = (\alpha +\beta i_t)s_t\dif t - \gamma i_t\dif t + \sigma s_ti_t\dif W_t,\\
        s_0 +i_0 =1,
    \end{cases}
\end{equation}
where we introduce white noise shocks $W_t$, which is a parameter perturbation of $\beta$ 
with volatility $\sigma$. Here,  $\alpha$ is the rate at which susceptible nodes (devices/users) are infected by external cyber threats (e.g., malicious emails,denial-of-service); $\beta$ denotes the rate at which the infected nodes propagate malware to 
susceptible nodes; and $\gamma$ is a baseline (unassisted) recovery rate of the compromised 
nodes become functional and threat-free, which may be due to the baseline cybersecurity 
measures like antivirus or basic IT policies. For simplicity, we assume that all the 
abovementioned model parameters in \eqref{eq:stochastic-SIS} are constants. 

Now, we extend \eqref{eq:stochastic-SIS} to a controlled stochastic SIS model with two control variables. On one hand, we allow the above SIS model to be controlled under cyber risk management protocols
through proactive protection measures. 
Let $\eta_t$ be the fraction of nodes (both susceptible and cyber-infected) at time $t$ in the system under protection with certain measures, which can 
reduce the effectiveness of both the external cyber threats to susceptible nodes and 
propagation from cyber-infected nodes to susceptible nodes.  On the other hand, we 
further replace $\gamma$ in \eqref{eq:stochastic-SIS} by $\gamma + \rho_t$, where $\rho_t$ 
is a controlled recovery rate at time $t$, which can be interpreted as the enhanced recovery rate due to certain reactive interventions to the cyber-infected nodes. 

We assume that the pair of cyber risk management and mitigation control actions $(\eta_\cdot,\rho_\cdot)$ takes value in a nonempty set $U$ in $[0,1]\times[0,\infty)$. Then, the dynamics of the controlled stochastic SIS model can be described by the following system of stochastic differential equations(SDEs):
\begin{equation}\label{eq:control-SIS}
    \begin{cases}
      \dif s_t = -\alpha\eta_ts_t\dif t - \beta \eta_t^2i_ts_t\dif t + (\gamma+ \rho_t)i_t\dif t - \sigma s_ti_t\dif W_t,\\
        \dif i_t = \alpha\eta_ts_t\dif t +\beta \eta_t^2i_ts_t\dif t - (\gamma+\rho_t) i_t\dif t + \sigma s_ti_t\dif W_t,\\
        s_0 +i_0 =1.
    \end{cases}
\end{equation}
Note that the control $\eta$ is applied to both susceptible nodes and infected nodes, hence we have a quadratic form in each of the second term in \eqref{eq:control-SIS}.

According to a similar analysis in Theorem 2.1  of \cite{tran2021optimal}, we can formulate the above controlled SIS model into a one-dimensional controlled diffusion process. To be specific, for any initial value $(s_0,i_0)\in(0,1)^2 $ satisfying $s_0+i_0=1$ and $(\eta_t,\rho_t)\equiv (\eta_0,\rho_0) \in U $, the system of SDEs \eqref{eq:control-SIS} has a unique strong  global solution $(s_t,i_t)_{t\ge 0}$, and $s_t+i_t = 1$ almost surely for any $t\ge 0$.  Therefore,  in the following, we focus on the dynamics of the fraction of cyber-infected nodes in the system,
\begin{equation*}
\dif i_t = \Big[\alpha\eta_t(1-i_t) + \beta \eta_t^2i_t(1-i_t) - (\gamma+\rho_t)i_t\Big]\dif t+\sigma (1-i_t)i_t\dif W_t.
\end{equation*}
We shall provide a rigorous proof of the above assertion in Proposition \ref{prop:2.1} below. To proceed with our analysis under simplified notations, we reformulate the above cyber risk management and mitigation problem as a stochastic control problem for a general diffusion process given in \eqref{eq:SDE-X} below.

Let $(\Omega,\mathcal{F},\mathbb{F}=\{\mathcal{F}\}_{t\ge0},\mathbb{P})$ be a filtered probability space satisfying the usual condition, and on which a one--dimensional Brownian motion $W = (W_t)_{t\ge0}$ is defined. Let  $X=(X_t)_{t\ge 0}$ denote the fraction of cyber-infected nodes in the system at time $t$, which is driven by the following controlled diffusion process:
\begin{equation}\label{eq:SDE-X}
    \dif X_t = b(X_t, \eta_t, \rho_t) \dif t + \sigma(X_t) \dif W_t, \qquad  X_0 = x\in(0,1),
\end{equation}
where $ b(x, \eta, \rho)= \eta \alpha (1 - x) + x[\eta^2 \beta(1 - x) - (\gamma + \rho)]$, $\sigma(x)= \sigma x(1 - x)$ (with abuse of notation, we use the volatility 
parameter $\sigma$ in \eqref{eq:control-SIS} as the function $\sigma(\cdot)$ here in \eqref{eq:SDE-X} for notation simplicity), and $\alpha, \beta, \gamma$ and $\sigma$ are 
model parameters in the stochastic SIS system \eqref{eq:control-SIS}. The control processes $(\eta_{.}, \rho_{.}): \mathbb{R}_+\times\Omega \to U$ are progressively 
measurable, where $U$ is the control action space, we may just set  $U=[0,1]\times[0,\infty)$, $\mathbb{R}_+$ denotes the set $[0,\infty)$.  Furthermore, throughout the paper, we let $\mathbb{P}_x$ and $\mathbb{E}_x$ denote the probability measure and expectation operator when $X_0 = x$, respectively.

\begin{lemma}\label{lemma:a1}
    The mappings $x\mapsto b(x, \eta, \rho)$ and $x\mapsto \sigma(x)$ are continuous in $x$, and the former is uniformly in the control $(\eta, \rho)\in U$. There exists a constant $K_1 > 0$ such that for any $(\eta,\rho)\in U$, and all $x,y\in(0,1)$ we have 
    \begin{equation}\label{eq:lipchtiz-coef}
        |b(x, \eta, \rho) -  b(y, \eta, \rho)| + |\sigma(x) - \sigma(y)| \leq K_1|x-y|.
     \end{equation}
    Moreover, for all $(\eta, \rho)\in U$ and $x\in (0, 1)$, it also holds that
    \begin{equation*}
        |b(x, \eta, \rho)| \leq K_2(1 + |x| + |\eta| + |\rho|), 
    \end{equation*}
    and 
    \begin{equation*}
        |\sigma(x)| \leq K_2(1 + |x|),
    \end{equation*}
    for some $K_2>0$.
\end{lemma}

\begin{proof}
   The continuity and uniform continuity in control are obvious, we just need to prove the result inequalities. Consider any $x,y\in(0,1)$, one has
 \begin{align*}
        |b(x, \eta, \rho) -  b(y, \eta, \rho)| &= \left|\left(\eta^2 \beta - \eta\alpha - (\gamma + \rho)\right)(x - y)   - \eta^2 \beta (x^2 - y^2)\right|\\
        &\le \left|\left(\eta^2 \beta - \eta\alpha - (\gamma + \rho)\right)\right| |x-y| + \eta^2\beta(x+y)|x-y|\\
        &\leq \Big(\left|\left(\eta^2 \beta - \eta\alpha - (\gamma + \rho)\right)\right| + 2\eta^2\beta\Big) |x - y|, 
    \end{align*}
and
\[
|\sigma( x) - \sigma(y)| = |\sigma(x - y) - \sigma(x^2 - y^2)|\leq \sigma\Big(1 + (x+y)\Big)|x-y|\le 3\sigma|x - y|.
\]
Hence, by letting $K_1 = \max \left\{\Big(\left|\left(\eta^2 \beta - \eta\alpha - (\gamma + \rho)\right)\right| + 2\eta^2\beta\Big),  3\sigma\right\}$, one arrives at \eqref{eq:lipchtiz-coef}. Moreover, take any $x\in (0, 1)$ and $(\eta, \rho)\in U$, we have
\begin{equation*}
    \begin{split}
        |b(x, \eta, \rho)| & =  \left|\eta \alpha (1 - x) + x\left[\eta^2 \beta(1 - x) - (\gamma + \rho)\right] \right| \\
        &\leq \left|\eta [\alpha (1 - x) + x\beta(1 - x)]\right| + \left|x(\gamma + \rho)\right|\\
        &\leq (\alpha + \beta)\cdot|\eta| + \gamma |x| + |\rho|\\
    \end{split}
\end{equation*}
and 
$$|\sigma(x)|= |\sigma x(1 - x)| \leq \sigma|x|.$$
Then, by letting $K_2 = \max\{(\alpha+\beta), \gamma, 1, \sigma \}$, we complete the proof.
\end{proof}

Now, we are ready to define the cost functional associated with our cyber risk management and mitigation problem \eqref{eq:SDE-X} as
\begin{equation}\label{eq:performfunc}
    J(x;\eta,\rho) := \mathbb{E}_x\left[\int_{0}^{\infty} e^{-\delta t}f(X_t,\eta_t,\rho_t)\dif t \right], \quad x\in (0,1),(\eta,\rho)\in U,
\end{equation}
where $\delta>0$ is the discounting factor, $f: (0,1)\times [0,1]\times [0,\infty) \to [0,\infty)$ is a running cost function. 
In addition, without loss of generality, we restrict our study to the set of admissible controls given below. Let $\mathcal{U}$ denotes all progressively measurable random processes $(u, \nu) = \{(u_t,\nu_t), t \geq 0\}$ taking values in $U$, and define the set of admissible control processes by
$$\mathcal U_0:=\left\{(u, \nu)\in \mathcal U: |u_s| + |\nu_s|\leq M_{u, \nu} <\infty \; \text{a.s. uniformly in}\; s\in\mathbb{R}_+\right\}. $$
Then, the value function is defined as 
\begin{equation}\label{problem}
    V(x) := \inf_{(\eta, \rho)\in \mathcal{U}_0} J(x; \eta, \rho),\qquad x\in(0,1).
\end{equation}
To show the well-posedness of the controlled diffusion process \eqref{eq:SDE-X}, we prove the following proposition. 

\begin{proposition} \label{prop:2.1} 
For any control $(\eta, \rho)\in \mathcal{U}_0$ and initial value $X_0=x\in(0,1)$, the controlled SDE \eqref{eq:SDE-X} has a unique global positive solution $X_t$ for all $t\in \mathbb R_+$  such that
\[
\mathbb P_x(X_t \in (0, 1), \forall t\in \mathbb R_+)= 1.
\]
\end{proposition}
\begin{proof}
Let $\tau$ be the explosion time of the $X_t$ driven by \eqref{eq:SDE-X}. Since the drift and diffusion coefficients are locally Lipschitz, there exists a unique local solution $X_t$ for $t\in[0,\tau)$. For any initial value $X_0=x\in(0,1)$, consider a sufficiently large $N>0$ such that  $x\in(\frac{1}{N},1-\frac{1}{N})$. Then, for each $n\ge N$, define the first exit time 
\[
\tau_n := \inf\Big\{t\in[0,\tau): X_t\notin \left(\frac{1}{n}, 1- \frac{1}{n}\right)\Big\},
\]
with the convention that $\inf\{\emptyset\}=\infty$. Obviously, $\tau_n$ is increasing in $n$, hence we let $\tau_\infty = \lim_{n\to \infty}\tau_n$, and $\tau_\infty\le \tau$ almost surely. Then, the rest is to show $\tau_\infty = \infty$ almost surely. We prove the statement by the method of contradiction.  Assume that there exists $T>0$ and $\epsilon\in(0,1)$ such that
\[
\mathbb{P}_x(\tau_\infty \le T)>\epsilon.
\]
Then, there is a sufficiently large $\tilde{n}\ge N$ such that $\mathbb{P}_x(\tau_n\le T)\ge \epsilon$ for all $n\ge \tilde{n}$. We define an auxiliary function $v(z):= \frac{1}{z}+ \frac{1}{1-z}$ for $z\in(0,1)$, then for any $t\in[0,T]$ and $ n\ge \tilde{n}$, an application of the Dynkin's formula gives,
\[
\mathbb{E}_x(v(X_{t\wedge \tau_n})) = v(x) + \mathbb{E}_x\left[\int_{0}^{t\wedge\tau_n}\mathcal{L}^{\eta,\rho}v(X_s)\dif s\right],
\]
where $\mathcal{L}^{\eta,\rho}$ is the infinitesimal generator of $X$  for any fixed $\eta$ and $ \rho$ given as
\begin{align*}
  \mathcal{L}^{\eta,\rho}v(x) =& \eta\alpha\left(-\frac{1}{x^2} + \frac{1}{(1 - x)^2}\right)(1 - x) + x\left(-\frac{1}{x^2} + \frac{1}{(1 - x)^2}\right)[\eta^2\beta(1 - x) - (\gamma + \rho)]  \\
  &\quad  +\sigma^2x^2(1 - x)^2\left(\frac{1}{x^3} + \frac{1}{(1 - x)^3}\right).
\end{align*}
With some simple algebra, one has
\[
 \mathcal{L}^{\eta,\rho}v(x)\le \frac{\eta\alpha + \eta^2\beta}{1 - x} + \frac{\gamma + \rho}{x} + \sigma^2 \left(\frac{1}{x} + \frac{1}{1 - x}\right) \leq C(\eta,\rho) v(x), \quad \text{for } x\in(0,1),
\]
where $C(\eta,\rho) = \max\{\eta\alpha + \eta^2\beta,\gamma + \rho\} + \sigma^2$. Take a finite modification of random process $C(\eta_{s\wedge\tau}, \rho_{s\wedge\tau}), s\geq 0$ by defining
\[
\widetilde C(\eta_{s\wedge\tau}, \rho_{s\wedge\tau}) :=
\begin{cases}
C(\eta_{s\wedge\tau}, \rho_{s\wedge\tau}), & \omega \in \Omega',\\[1ex]
M_{u, \nu}, & \omega \in \Omega \setminus \Omega',
\end{cases}
\]
where $\mathbb P_x\left(\widetilde C(\eta_{s\wedge\tau}, \rho_{s\wedge\tau}) = C(\eta_{s\wedge\tau}, \rho_{s\wedge\tau})\right) = 1$ for every $s\geq 0$. Then $\widetilde C = C$ almost surely, but $\widetilde C$ is bounded by $M_{u,\nu}<\infty$ everywhere. By Tonelli's theorem applied to the nonnegative process $\widetilde C(\eta_{s\wedge\tau_n}, \rho_{s\wedge\tau_n}) \cdot v(X_{s\wedge\tau_n})$, we have
\begin{equation*}
\begin{split}
\mathbb{E}_x\big[v(X_{t\wedge \tau_n})\big] 
&= v(x) + \int_0^t \mathbb{E}_x\Big[\widetilde C(\eta_{s\wedge\tau_n}, \rho_{s\wedge\tau_n}) \cdot v(X_{s\wedge\tau_n})\Big] \, \dif s \\
&\le v(x) + \int_0^t \mathbb{E}_x\Big[ \sup_{\omega \in \Omega} \left\{\widetilde C(\eta_{s\wedge\tau_n}, \rho_{s\wedge\tau_n}) \right\}\cdot v(X_{s\wedge\tau_n}) \Big] \, \dif s \\
&\le v(x) + M_{u, \nu} \int_0^t \mathbb{E}_x\big[v(X_{s\wedge\tau_n})\big] \, \dif s,
\end{split}
\end{equation*}
where the last inequality follows from the boundedness of $\widetilde C$ on all $\omega \in \Omega$ by the definition of the admissible control set $\mathcal{U}_0$.  
 In particular, the joint measurability of $C(\eta_{s \wedge \tau_n}, \rho_{s \wedge \tau_n})$ and $v(X_{s \wedge \tau_n})$ w.r.t. product $\sigma$-algebra $\mathcal B([0, u]) \otimes \mathcal F_u,\forall u\geq 0$, together with the bounds
$$
C(\eta_{s \wedge \tau_n}, \rho_{s \wedge \tau_n}) \geq \gamma + \sigma^2 > 0,
\quad
v(X_{s \wedge \tau_n}) \geq \frac{1}{n} > 0 \quad \text{a.e.},
$$
ensure the validity of applying Tonelli’s theorem.

Then, with the help of Gr{\"o}nwall's inequality, one has
\begin{equation}\label{ineq: gronwall}
v(x)e^{\overline{C}T}\ge \mathbb{E}_x(v(X_{T\wedge \tau_n})\ge \mathbb{E}_x\Big(\bm{1}_{\{\tau_n\le T\}}v(X_{\tau_n})\Big)= n\mathbb{P}_x(\tau_n\le T)\ge n\epsilon,
\end{equation}
where we applied the fact that $X_{\tau_n}$ equals to either $\frac{1}{n}$ or $1-\frac{1}{n}$. Then by sending $n\to \infty$ in \eqref{ineq: gronwall}, one arrives at the contradiction $\infty >v(x)e^{\overline{C} T} = \infty$, and completes the proof.
\end{proof} 

\begin{assumption}\label{a2}
    There exists a constant $K > 0$, and $m\in \mathbb N$ such that for any pair of controls $ (\eta,\rho)\in U$, we have 
    $$|f( x, \eta, \rho) - f(y, \eta, \rho)| \leq K (1 + |x|^m + |y|^m)|x - y|, $$
    
    $$|f(x, \eta, \rho)| \leq K(1 + |x| + |\eta|^2 + |\rho|^2), $$ for all $x, y\in (0, 1)$. Additionally, let $\delta > 0$, the cost function $f( x, \eta, \rho)$ is  bounded by a constant $C_f$ such that
    $$\mathbb E_x\left[\int_0^\infty e^{-\delta t}|f(X_t, \eta_t, \rho_t)|\dif t\right] < \infty. $$
\end{assumption}

\begin{remark}
A typical form of the cost function we shall investigate later in this paper can be expressed as  
\begin{equation}\label{eq:costf}
    f(x, \eta, \rho):= a_0 +a_I x+ a^S_{m}(1-\eta)^2 + (a^I_{m} - a^S_{m})x(1 - \eta)^2 +a_{r}x\rho^2, 
\end{equation}
where $a_0>0$ is the underlying marginal costs of running the system, $a_I>0$ is the 
marginal cost generated from infected nodes, $a^S_{m}$ and $a^I_{m}$ are the marginal costs 
associated with cyber risk management control for susceptible nodes and cyber-infected nodes, respectively. We shall assume that $a^I_{m}> a^S_{m}>0$, which means that costs associated 
with management for infected nodes are, in general, higher than those for 
functional and threat-free nodes. And, $a_r>0$ is the marginal cost of cyber risk mitigation control.
Not that Assumption \ref{a2} is satisfied for the class of cost functions defined in \eqref{eq:costf}. 

Note that,  in many operational risk models, the cost of infection is often assumed to be linear in the number of infected nodes because each infection event incurs a direct cost (e.g., downtime, data loss) that is proportional to the number of infected units. On the other hand, quadratic costs are standard in stochastic control literature because they represent the idea that the marginal cost of control increases with the intensity of control. In the context of cyber risk, increasing prevention measures ($\eta$) might require more expensive technologies or more highly skilled personnel, leading to increasing marginal costs. Similarly, reactive mitigation ($\rho$) might involve overtime pay for IT staff or the cost of rapid deployment of resources, which also has increasing marginal costs.
\end{remark}

To proceed, we show some important properties of the value function in the following proposition. Throughout the paper, we assume that the Assumption \ref{a2} holds true. 
\begin{proposition}\label{pp1}
\begin{enumerate}
\item[(i)] 
 If the running cost function $f(x, \cdot, \cdot)$ is nondecreasing in $x$ for each pair of admissible controls, then $V(x)$ is nondecreasing in $x$ for all $x\in(0,1)$.
 \item[(ii)]
For any $x,y\in(0,1), m\in\mathbb{N}$, there exists a constant $C>0$ such that
 \begin{equation}\label{ineq:V}
        |V(x) - V(y)|\leq C(1 + x^m + y^m)|x - y|.
    \end{equation}
   \end{enumerate}
\end{proposition}
\begin{proof}  
See Appendix \ref{App:A}.
\end{proof}

\section{Hamilton-Jacobi-Bellman equation and the viscosity solution }
\label{sec:3}
In this section, we first state the dynamic programming equation associated with our stochastic control problem and derive (heuristically) the corresponding Hamilton-Jacobi-Bellman (HJB) equation.
Let $\theta$ be any $\mathbb{F}$-stopping time, for any $x\in(0,1)$,  we have 
\begin{equation}\label{eq:DPP}
    V(x) = \inf_{(\eta, \rho)\in \mathcal U_0} \mathbb E_x\left[\int_0^\theta e^{-\delta t} f(X_t,\eta_t,\rho_t)\dif t + e^{-\delta\theta}V(X_\theta)\right]
\end{equation}
The proof of \eqref{eq:DPP} for controlled diffusion processes follows the classical arguments in the literature, see for example, Theorem 3.1.6 of \cite{krylov1980controlled}, where the main point is the continuity of the value function, which is the present case in our study.

If the value function is sufficiently smooth, then by following the standard arguments with the help of the dynamic programming principle and It\^o's formula, we obtain the following HJB equation,
\begin{equation}\label{eq:HJB}
\inf_{(\rho,\eta)\in U}\left\{b(x,\eta,\rho)u_x(x) + \frac{1}{2}\sigma^2(x) u_{xx}(x) - \delta u(x) + f(x, \eta, \rho) \right\} = 0,\quad x\in(0,1).
\end{equation}
Note that the heuristic arguments verifying that value function $V$ is a classical solution to the HJB equation \eqref{eq:HJB} assume in prior that $V$ is twice continuously differentiable on $(0,1)$, which is not the present case, where we only have Lipchitz continuity in $V$. Hence, we adopt the concept of the viscosity solution introduced in \cite{crandall1983viscosity} and characterize the optimal value function $V$ as the unique viscosity solution to the HJB equation \eqref{eq:HJB}.  

\begin{definition}\label{def:vis}
Let $u:(0,1) \to \mathbb{R}$ be a locally Lipschitz continuous function.
\begin{enumerate}
    \item[(i)]  $u$ is local viscosity supersolution of \eqref{eq:HJB} at $x\in(0,1)$, if for any $\varphi\in C^2((0,1))$ with $u\ge \varphi$, whenever $u-\varphi$ attains a local minimum at $x $, we have
 \[
 \inf_{(\rho,\eta)\in U}\left\{b(x,\eta,\rho)\varphi_x(x) + \frac{1}{2}\sigma^2( x) \varphi_{xx}(x) - \delta \varphi(x) + f(x, \eta, \rho)\right\} \le 0.
 \]
 \item[(ii)] $u$ is a local viscosity subsolution of \eqref{eq:HJB} at  $x\in (0,1)$, if for any $\phi\in C^2((0,1))$ with $u\le \phi$, whenever $u-\phi$ attains a local maximum at $x$, we have
 \[
 \inf_{(\rho,\eta) \in U}\left\{b(x,\eta,\rho)\phi_x(x) + \frac{1}{2}\sigma^2( x) \phi_{xx}(x) - \delta \phi(x) + f(x, \eta, \rho) \right\} \ge 0.
 \]
 \item[(iv)]  $u$ is a viscosity solution of \eqref{eq:HJB} on $(0,1)$ if it is both a viscosity supersolution and subsolution of \eqref{eq:HJB} on  all $x\in (0,1)$.
\end{enumerate}
\end{definition}
We first provide the result regarding the existence of the viscosity solution for the HJB equation \eqref{eq:HJB} on $(0,1)$.
\begin{proposition}
\label{prop:existence}
    The value function $V$ is a viscosity solution of \eqref{eq:HJB} on $(0,1)$.
\end{proposition}
\begin{proof}
See Appendix \ref{app:existence}.
\end{proof}

To prove the uniqueness, we introduce the following alternative definition of viscosity solution (for second-order differential equations), see, for example, \cite{yong1999stochastic}.
For any function $u\in C((0,1))$ and $x\in(0,1)$, the so-called \textit{second--order superdiffiential} of $u$ at $x$ is defined as 
\[
D^{2,+}_x u(x):= \Bigg\{(p,q)\in\mathbb{R}\times\mathbb{R}: \lim\sup_{h\to 0}\frac{u(x+h) -u(x) - ph -\frac{q}{2}h^2 }{h^2} \le 0 \Bigg\},
\]
and the \textit{second--order subdiffiential} of $u$ at $x$ is defined as
\[
D^{2,-}_x u(x):= \Bigg\{(p,q)\in\mathbb{R}\times\mathbb{R}: \lim\inf_{h\to 0}\frac{u(x+h) -u(x) - ph -\frac{q}{2}h^2 }{h^2} \ge 0 \Bigg\}.
\]
In addition, we let 
 \[
 F(x,u(x),u_x(x),u_{xx}(x))  :=\inf_{(\rho,\eta)\in U}\left\{b(x,\eta,\rho)u_x(x) + \frac{1}{2}\sigma^2( x) u_{xx}(x) - \delta u(x) + f(x, \eta, \rho) \right\},
 \]
 and rewrite \eqref{eq:HJB} as
\begin{equation}\label{eq:HJB2}
 F(x,u(x),u_x(x),u_{xx}(x)) = 0,\quad \text{for } x\in(0,1).
\end{equation}
\begin{definition}
 A Lipschitz continuous function $u: (0,1)\to \mathbb{R}$ is called a viscosity supersolution of \eqref{eq:HJB2}  at $x\in (0,1)$ if 
 \[
F(x,u(x),p,q) \le 0,
 \]
 for all $(p,q)\in D^{2,-}_x u(x)$.

 A  Lipschitz continuous function $u: (0,1)\to \mathbb{R}$ is called a viscosity subsolution of \eqref{eq:HJB2} at $x\in(0,1)$ if 
 \[
 F(x,u(x),p,q) \ge 0,
 \]
 for all $(p,q)\in D^{2,+}_xu(x)$.
 
If $u$ is both a viscosity supersolution and a viscosity subsolution of \eqref{eq:HJB2} at $x\in(0,1)$, then it is a viscosity solution at $x$.
\end{definition}
\begin{proposition}(Comparison Principle)
\label{prop:comparison}
    Let the increasing and Lipschitz continuous functions $u$ and $v$ be a viscosity supersolution and a viscosity subsolution of the HJB equation \eqref{eq:HJB} (as well as \eqref{eq:HJB2}) respectively. For any closed interval $\mathcal{O}\subset (0,1)$, if $u\ge v$ on $\partial \mathcal{O}$, then $u\ge v$ on $\mathcal{O}$.
\end{proposition}
\begin{proof}
See Appendix \ref{app:unique}. 
\end{proof}

\begin{proposition}\label{prop:verification}
   Let $v\in C((0,1))$ be any increasing and Lipschtiz continuous viscosity supersolution of \eqref{eq:HJB}, then $v(x)\ge V(x)$ for all $x\in(0,1)$.
\end{proposition}
\begin{proof}
    The proof follows the standard arguments of using It\^o's formula with a density argument dealing with the nonsmoothness of any viscosity supersolution $v$ of \eqref{eq:HJB}, see, for example, \cite{nguyen2004some, azcue2005optimal}. We omit the details here.  
\end{proof}

By Proposition \ref{prop:verification}, for any closed interval $\mathcal{O}\subset (0,1)$, we can characterize the value function $V$ as the (local) viscosity solution of \eqref{eq:HJB} with the smallest value on the boundary $\partial  \mathcal{O}$ in the class of increasing and Lipschitz continuous viscosity solutions of \eqref{eq:HJB}. Let $\mathcal{V}$ denote the set of all increasing and Lipschitz continuous functions  that are  (local) viscosity solutions of \eqref{eq:HJB} on $\mathcal{O}$, then, we characterize the value function $V$ as  
\begin{equation}\label{eq:charaV}
    V = \{v\in \mathcal{V}~|~ v(x)\le h(x) \text{ for all } h\in\mathcal{V} \text{ and } x\in\partial\mathcal{O} \}.
\end{equation}
\begin{proposition}\label{prop:uniqueVS}
    For any closed interval $\mathcal{O}\in(0,1)$, the value function $V$ characterized by \eqref{eq:charaV} is the unique viscosity solution of \eqref{eq:HJB} on $\mathcal{O}$ in the class of increasing and Lipschitz continuous functions.
\end{proposition}
\begin{proof}
Let $\varphi$ be another increasing and Lipschitz continuous viscosity solution $\varphi$ of \eqref{eq:HJB} on $\mathcal{O}$ satisfying \eqref{eq:charaV}. On one hand, $\varphi$ is a viscosity supersolution, and by Proposition  \ref{prop:verification}, we have $\varphi\ge V$ on $\mathcal{O}$. On the other hand, $\varphi$ is also an increasing and Lipschitz continuous viscosity subsolution of \eqref{eq:HJB} with the fact that $\varphi(x)\le V(x)$ for $x\in\partial \mathcal{O}$ (since $V\in \mathcal{V}$), then by the comparison principle given in Proposition \ref{prop:comparison}, we have $V\ge \varphi$ on $\mathcal{O}$, whenever $V$ is considered as a viscosity supersolution and $\varphi$ as a viscosity subsolution. Therefore, we obtain $\varphi=V$ on $\mathcal{O}$ and complete the proof.
\end{proof}

\begin{proposition}\label{prop:uniqueVS-2}
   The value function $V$ given by \eqref{problem}  is the unique viscosity solution of the HJB equation \eqref{eq:HJB} on $(0,1)$.
\end{proposition}
\begin{proof}
    Let $u \in C((0,1))$ be any increasing and Lipschtiz continuous viscosity solution of \eqref{eq:HJB}. We aime to show $u=V$ on $(0,1)$. Fix any arbitrary point  $x_0\in(0,1)$, there exists  a closed interval $\mathcal{O}\subset(0,1)$ such that $x_0\in \mathcal{O}$. By Propositions \ref{prop:verification} and \ref{prop:uniqueVS}, we have $V$ is the unique viscosity solution on $\mathcal{O}$. Becasue viscosity solutions are defined locally (see Definition \ref{def:vis}), we have $u$ retricted on $\mathcal{O}$ is also a viscosity solution of the HJB equation \eqref{eq:HJB} on $\mathcal{O}$. By the uniqueness, we must have $u=V$ on $\mathcal{O}$, or in particular $u(x_0)=V(x_0)$, Then, by the arbitrariness of $x_0$ in $(0,1)$, we arrive at $u=V$ on $(0,1)$. Hence, $V$ is the unique viscosity solution of \eqref{eq:HJB} on $(0,1)$.
    
\end{proof}

\section{Numerical algorithm and examples}
\label{sec:4}
\subsection{Policy improvement algorithm}
Obtaining a closed-form solution to \eqref{eq:HJB} is seldom feasible; hence, in this 
paper, we apply a policy improvement algorithm, namely Bellman--Howard policy 
improvement/iteration algorithm (see Algorithm \ref{PIA} below), to solve the cyber risk management and mitigation 
problem numerically.  Iterative algorithms for solving optimal control problems can trace 
their origins to Bellman’s pioneering work \cite{bellman1955functional,bellman1966dynamic}, which introduced value iteration methods 
for finite space-time problems and established their convergence properties. \cite{howard1960dynamic} later developed the policy improvement algorithm in the context 
of discrete space-time Markov decision processes (MDPs). 

The proof of the convergence is of paramount importance in using policy improvement algorithms. Among the earliest convergence analyses for policy iteration in MDPs is the work of \cite{puterman1979convergence}, which employed an abstract function--space framework applicable to both discrete and continuous settings. A key insight from their study is that policy iteration can be interpreted as a form of Newton’s method, inheriting similar convergence properties, i.e., whenever initialized near the true solution, the algorithm achieves quadratic convergence. Later \cite{puterman1981convergence} provides similar results on the convergence of the policy iteration algorithm for controlled diffusion processes. Further extensions were made by \cite{santos2004convergence}, who examined discrete-time problems with continuous state and control spaces. Their work generalizes the results of \cite{puterman1979convergence}, demonstrating global convergence while retaining local quadratic convergence rate under standard conditions and superlinear convergence rate under broader assumptions.  More recently, under a setting of continuous space--time controlled diffusion processes,  \cite{kerimkulov2020exponential} established a global rate of convergence and stability of the Bellman--Howard policy iteration algorithm with the help of techniques in Backward Stochastic Differential Equations (BSDEs). Therefore, we follow the main steps in \cite{kerimkulov2020exponential} when proving the convergence of our policy iteration algorithm, the main theorem is given in Theorem \ref{them:convegencePIA} below. Note that the study in \cite{kerimkulov2020exponential} focuses on finite-horizon stochastic control problems; a brief discussion of the infinite-horizon counterpart can be found in Remark 4.3 of \cite{kerimkulov2020exponential}, which suggests that the convergence result still holds for some sufficiently large discount factor $\delta > 0$, although no formal proof is provided.

For numerical illustration, we consider a sufficiently large closed interval $\mathcal{O}=[\underline{x},\overline{x}]\subset (0,1)$. Further, we let the running cost function $f$ be in the form of \eqref{eq:costf},  
\begin{equation*}
    f(x, \eta, \rho)= a_0 +a_I x+ a^S_{m}(1-\eta)^2 + (a^I_{m} - a^S_{m})x(1 - \eta)^2 +a_{r}x\rho^2, \quad  (x,\eta,\rho)\in \mathcal{O}\times U,
\end{equation*}
where $a_0$ is the baseline marginal costs associated with the management of the system;  $a_I$ is 
the (extra) marginal running costs incurred by cyber-infected nodes in the system; $a_m^I $ and $a_m^S$ denote the marginal costs associated with proactive risk management for cyber-infected 
nodes and susceptible (functional and threat-free) nodes, respectively;  $a_r$ denotes the marginal 
costs generated by reactive risk mitigation (intervention) to enhance the recovery rate in the system.
The Policy Improvement Algorithm (PIA) for solving \eqref{eq:HJB} on $\mathcal{O}$ is given in Algorithm \ref{PIA}.
    \begin{algorithm}[H]
\caption{Policy improvement algorithm}
\label{PIA}
\begin{minipage}{0.9\linewidth}
\begin{algorithmic}[1]
\State \textbf{Initialize}: Set spatial discretizations $\{x_k\}_{k=1}^N$ with $N$ grid points of $[\underline{x},\overline{x}]$. Choose an initial guess of control $ (\eta^0, \rho^0) \in \mathcal{U}_0$, constant for each $x_k$ 
\Repeat
    \State Solve the Bellman--type ODE:
    \begin{equation}\label{pia}
        \frac{1}{2} \sigma^2(x) D_{xx} v^{n}(x) + b(x,\eta^n,\rho^n) D_x v^{n}(x) + f(x,\eta^n,\rho^n) - \delta v^{n}(x) = 0, \quad x \in [\underline{x}, \overline{x}].
    \end{equation}

    \State Control update by the first-order conditions \footnote{\noindent This is straightly achieved by solving a convex optimization problem. In other words, minimize a quadratic Hamiltonian over a convex and compact set. }
    \[
    \eta^{n+1}(x) =
    \begin{cases}
        1 - \frac{(\alpha + 2\beta x)(1 - x) D_x v^{n}}{2 \left[(a_m^I - a_m^S)x + a^S_m + \beta x (1 - x) D_x v^n\right]} & \text{if in } [0, 1], \\
        0 & \text{if below 0}, \\
        1 & \text{otherwise}.
    \end{cases}
    \]
    \[
    \rho^{n+1}(x) =
    \begin{cases}
        \frac{D_x v^{n}}{2 a_r x} & \text{if } D_x v^{n} > 0, \\
        0 & \text{otherwise}.
    \end{cases}
    \]
\Until The normalized $L^2$-norm between two nearest iterations 
\[
    \frac{\|v^{n+1} - v^n\|_{2}}{\sqrt{N}} < \epsilon.
\]
\State \textbf{return:} $v^n$ and $(\eta^{n+1},\rho^{n+1})$.
\end{algorithmic}
\end{minipage}
\end{algorithm}

We shall remark that to solve the Bellman-type ODE \eqref{pia} during each iteration step of Algorithm \ref{PIA}, we need to impose boundary conditions which are not available in our problem\footnote{The conventional comparison principle states that the value function is actually the solution to ODE located on $[\underline{x}, \overline{x}]$ as well.}. Hence, we impose an approximated Dirichlet condition at the left boundary ($\underline{x}$) and a Neumann condition at the
right boundary ($\overline{x}$) based on the conventional Monte-Carlo simulation method\footnote{Any
$n$th-order ordinary differential equation can be reduced to an equivalent system first-order ODEs. If the corresponding vector field $\begin{bmatrix}
D_x v^n \\
\displaystyle \frac{2}{\sigma^2}
\left(\delta v^n - f(,\eta^n,\rho^n) - b(,\eta^n,\rho^n) D_x v^n\right)
\end{bmatrix}
$ is locally Lipschitz, the associated initial value problem is well-posed and admits a unique solution (see chapter 1.6 of \cite{CoddingtonLevinson1955}). Consequently, a two-point boundary value problem may be equivalently formulated as a shooting problem by treating the missing initial condition as a free parameter.} for the performance function \eqref{eq:performfunc} under the current control strategy (extrapolated for all $x\in(0,1)$) at each iteration step.  We can show (numerically) in the Remark \ref{re:MCintervals} below that the value functions (subject to possible simulation errors on the boundaries) are insensitive to the choices of the closed interval $[\epsilon,1-\epsilon]$ as long as $\epsilon$ is sufficently close to zero (see Table \ref{tab:placeholder}).

For completeness, we present the convergence theorem of the Algorithm~\ref{PIA} in Theorem \ref{them:convegencePIA}, together with the policy improvement theorem (Corollary \ref{them:PI_infinite}) and algorithm stability under perturbations to the solution of the Bellman-type ODE \eqref{pia} (Corollary \ref{them:exp_convergence_infinite}).  For notation simplicity, we write $u = (\eta,\rho)\in U$, where $U=[0,1]\times[0,\infty)$ is the action space of our cyber risk control problem \eqref{eq:performfunc},  $b^{u}(x) = b(x,\eta,\rho)$, $f^u(x)=f(x,\eta,\rho)$, which are the drift term and cost function in \eqref{eq:HJB}, and $\sigma^{-1}(x)=1/\sigma(x)$.

\begin{assumption}\label{a3}
For each fixed $ (x, z)\in (0, 1)\times \mathbb R$, we assume the function
$$u(x, z) :=\argmin_{u\in U} \left\{\big(b^{u}\sigma^{-1}\big)(x)z + f^{u}(x)\right\}, $$ is measurable.  
\end{assumption}
 One can refer to \cite{kerimkulov2020exponential} for a short discussion on the validity of the above measurability. 

\begin{assumption}\label{a4}
There are constants $K, \theta\ge 0$ such that the following hold: 
    \begin{enumerate}
        \item[(i)]For $x\in (0, 1)$, $u, u'\in U$, 
        $$|b^{u}(x) - b^{u'}(x)|\leq \sqrt{\theta}|u- u'|, $$ and for all $x\in (0, 1)$, $u\in U$, we have $$|\big(b^{u}\sigma^{-1}\big)(x)|< K.$$
        \item[(ii)] For all $x, x'\in (0, 1), z, z'\in \mathbb R$, and $u\in  U$ we have that
        \begin{align*}
         |u( x, z) - u( x', z)|&  \leq K|x - x'|,\\
            |u(x, z) - u(x, z')| & \leq \sqrt{\theta}|z - z'|.
        \end{align*}
        \item[(iii)] For $x\in (0, 1)$, $u, u'\in U$, 
        $$|f^u(x) - f^{u'}(x)|\leq \sqrt{\theta}|u - u'|.$$
    \end{enumerate}
\end{assumption}

\begin{theorem}\label{them:convegencePIA} 
    Assume Assumptions~\ref{a2},~\ref{a3}, and~\ref{a4} hold. Let $v \in C^1(\mathcal O)$ be the (viscosity) solution to the HJB equation~\eqref{eq:HJB} on $\mathcal{O}$, and let $(v^n)_{n\in \mathbb N} \in C^2(\mathcal O)$ be the sequence of smooth approximations generated by Algorithm~\ref{PIA}. Then there exists $q \in (0,1)$ and the initial guess $v^0$, such that for all $x \in \mathcal O$ there is a constant $C=C(x,\delta)$ ($\delta$ is a given discounting rate) satisfying  
    \begin{equation}\label{eq:converge}
           |v(x) - v^n(x)|^2 \leq C(x,\delta) q^n. 
    \end{equation}
\end{theorem}
\begin{proof}
The proof of Theorem \ref{them:convegencePIA} adapts the arguments in the proof of Theorem 4.1 in \cite{kerimkulov2020exponential}. At the $n$th iteration, the solution $v^n$ to the Bellman-type ODE \eqref{pia} is shown to be equivalent to the solution $Y^n$ of a corresponding BSDE, by the uniqueness property of infinite-horizon BSDEs. Meanwhile, the solution to the HJB equation \eqref{eq:HJB} is represented via a BSDE using Lemma~\ref{lemmad3}, which hinges on the BSDE comparison principle. The convergence then follows from the contraction property in Lemma~\ref{lemmaa5}. Given the technical complexity, the detailed proof and auxiliary lemmas are deferred to Appendix \ref{App:B}.
\end{proof}

Moreover, the following theorem establishes the monotone improvement property of the policy improvement algorithm.
\begin{corollary}\label{them:PI_infinite} 
    Let Assumptions~ \ref{a2}, \ref{a3}, and \ref{a4} hold, and fix $n \in \mathbb{N}$. Let $v^n$ and $v^{n+1}$ be the solutions of \eqref{pia} 
    at steps $n$ and $n+1$ of the Algorithm \ref{PIA}. Then, for all $x \in (0, 1)$, it holds that
    $$
        v^{n+1}(x) \leq v^n(x).
    $$
\end{corollary}
\begin{proof}
The proof proceeds by adapting the argument of Theorem 5.1 of \cite{kerimkulov2020exponential} and appealing to the comparison principle for infinite-horizon BSDEs (Lemma~\ref{cmpbsde} in Appendix \ref{App:B}). 
\end{proof}

To finish this subsection, we provide a stability property of the policy improvement algorithm under the perturbations to the solution of the Bellman-type ODE \eqref{pia} (note that the perturbations come from both the fact that Eq. \eqref{pia} is only solved approximately and our approximated boundary conditions). Hence, updating (with the first-order condition) the controls $\eta$ and $\rho$ at each iteration step in Algorithm \ref{PIA} is essentially performed only with the approximated solution, which can cumulate errors in the iteration. 

Let $\varepsilon$ be a set of parameters that determines the accuracy of the solution to the ODE \eqref{pia}. Let $u^n_\varepsilon$ be the policy at iteration $n$ obtained from an approximate 
solution to the ODE \eqref{pia}, let $v^n_{\varepsilon}$ be the solution of 
\begin{equation}\label{pia-n}
         \frac{1}{2}\sigma^2 D_{xx} v^n_\varepsilon(x)
    + b^{u^n_\varepsilon} D_x v^n_\varepsilon(x) 
    + f^{u^n_\varepsilon}(x)- \delta v^n_\varepsilon(x)  = 0,
    \quad x \in  [\underline{x},\overline{x}],
\end{equation}
with true boundary conditions. And let $\tilde{v}^n_{\varepsilon}$ be the approximate solution to
\begin{align*}
       \frac{1}{2}\sigma^2 D_{xx} v^n_\varepsilon(x)
    + b^{u^n_\varepsilon} D_x v^n_\varepsilon(x) 
    + f^{u^n_\varepsilon}(x)- \delta v^n_\varepsilon(x) & = 0,
    \quad x \in  (\underline{x},\overline{x}),\\ 
     v^n_\varepsilon(\underline{x}) &= m^{u^n_\varepsilon}_1(\underline{x}),\\
     D_xv^n_{\varepsilon}(\overline{x}) & =  m^{u^n_\varepsilon}_2(\underline{x}),
\end{align*}
 where  $m^{u^n_\varepsilon}_1(\underline{x})$ and $m^{u^n_\varepsilon}_2(\underline{x})$ denote the approximate values of $J(\underline{x};u^n_{\varepsilon})$ and the left derivative of $J(\cdot;u^n_{\varepsilon})$ at $\overline{x}$ respectively. 
Then, the policy function (see the definition of this function in Assumption \ref{a3}) for the next iteration step is given by
$$
    u^{n+1}_\varepsilon
    = u\bigl(x, (\sigma D_x \widetilde{v}^n_\varepsilon)(x)\bigr) 
    = \argmin_{u \in U} 
        \Bigl\{ (b^u  D_x \widetilde{v}^n_\varepsilon)(x) + f^u(x) \Bigr\}.
$$

\begin{corollary}\label{them:exp_convergence_infinite}
Let Assumptions~\ref{a2}, \ref{a3}, and \ref{a4} hold.  
Let $(v^n)_{n \in \mathbb{N}}$ be the approximation sequence given by Algorithm~\ref{PIA}.  
Let $(v^n_\varepsilon)_{n \in \mathbb{N}}$ be the approximation sequence given by \eqref{pia-n}.  
Let $u^\ast$ and $X^{x,u^\ast}$ be the optimal control process for \eqref{eq:HJB} 
and the associated controlled diffusion started from $x \in (0, 1)$, respectively.  
Assume that $D_x \tilde{v}^n_\varepsilon$ is uniformly bounded. Then there exist $q \in (0,1)$ and $ 0<\gamma<\delta$, such that 
for all $x \in (0,1)$, there exists a constant $C = C(x,\delta)$ with
$$
    |v^n(x) - v_\varepsilon^n(x)|^2 
    \leq C(x,\delta) q^n + 2 \sum_{k=1}^n q^k  \left\| 
        \Bigl(\sigma(D_x v_\varepsilon^k - D_x \tilde v_\varepsilon^k)\Bigr)(X^{x,u^\ast})
    \right\|_{\widehat{\mathbb{L}}_{\gamma-\delta}^2},
$$
where ${\widehat{\mathbb{L}}_{\gamma-\delta}^2}$ is the $\mathbb{L}^2_{\gamma-\delta}$-norm under probability measure $\widehat{\mathbb{P}}$,  see the notations introduced in Appendix \ref{App:B}.
\end{corollary}
\begin{proof}
The proof follows the same logic as Theorem 6.1 of \cite{kerimkulov2020exponential}, but replaces a key step with Lemma~\ref{lemmaa5} from Appendix \ref{App:B}. The remaining steps are analogous and thus omitted.
\end{proof}

\subsection{Benchmark example}\label{subsection:exp1}
The following Example \ref{example:4.1} presents a benchmark example in which the Algorithm \ref{PIA} is applied to numerically solve the cyber risk management and mitigation problem. It is important to note that the parameter $a^S_m$ shall be small (which is not unreasonable) to ensure the convergence of the PIA. Additionally, the initial guess for $\eta$ is set to zero. These choices are made to help obtain a smooth solution for the severely stiff ODE. Moreover, a combination of a Dirichlet condition at the left boundary and a Neumann condition at the right boundary of the computing interval $\{\underline{x},\overline{x}\}$ will be used to solve the stiff ODE~\eqref{pia}. Those are simulated by the conventional Monte-Carlo method.

\begin{example}
\label{example:4.1}
    We first provide a benchmark example for the cyber risk control problem \eqref{problem}.  Let $a_0 = 0.5, a_I = 5, a_m^I = 2.5, a_m^S = 0.5, a_r = 5$. Furthermore, we set the discount factor $\delta = 0.05 $. Additionally, let the external cyber attacks rate be \( \alpha = 0.5 \), internal contagion rate \( \beta = 0.5 \), (unassisted) recovery rate \( \gamma = 0.15 \), and the diffusion coefficient \( \sigma = 0.3 \). Finally, we set \( \underline{x} = 0.01 \), \( \overline{x} = 0.99 \), the grid points $N=1000$, and  tolerance $\epsilon=10^{-4}$. 
    \begin{figure}[!t]
        \centering
         \begin{subfigure}{0.45\textwidth}
         \includegraphics[width=\textwidth]{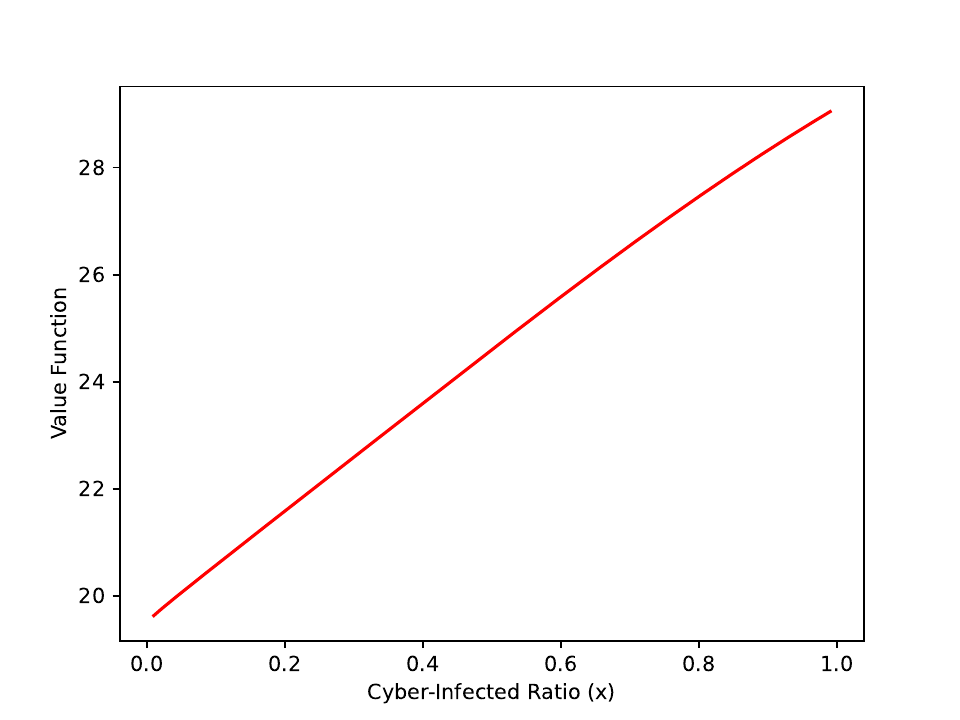}
        \caption{Value function}
        \label{fig:valueexp1}
         \end{subfigure}
        \begin{subfigure}{0.45\textwidth}
        \includegraphics[width=\textwidth]{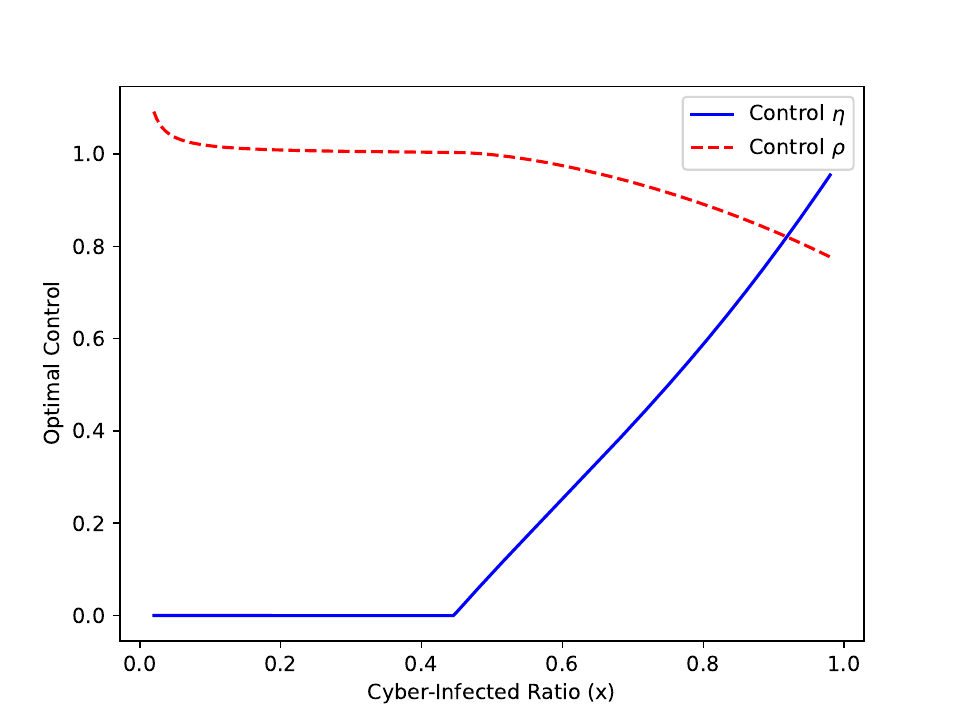}
        \caption{Optimal control strategies}
        \label{fig:controlexp1}
        \end{subfigure}
        \\
        \begin{subfigure}{0.45\textwidth}
          \includegraphics[width = \textwidth]{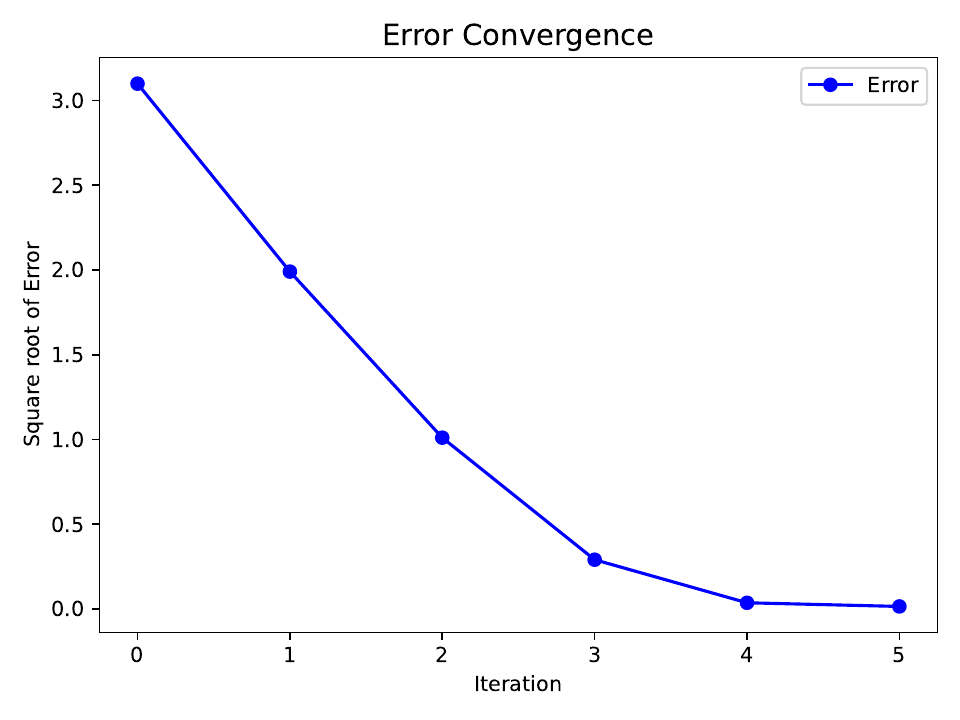}
        \caption{Convergence of the errors in the PIA.}
        \label{fig:errorexp1}  
        \end{subfigure}
        \caption{Results of benchmark example \ref{example:4.1}}
        \label{fig:example4.1}
    \end{figure}
    \end{example}
The results of the optimal strategy and value function are given in Figures \ref{fig:controlexp1} and \ref{fig:valueexp1}, respectively; we also provide the convergence 
of the errors (in terms of normalized \(L^2\)-norm of the difference between two 
successive value function), which shows the computational efficiency of the algorithm, see 
Figure \ref{fig:errorexp1}. The algorithm converges with an error less than $10^{-4}$ within 
eight steps starting from the initial guess $\eta^0 = 0$ and $\rho^0 = 0$. 
To better understand the behavior of the optimal control $(\eta^*, \rho^*)$, we observe 
that $\eta^*$ remains at zero (strong proactive control) when the ratio of cyber-infected 
nodes is considerably small, while $\rho^*$ decays fast from its maximum. This pattern 
indicates that when the current system has a small ratio of cyber-infected nodes, applying 
the risk mitigation control to enhance the recovery rate is less effective than 
implementing proactive management control to prevent the internal contagions and external cyber attacks. Consequently, the 
optimal strategy prioritizes risk management over risk mitigation at the outset. 

Furthermore, Figure \ref{fig:controlexp1} shows that both the optimal controls $\eta^*$ (cyber risk management) and $\rho^*$ (cyber risk mitigation) decrease as the ratio of cyber-infected nodes $x$ increases. In particular, we have the following observations:
    \begin{itemize}
        \item When the cyber-infected ratio is low (i.e., $x\approx 0$), both controls are set high, which reflects a strong incentive for the decision maker to invest in cyber risk prevention and mitigation in an early stage.
        \item As the cyber-infected ratio rises ($x\uparrow 1$), both controls decline. This suggests that once the system is already heavily compromised, additional investments in prevention or mitigation have diminishing impact on reducing risk, especially for the proactive risk management control $\eta$. Hence, the decision maker may want to prioritize resource allocation toward risk mitigation strategies.
    \end{itemize}
    The above observations align with the intuition of cyber risk control under limited resources: it is most effective to intervene early, when cyber-infected ratios are still small; while intervention becomes less valuable (and less cost-effective) when the system is already in a state of widespread infection. 

\begin{remark}
\label{re:MCintervals}
For the robustness, in the following Table \ref{tab:placeholder}, we show the sensitivity of the value function (evaluated at few fixed interior points $x = \{0.1, 0.3, 0.5, 0.7, 0.9\}$) to the choices of the closed interval $\mathcal{O}= [\epsilon, 1-\epsilon]$ for $\epsilon = 0.05, 0.01, 0.005, 0.0025, 0.001$.
\begin{table}[ht!]
    \centering
    \begin{tabular}{lccccc}
\toprule[1pt]
 $[\varepsilon,1-\varepsilon]$ 
& $V(0.1)$ & $V(0.3)$ & $V(0.5)$ & $V(0.7)$ & $V(0.9)$ \\
\midrule[0.5pt]
$\underline{[0.05,\,0.95]}$   
& 20.5470 & 22.5586 & 24.5761 & 26.5137 & 28.3010 \\
$[0.01,\,0.99]$   
& 20.5753 & 22.5958 & 24.6048 & 26.5423 & 28.3227 \\
$[0.005,\,0.995]$ 
& 20.5856 & 22.6062 & 24.6056 & 26.5533 & 28.3340 \\
$[0.0025,\,0.9975]$ 
& 20.5962 & 22.6164 & 24.6158 & 26.5634 & 28.3441 \\
$[0.001,\,0.999]$ 
& 20.6060 & 22.6218 & 24.6269 & 26.5800 & 28.3566 \\
\bottomrule[1pt]
\end{tabular}
    \caption{Sensitivity of the value function on boundary choices}
    \label{tab:placeholder}
\end{table}

\end{remark}

\subsection{Suboptimal control analysis}

\begin{example}
\label{example4.2}
In this example, we provide numerical analysis  when we fix either the risk management control at 
zero ($\eta(\cdot) \equiv 1$) or the risk mitigation control at zero ($\rho(\cdot)\equiv0$). The resulting optimal 
strategies (of the single control) and value functions are given in Figure \ref{fig:example4.2}.
\end{example}
\begin{figure}[ht]
        \centering
        \begin{subfigure}{0.48\textwidth}
            \includegraphics[width = \textwidth]{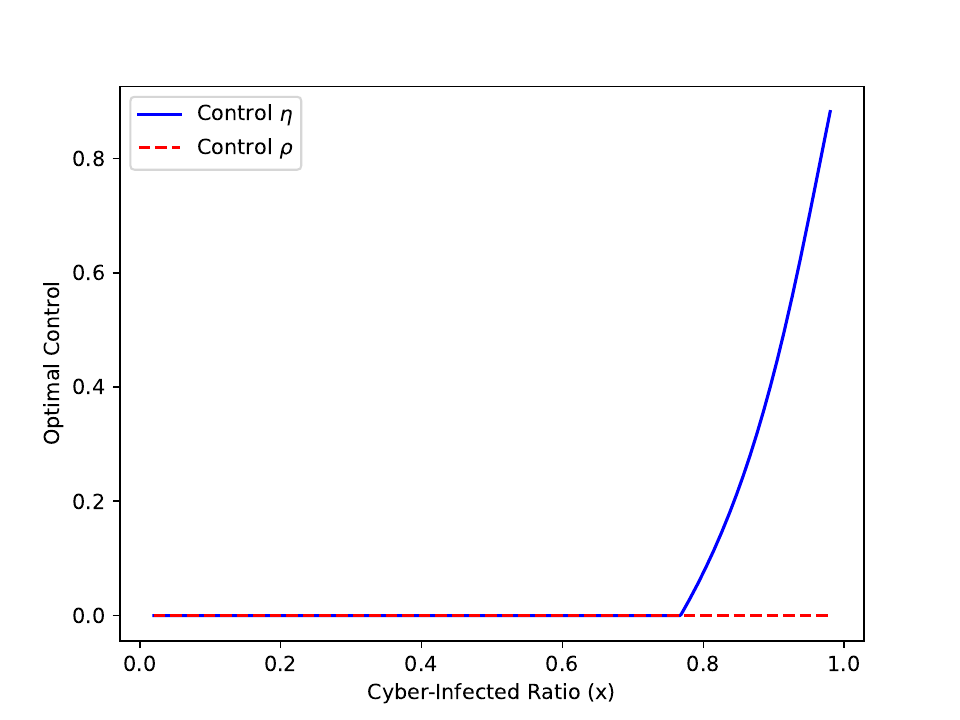}
        \caption{Optimal $\eta^*$ with $\rho\equiv0$.}
        \label{fig:etaonly}
        \end{subfigure}
        \begin{subfigure}{0.48\textwidth}
           \includegraphics[width = \textwidth]{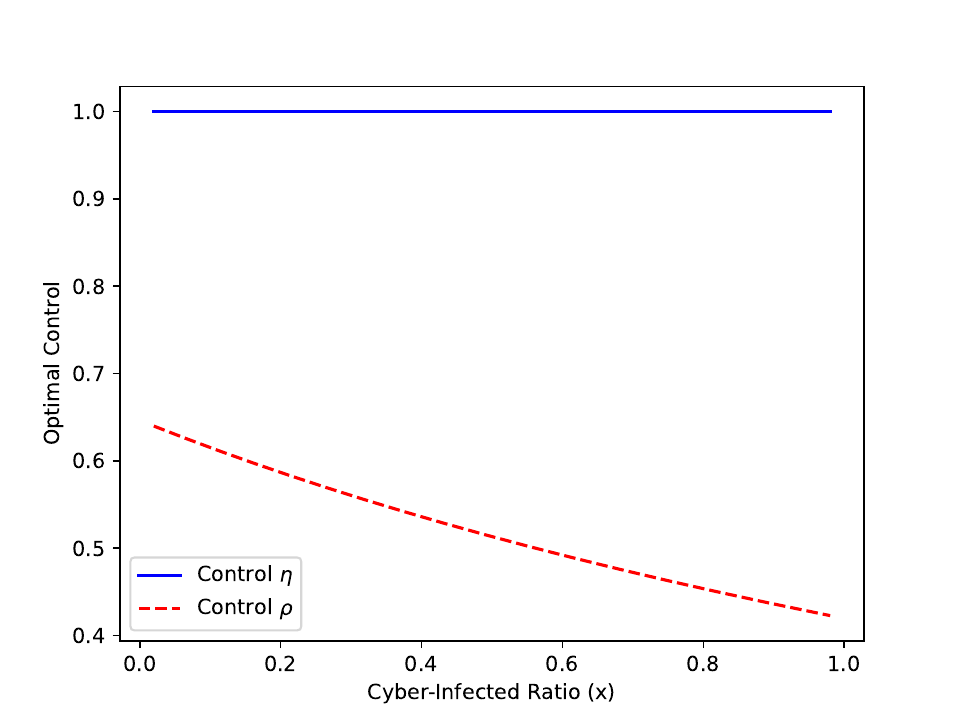}
        \caption{Optimal $\rho^*$ with $\eta\equiv 1$.} 
        \label{fig:rhoonly}
        \end{subfigure}       \\
        \begin{subfigure}{0.48\textwidth}
            \includegraphics[width = \textwidth]{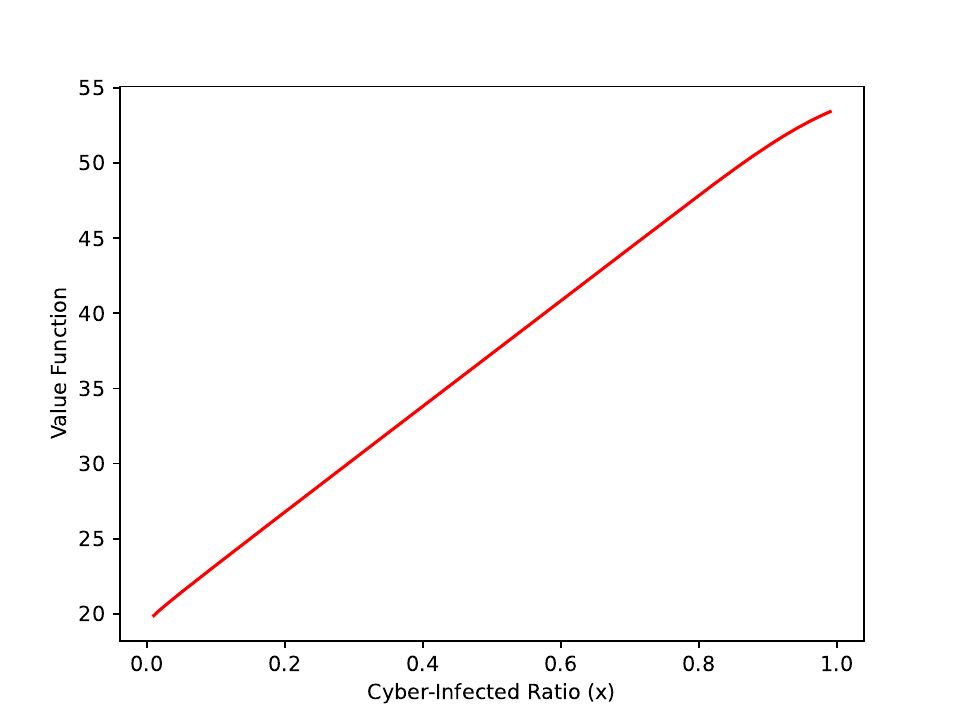}
        \caption{Value function when only risk management control is present.}
        \label{fig:valueeta}
        \end{subfigure}       
        \begin{subfigure}{0.48\textwidth}
            \includegraphics[width = \textwidth]{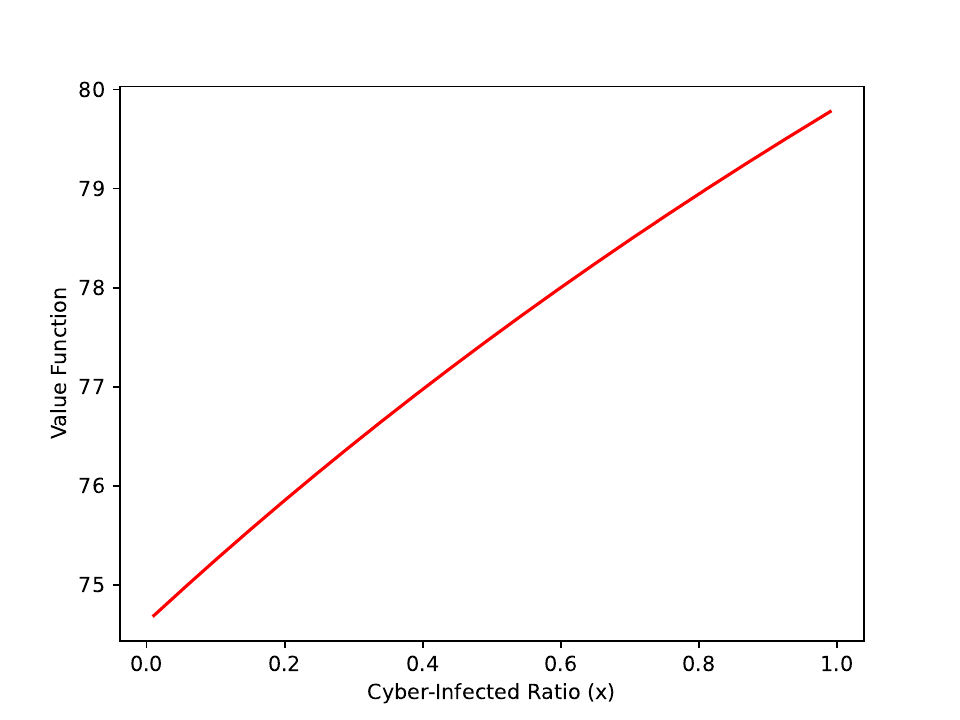}
        \caption{Value function when only risk mitigation control is present.}
        \label{fig:valuerho}
        \end{subfigure}       
        \caption{Optimal strategy and value function for suboptimal controls}
        \label{fig:example4.2}
        \end{figure}
When removing the reactive cyber risk mitigation control  (i.e., $\rho \equiv 0$), the optimal 
proactive risk management control $\eta^*$ becomes noticeably stronger (see Figure 
\ref{fig:etaonly}) compared to the benchmark example, which compensates for the absence of 
reactive mitigation controls. The value function remains nearly unchanged at small $x$, 
but rises substantially (reaching about $55$ 
versus the benchmark level of $28$) when the cyber-infected ratio is close to one, see Figure \ref{fig:valueeta}. This suggests that as cyber-infection level increases, the marginal effectiveness of risk management control drops off, and proactive risk management alone cannot effectively substitute for reactive mitigation control in a highly infected system. 

On the other hand, when we remove the risk management control (i.e., $\eta \equiv 1$), the optimal reactive control $\rho^*$ decreases relative to the benchmark (Figure \ref{fig:rhoonly}), reflecting the limited effectiveness of mitigation when proactive control from a risk management perspective is unavailable. In this case, the value function rises overall (Figure \ref{fig:valuerho}), shifting from a range between 20 to 30 under the benchmark example to a range between 75 to 80. This shows that relying solely on reactive mitigation strategies results in higher expected costs, confirming that mitigation control cannot fully substitute for proactive management.



\subsection{Sensitivity analysis}   
In this subsection, we numerically analyze the distinct roles of proactive control $\eta$ (risk management) and reactive control $\rho$ (risk mitigation) in shaping the optimal value function.
In particular, we deviate by a series of small changes (uniformly in $x$) for $\eta$ and $\rho$, respectively, from the optimal strategy $(\eta^*,\rho^*)$ in the benchmark example. The results are plotted in Figure \ref{fig:eta+}--\ref{fig:rho-}. 
    \begin{figure}[ht]
        \centering
        \begin{subfigure}{0.48\textwidth}
            \includegraphics[width = \textwidth]{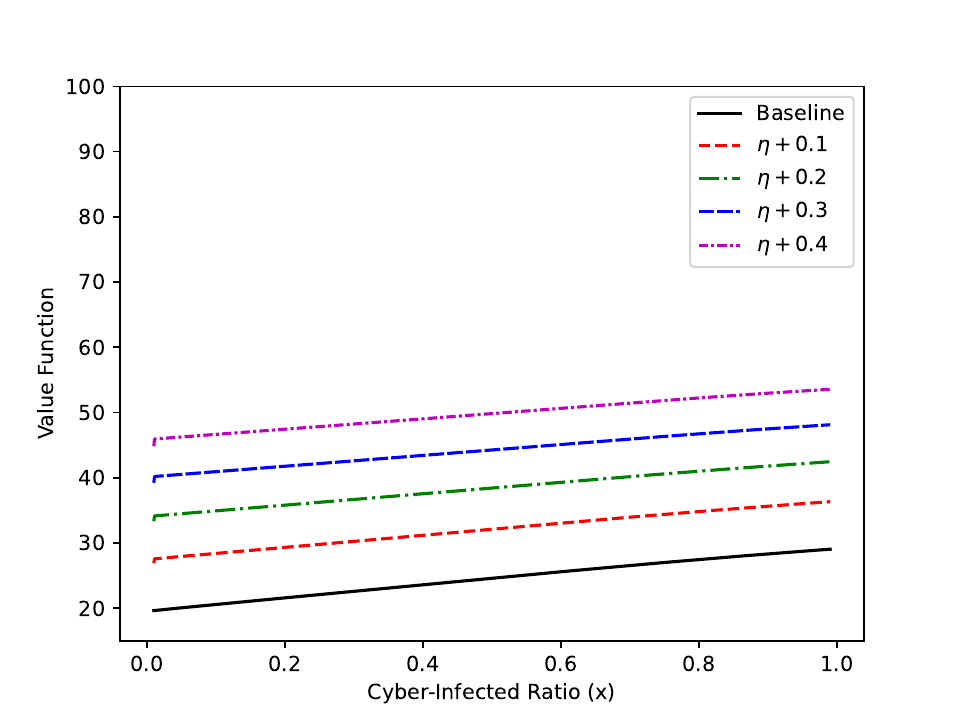}
        \caption{Value functions for positive perturbations\\ on $\eta^*$.}
        \label{fig:eta+}
        \end{subfigure}
        \begin{subfigure}{0.48\textwidth}
            \includegraphics[width = \textwidth]{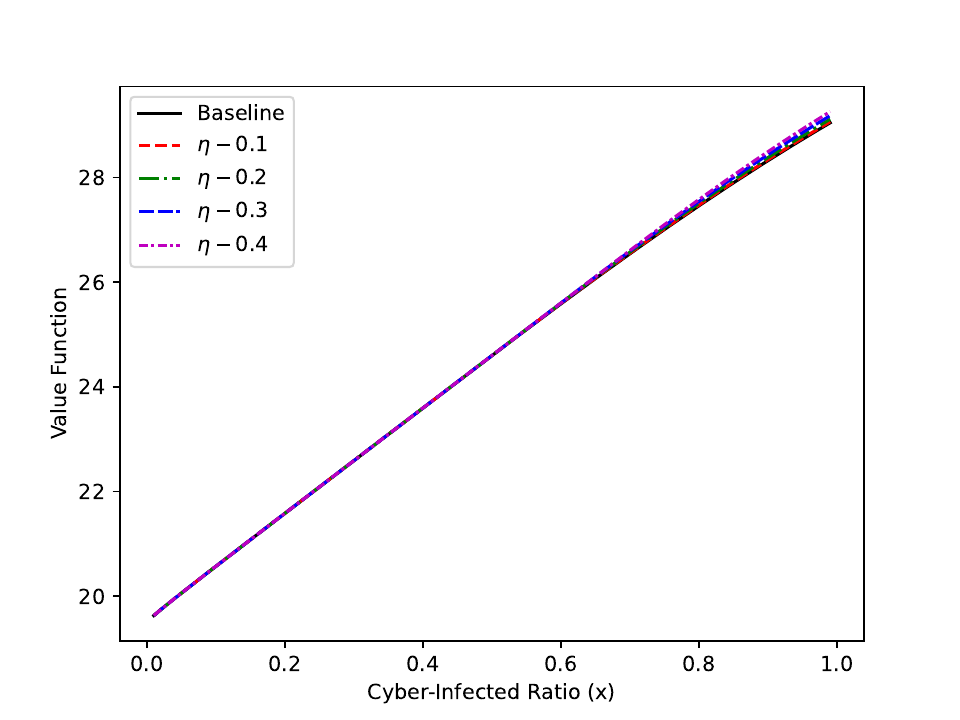}
        \caption{Value functions for negative perturbations\\ on $\eta^*$.}
        \label{fig:eta-}
        \end{subfigure}
       \begin{subfigure}{0.48\textwidth}
             \includegraphics[width = \textwidth]{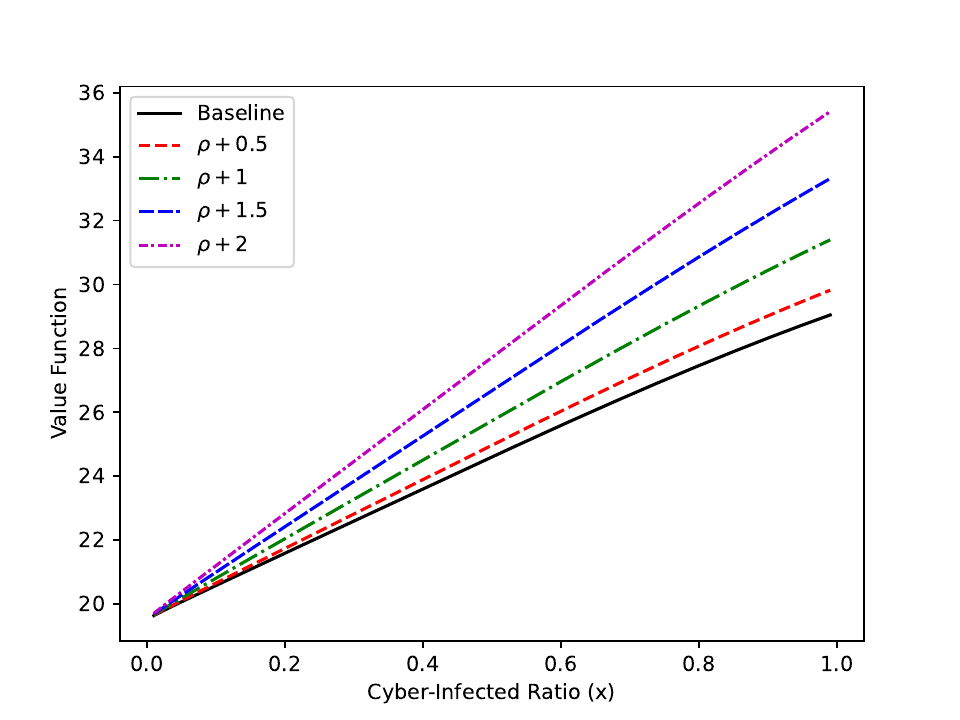}
        \caption{Value functions for positive perturbations\\ on $\rho^*$.}
        \label{fig:rho+}
       \end{subfigure}
       \begin{subfigure}{0.5\textwidth}
            \includegraphics[width = \textwidth]{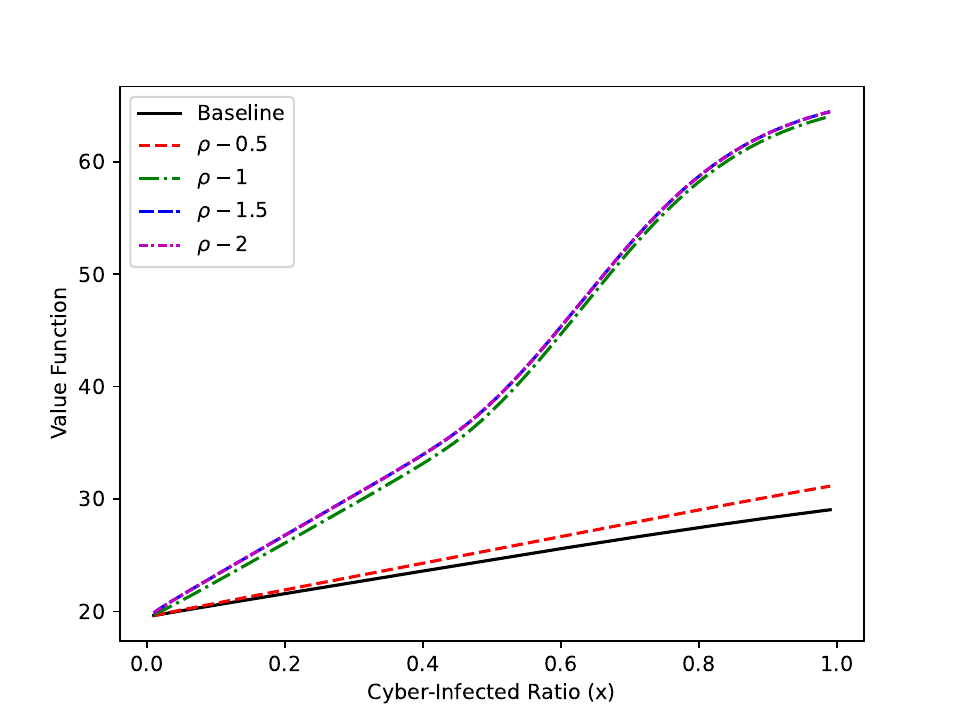}
        \caption{Value functions for negative perturbations\\ on $\rho^*$.}
        \label{fig:rho-}
       \end{subfigure}
       \caption{Sensitivity analysis on small deviations from the optimal control.}
    \end{figure}
   We conclude this subsection with the following observations:
   \begin{itemize}
       \item \textit{Proactive risk management control ($\eta$):} According to Figure \ref{fig:eta+}, increasing $\eta$ produces a substantial and nearly uniform increase in value function (higher expected discounted costs) across all initial states of cyber-infected level, confirming the consistent effectiveness of proactive measures when managing cyber risks in the system. Conversely, decreasing $\eta$ has limited effects on the magnitude of the value function (see Figure \ref{fig:eta-}). Note that, since $\eta^*(x)=0$ for small values of $x$, the decrease of $\eta$ cannot be applied in this case; hence, we do not observe any changes in the corresponding value function.  But, when the cyber-infected ratio is high, the inability to sustain strong proactive control leads to a noticeable increase in the value function (but not comparable to the corresponding scenario when increasing $\eta$). 
       
       \item  \textit{Reactive risk mitigation control ($\rho$):} Both Figures \ref{fig:rho+} and \ref{fig:rho-} show that perturbations on the value of $\rho^*$ exert their strongest influence when the system is in a status with a high cyber-infected ratio. Such an observation indicates that the reactive mitigation control is most valuable once the system contains a large number of infected nodes. In addition, unlike the risk management control, we can observe an obvious non-uniform change in the value function with respect to $x$. In particular, when we add positive perturbations to $\rho^*$ (i.e., adding redundant mitigation controls), the increase in the value function (i.e., expected discounted costs) is moderate and consistent. However, adding negative perturbations to $\rho^*$, which refers to insufficient mitigation controls, can distort the shape of the value function and sharply increase expected discounted costs. One may also conjecture (see Figure \ref{fig:rho-}) that \textit{there exists a critical ``threshold level" for cyber risk mitigation control, such that insufficient reactive mitigation control below the ``level" can cause tremendous losses.} We leave this interesting observation for future research.
       \item  \textit{Overall conclusion:} The sensitivity analysis together with the suboptimal control analysis in Example \ref{example4.2} highlight an ``asymmetry" between the two types of control strategies. Proactive risk management provides consistent and broad benefits, and can partially substitute for the absence of mitigation controls. By contrast, reactive mitigation is valuable only when the system is heavily compromised with a high ratio of cyber-infected nodes, and cannot substitute for missing proactive risk management control. In practice, this implies that effective cyber risk control strategies require front-loaded investment in proactive defense, with reactive mitigation serving as a complementary safeguard against severe system breakdown rather than a stand-alone strategy. 
       \end{itemize}


\begin{remark}\label{rmk1} 
One may notice that each ``value function" obtained (by solving the corresponding ODEs under perturbed control) in the above sensitivity analysis (Figure~\ref{fig:eta+}--\ref{fig:rho-}) 
are not necessarily the objective function $J$ given in \eqref{eq:performfunc} under perturbed control. 

In this remark, we assume the solution to the Bellman ODE is twice continuously differentiable. While the numerical solution is not necessarily twice continuously differentiable, it approximates a $C^2((0, 1))$ solution under standard regularity assumptions and sufficient discretization accuracy. If we fixed the control by each perturbation above, such as $\eta\pm \Delta \eta$ and $\rho \pm \Delta \rho$, then the control space $U$ is reduced to a singleton. As the consequence, one could apply the It\^o's formula for the discounted process $e^{-\delta t} u(X_t^{x})$ between 0 and ${T\wedge \tau_n}$ with a sequence of stopping times $\tau_n:= \inf\{t\geq 0: \int_0^{T\wedge \tau_n} e^{-\delta t}D_xu(X_t^x)\sigma(X_t^x)\dif W_t \geq n\}$, 
{\small
\begin{equation*}
    \begin{split}
        & \mathbb E \left[e^{-\delta {T\wedge \tau_n}} u\left(X_{T\wedge \tau_n}^{x}\right)\right]\\
        &= u(x) + \mathbb E\left[\int_0^{T\wedge \tau_n} e^{-\delta t}[b(X_t^x) D_x u(X_t^x) + \frac{1}{2}\sigma^2(X_t^x) D_{xx}u(X_t^x) - \delta u(X_t^x) ]\dif t \right]\\
        &\qquad  + \mathbb E\left[\int_0^{T\wedge \tau_n} e^{-\delta t}D_xu(X_t^x)\sigma(X_t^x)\dif W_t\right]\\
        &= u(x) + \mathbb E\left[\int_0^{T\wedge \tau_n} e^{-\delta t}[b(X_t^x) D_x u(X_t^x) + \frac{1}{2}\sigma^2(X_t^x) D_{xx}u(X_t^x) - \delta u(X_t^x) ]\dif t \right], 
    \end{split}
\end{equation*}
}
where denote $X^x_t:=X^{0, x, \eta, \rho}_t$, $b(X^x_t):=b^{\eta, \rho}(X^x_t)$ and note the stochastic integral is a local martingale. Sending $n$ to infinity, then 
$$\mathbb E \left[e^{-\delta {T}} u\left(X_{T}^{x}\right)\right] = u(x) + \mathbb E\left[\int_0^{T} e^{-\delta t}[b(X_t^x) D_x u(X_t^x) + \frac{1}{2}\sigma^2(X_t^x) D_{xx}u(X_t^x) - \delta u(X_t^x) ]\dif t \right]$$ holds by dominated convergence theorem. By using the fact that $u$ is a solution to the ODE stated above and then sending $T$ to infinity, we observe that 
$$u(x) =  \mathbb E\left[\int_0^{\infty} e^{-\delta t} f^{\eta, \rho}(x)\dif t \right]$$ coincides with the definition of objective function $J(x, \eta, \rho)$. 
\end{remark}


\subsection{Comparative statics}  
In this final subsection, we perform a comparative statics analysis across all model parameters, including $\alpha$, $\beta$, $\sigma$, and the marginal cost parameters ($a_I,a_m^I,a_m^S, a_r$) in the cost function given in \eqref{eq:costf}. Note that it is not necessary to investigate the parameter $\gamma$, which contributes additively to the mitigation control $\rho$ in numerical results.

 \textbf{Comparative statics for} \bm{$\alpha$} \textbf{and} \bm{$\beta$}. We first compare different values of the external cyber attack rate $\alpha$ and the internal contagious rate $\beta$, keeping other parameters unchanged in Example \ref{example:4.1}. The resulting optimal strategies are given in Figures \ref{fig:compare_eta_alpha}, \ref{fig:compare_rho_alpha} (for $\alpha$) and \ref{fig:compare_eta_beta}, \ref{fig:compare_rho_beta} (for $\beta$), respectively.

    \begin{figure}[ht]
        \centering
        \begin{subfigure}{0.48\textwidth}
            \centering
            \includegraphics[width=\textwidth]{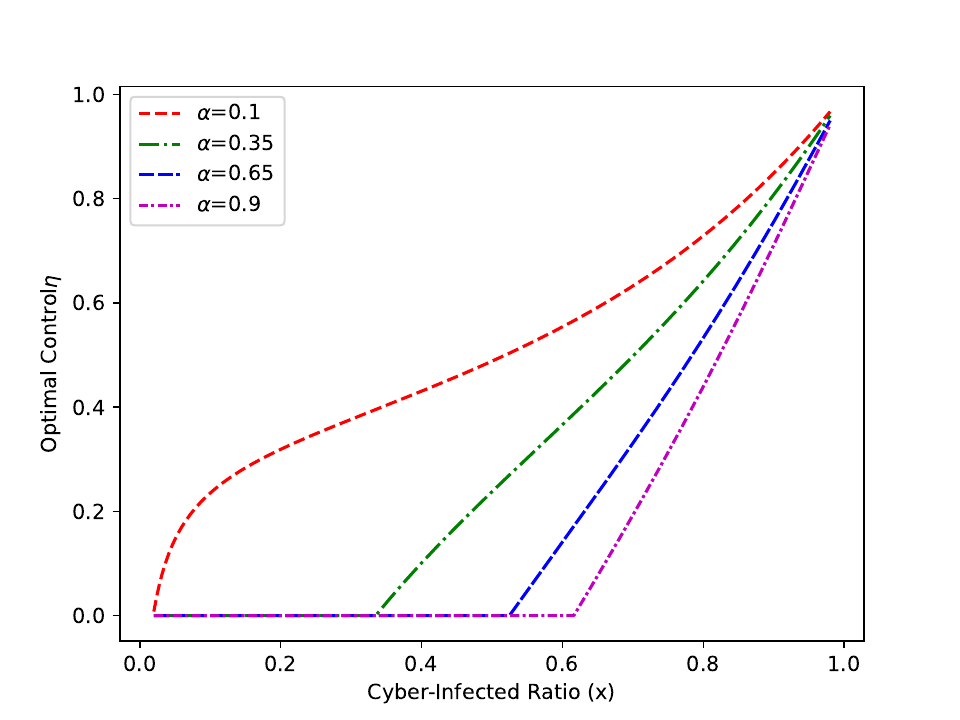}
            \caption{Optimal $\eta^*$ for different $\alpha$}
            \label{fig:compare_eta_alpha}
        \end{subfigure}
        \hfill
        \begin{subfigure}{0.48\textwidth}
            \centering
            \includegraphics[width=\textwidth]{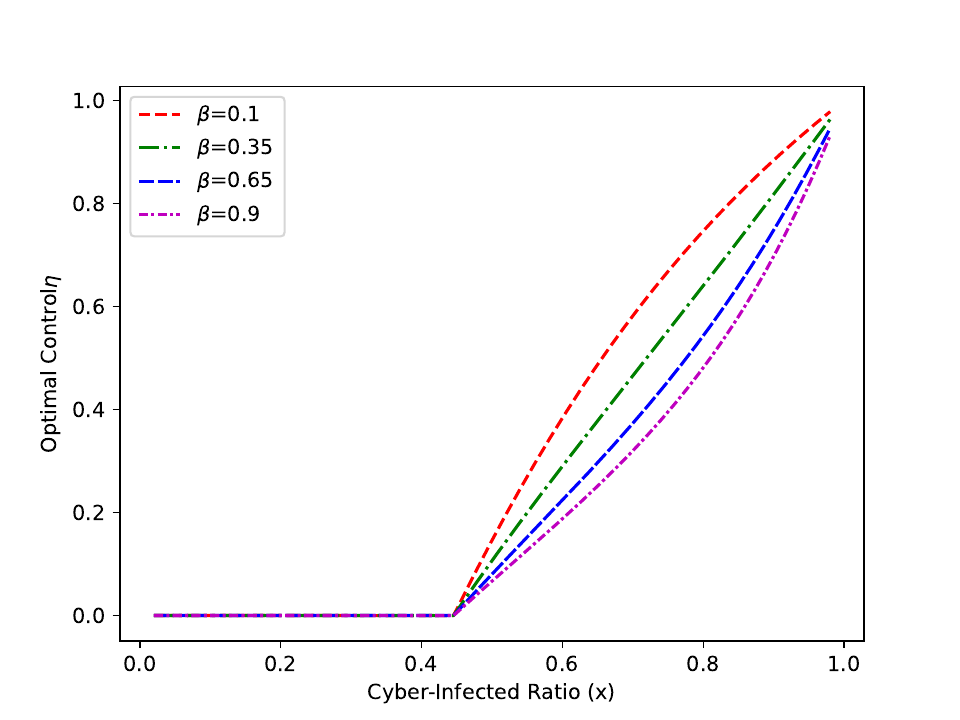}
            \caption{Optimal $\eta^*$ for different $\beta$}
            \label{fig:compare_eta_beta}
        \end{subfigure}
        \begin{subfigure}{0.48\textwidth}
            \centering
            \includegraphics[width=\textwidth]{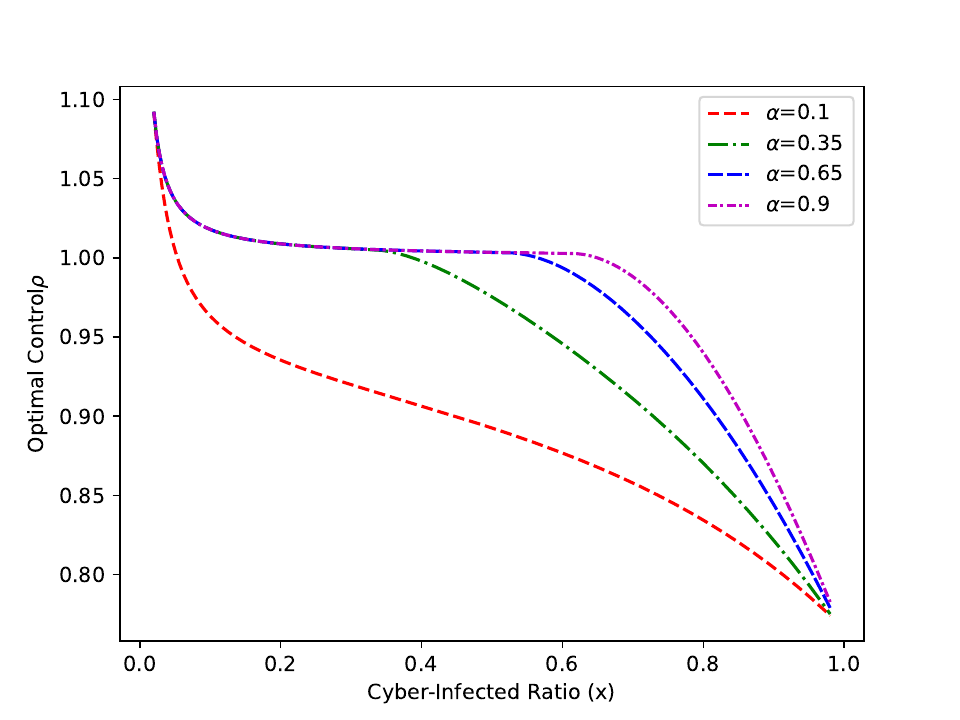}
            \caption{Optimal $\rho^*$ for different $\alpha$}
            \label{fig:compare_rho_alpha}
        \end{subfigure}
        \hfill
        \begin{subfigure}{0.48\textwidth}
            \centering
            \includegraphics[width=\textwidth]{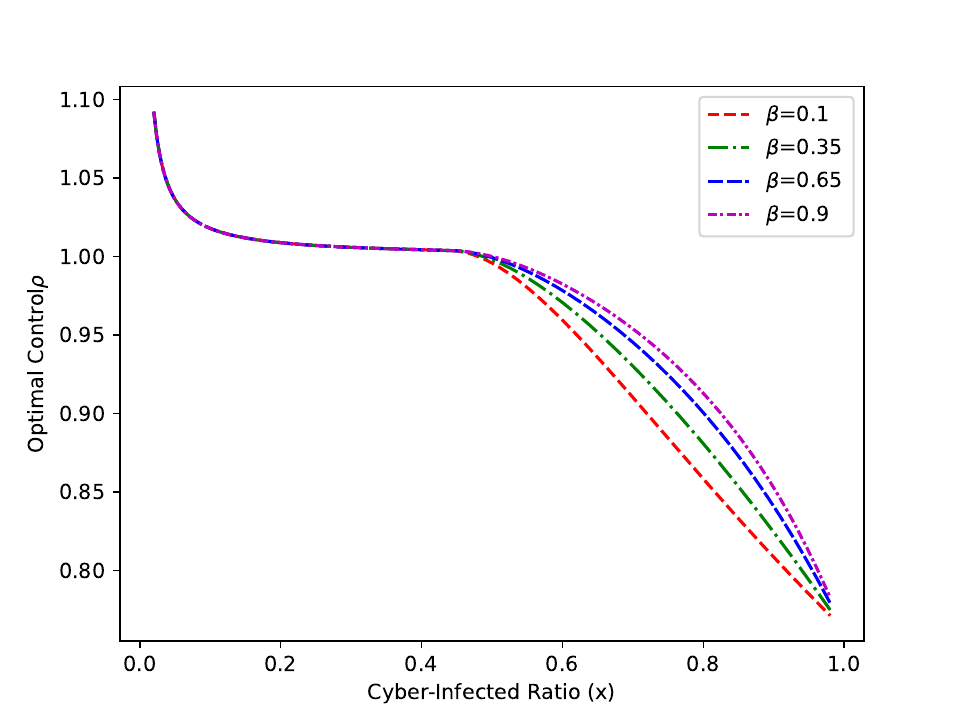}
            \caption{Optimal $\rho^*$ for different $\beta$}
            \label{fig:compare_rho_beta}
        \end{subfigure}
        \caption{Comparison analysis for $\alpha$(left) and $\beta$(right).}
        \label{fig:comparison-ab}
    \end{figure}
Figures \ref{fig:comparison-ab} show that an increase in both $\alpha$ (the rate of external cyber attacks) and $\beta$ (the rate of internal cyber risk propagation) requires stronger risk management and mitigation controls, with the optimal proactive management $\eta^*$ moving closer to $0$ and the optimal reactive mitigation $\rho^*$ rising. This reflects that it is optimal to simultaneously reinforce both preventive measures and reactive responses when cyber risks escalate. The impact of $\alpha$ is more pronounced than that of $\beta$, leading to sharper adjustments in 
both controls. In other words, $\eta^*$ exhibits a larger flatten zero-valued region and $\rho^*$ rises for almost all states with $\alpha$ increased. However, when $\beta$ increases, both optimal strategies are strengthened for highly infected states, although the curvature of $\eta^*$ becomes convex whereas that of $\rho^*$ becomes concave. From a cyber risk perspective, this distinction is natural: a surge in external attacks compels the decision maker to rapidly escalate defensive risk management measures and amplify mitigation controls. In contrast, internal risk propagation dynamics among susceptible and infected nodes induce a more gradual--though less pronounced--reinforcement of the two controls. Such a result thus demonstrates the necessity of preferential allocation of resources to the management of external cyber attacks.

\textbf{{Comparative statics for} \bm$\sigma$.} We compare different values of the volatility parameter $\sigma$ in the stochastic system, assuming other parameters remain the same as in Example \ref{example:4.1}.
Figure \ref{fig:compare_sigma} illustrates the impact of increasing 
the volatility parameter $\sigma$ on the optimal control $(\eta^*,\rho^*)$. As $\sigma$ rises from $0.1$ to $0.5$, $1$, and $2$, both the proactive management $\eta$ and the reactive 
mitigation $\rho$ weaken slightly. Intuitively, higher volatility increases uncertainty in the evolution of the cyber-infected ratio in the system, making aggressive interventions less effective. 
The small magnitude (compared with the cases when changing $\alpha$ and $\beta$) of the change indicates that the control policy is primarily driven by the system's drift dynamics, with stochastic fluctuations playing a secondary role. 
\begin{figure}[H]
    \centering
\begin{subfigure}{0.48\textwidth}
            \centering
            \includegraphics[width=\textwidth]{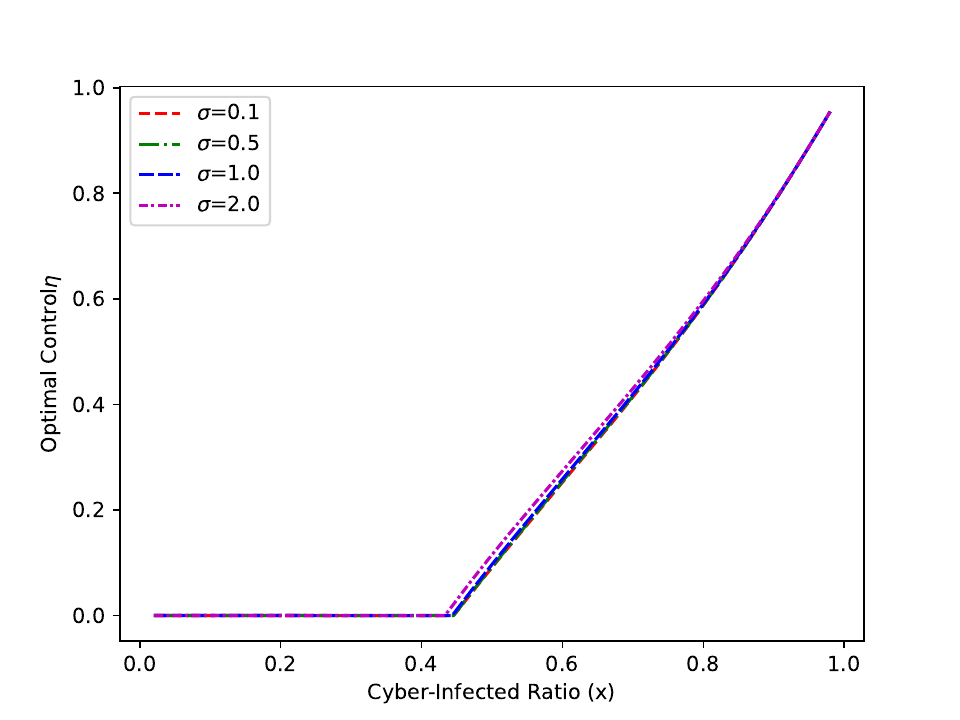}
            \caption{Optimal $\eta^*$ for different $\sigma$}
            \label{fig:compare_eta_sigma}
        \end{subfigure}
        \hfill
        \begin{subfigure}{0.48\textwidth}
            \centering
            \includegraphics[width=\textwidth]{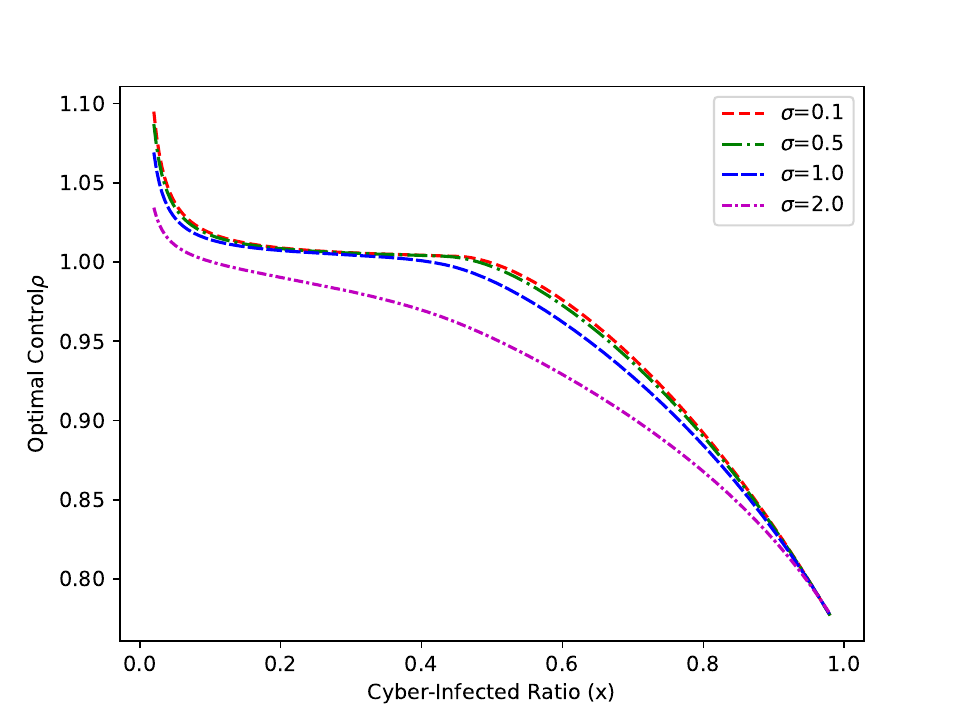}
            \caption{Optimal $\rho^*$ for different $\sigma$}
            \label{fig:compare_rho_sigma}
        \end{subfigure}
        \caption{Comparison analysis for $\sigma$.}
    \label{fig:compare_sigma}
\end{figure}
    
\textbf{Comparative statics for $a_m^I$ and $a_m^S$.} We further compare different values of the marginal costs associated with cyber risk management for cyber-infected nodes and susceptible nodes, respectively. We plot the resulting optimal strategies in Figure \ref{fig:comparison-am}.

 \begin{figure}[ht]
        \centering
        \begin{subfigure}{0.48\textwidth}
            \centering
            \includegraphics[width=\textwidth]{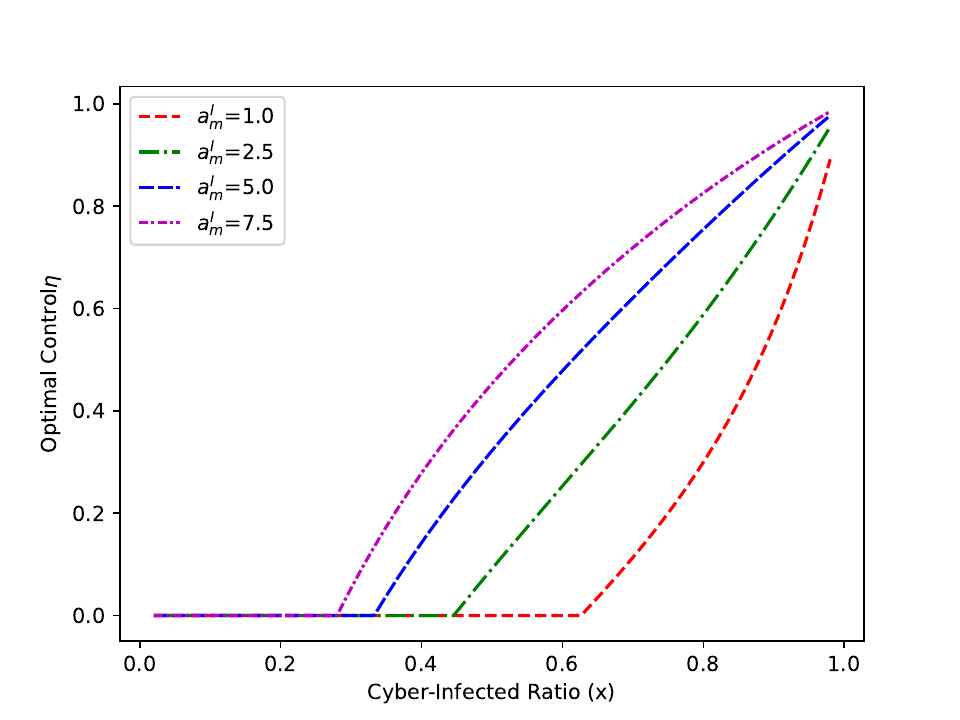}
            \caption{Optimal $\eta^*$ for different $a^I_m$.}
            \label{fig:compare_eta_aIm}
        \end{subfigure}
        \hfill
        \begin{subfigure}{0.48\textwidth}
            \centering
            \includegraphics[width=\textwidth]{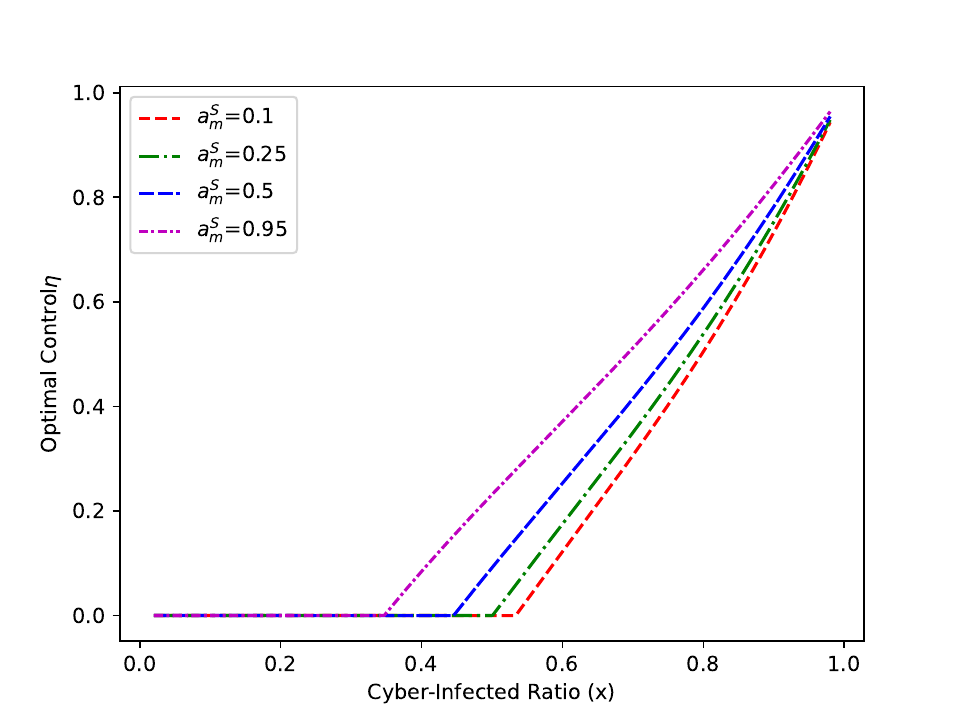}
            \caption{Optimal $\eta^*$ for different $a^S_m$.}
            \label{fig:compare_eta_aSm}
        \end{subfigure}
         \begin{subfigure}{0.48\textwidth}
            \centering
            \includegraphics[width=\textwidth]{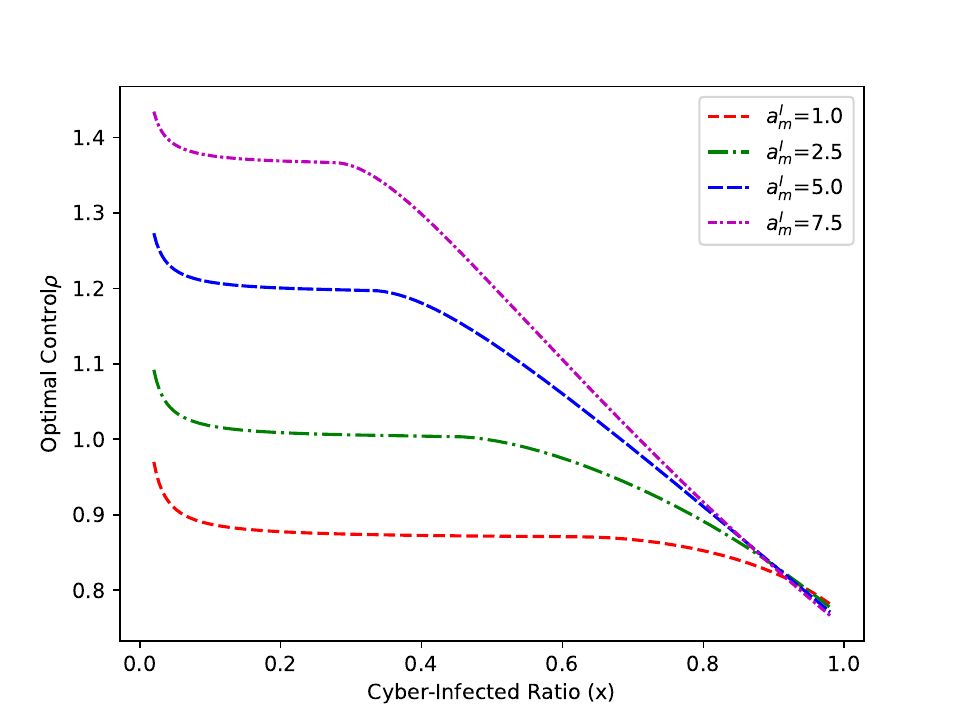}
            \caption{Optimal $\rho^*$ for different $a^I_m$.}
            \label{fig:compare_rho_aIm}
        \end{subfigure}
        \hfill
        \begin{subfigure}{0.48\textwidth}
            \centering
            \includegraphics[width=\textwidth]{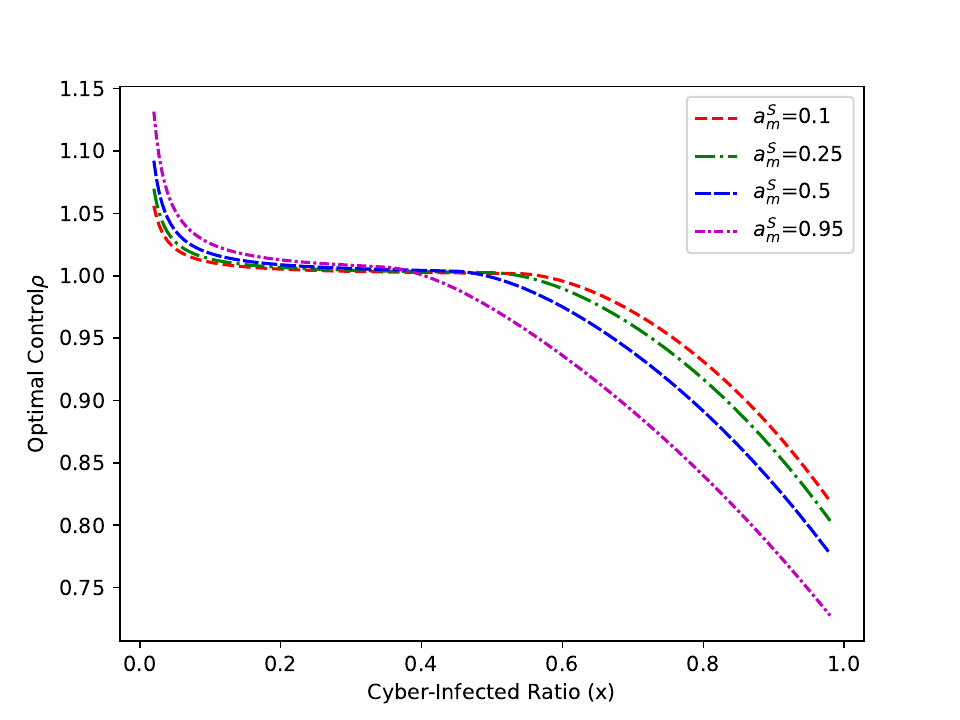}
            \caption{Optimal $\rho^*$ for different $a^S_m$.}
            \label{fig:compare_rho_aSm}
        \end{subfigure}
        \caption{Comparison analysis for $a_m^I$(left) and $a_m^S$(right).}
        \label{fig:comparison-am}
    \end{figure}
It is reasonable to observe that when the marginal costs of proactive management control associated with either cyber-infected nodes $a_m^I$ or susceptible nodes $a_m^S$ increase, the optimal proactive 
management control $\eta^*$ weakens. On the other hand, the optimal reactive mitigation control $\rho^*$ becomes stronger when the marginal costs of management control incurred by cyber-infected nodes ($a_m^I$) increase, especially in low-infected states. This is because increasing $a^I_m$ makes mitigation failure substantially more costly in infected states, thereby amplifying the expected future cost of infection. Since the role of $\rho$ is to suppress infection growth and mitigate the transition from low to high infection levels (see Example~\ref{example4.2}), it becomes optimal to strengthen $\rho$ at early stages in order to prevent the system from entering highly infected states.
When $a^S_m$ increases, strong mitigation 
($\eta \approx 0$) becomes more costly, which leads the controller to relax mitigation efforts. At high infection levels, where mitigation is already less effective due to saturation, the marginal benefit of increasing the recovery control $\rho$ cannot offset its quadratic cost. As a result, the optimal policy reduces $\rho^*$ for large infected states. In contrast, for small infection levels, where the cost of $\rho$ is negligible, $\rho$ may be slightly strengthened as a partial substitute for mitigation.
In addition, the optimal reactive mitigation control is more sensitive to the change of $a_m^I$ when the cyber-infected ratio in the system is low, but it is more sensitive to the change of $a_m^S$ when the system is heavily compromised.
Furthermore,  one can observe that the optimal risk management control $\eta^*(x)$, as a function of the cyber-infected ratio $x$, exhibits a convex shape when $a^I_m$ has not reached at the level of infection cost $a_I$ (see e.g., Figure \ref{fig:compare_eta_aIm}), and the convexity diminishes and eventually turns out to be concavity when $a_m^I$ decreases. Such a phenomenon may be rooted in the interaction (or trade-off) between the marginal costs $a_I$ and $a_m^I$. When $a_I$ is (sufficiently) larger than $a_m^I$, that is, the costs associated with cyber-infected nodes under management control are negligible compared to the baseline management costs of all cyber-infected nodes, then the optimal management control strategy is ``aggressive'' to the cyber-infected ratio, hence results in bending towards zero. But, when $a_m^I$ is sufficiently larger than $a_I$, the optimal management control, as a function of cyber-infected ratio, becomes a concave function. It means that, to minimize the total expected discounted costs, the decision maker might want to reduce the strength of proactive management control aggressively when the cyber-infected ratio is at a moderate level, and the tendency declines when the cyber-infected ratio approaches one (hence, a concave form). One can observe a similar result when changing the value of $a_I$ and keeping $a_m^I$ fixed, see Figures \ref{fig:compare_eta_aI} and \ref{fig:compare_rho_aI}.

\textbf{Comparative statics for $a_I$ and $a_r$.} We change the value of marginal costs ($a_I$) incurred by cyber-infected nodes in the system, or the marginal costs ($a_r$) associated with reactive risk mitigation control, keeping other parameters unchanged in Example \ref{example:4.1}.
\begin{figure}[ht]
        \centering
        \begin{subfigure}{0.48\textwidth}
            \centering
            \includegraphics[width=\textwidth]{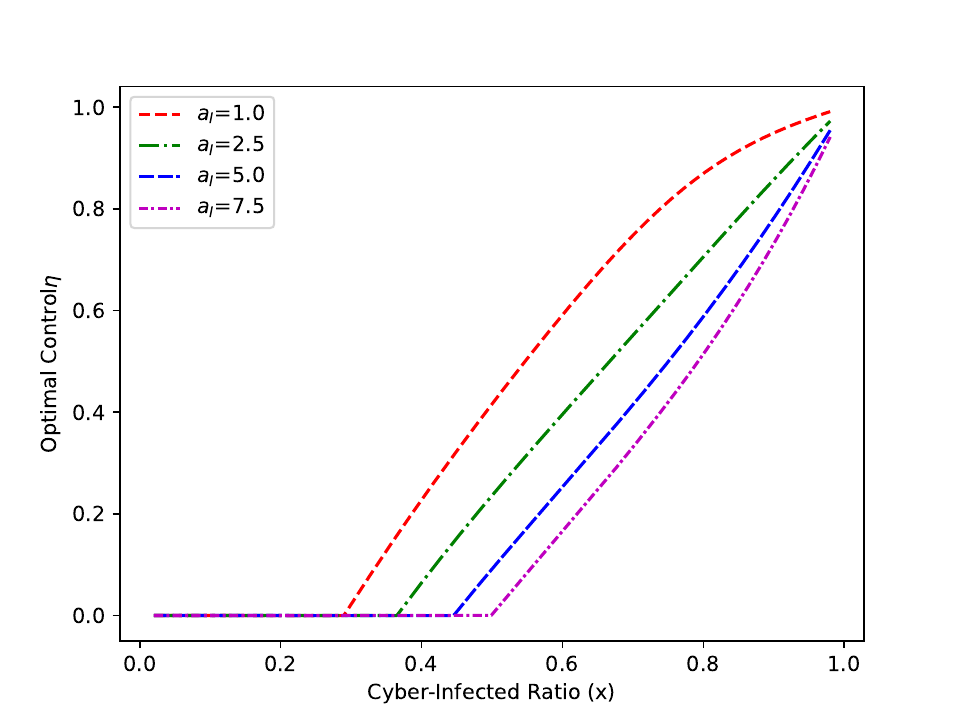}
            \caption{Optimal $\eta^*$ for different $a_I$}
            \label{fig:compare_eta_aI}
        \end{subfigure}
        \hfill
        \begin{subfigure}{0.48\textwidth}
            \centering
            \includegraphics[width=\textwidth]{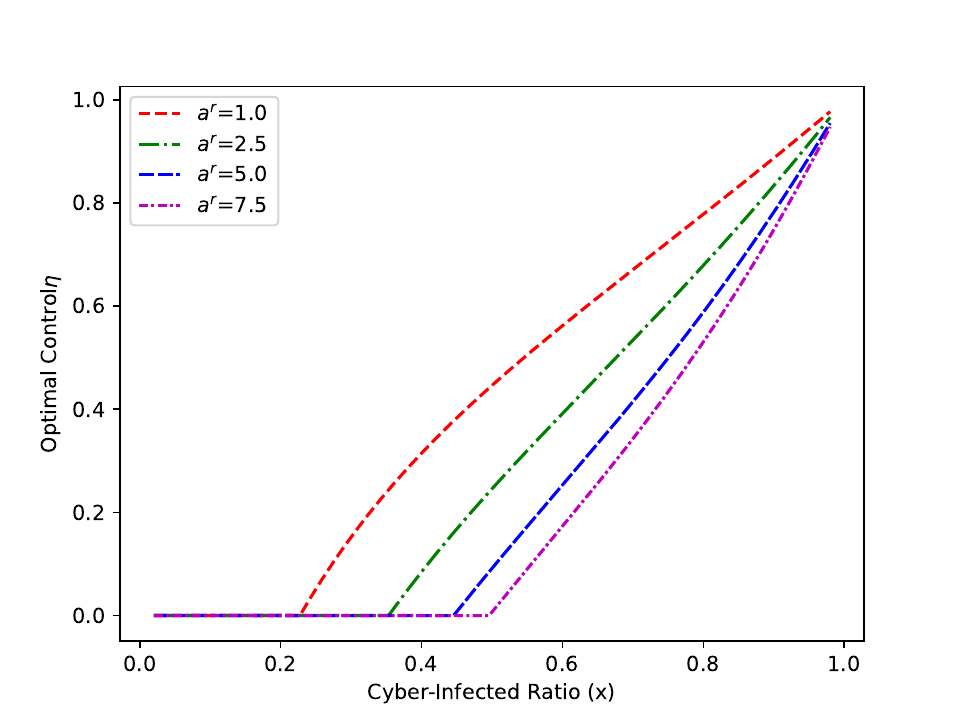}
            \caption{Optimal $\eta^*$ for different $a_r$}
            \label{fig:compare_eta_ar}
        \end{subfigure}
        \begin{subfigure}{0.48\textwidth}
            \centering
            \includegraphics[width=\textwidth]{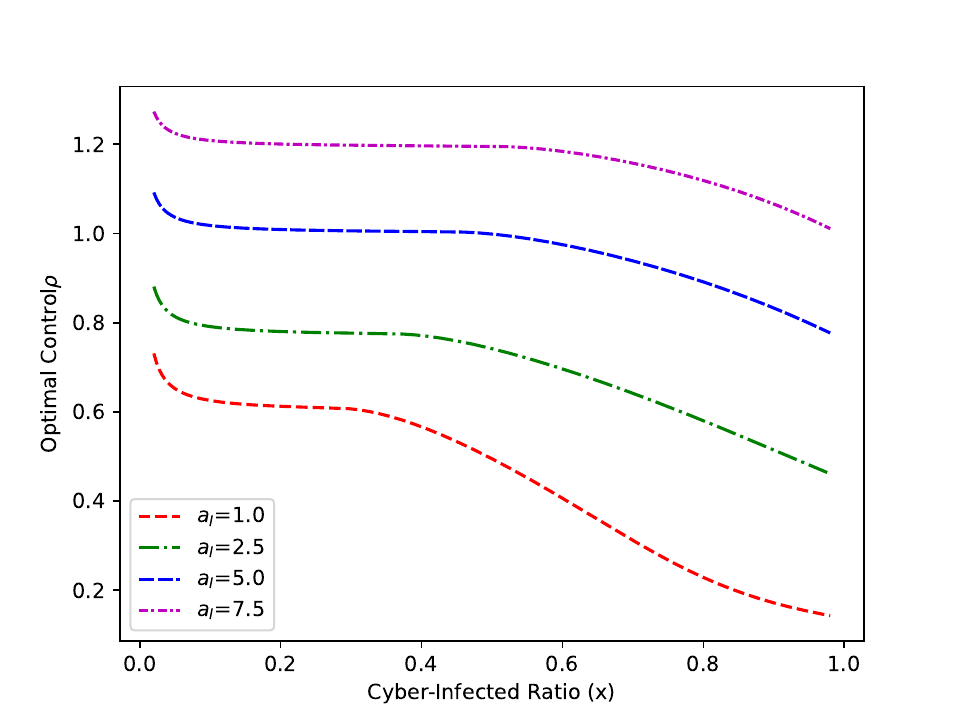}
            \caption{Optimal $\rho^*$ for different $a_I$}
            \label{fig:compare_rho_aI}
        \end{subfigure}
        \hfill
        \begin{subfigure}{0.48\textwidth}
            \centering
            \includegraphics[width=\textwidth]{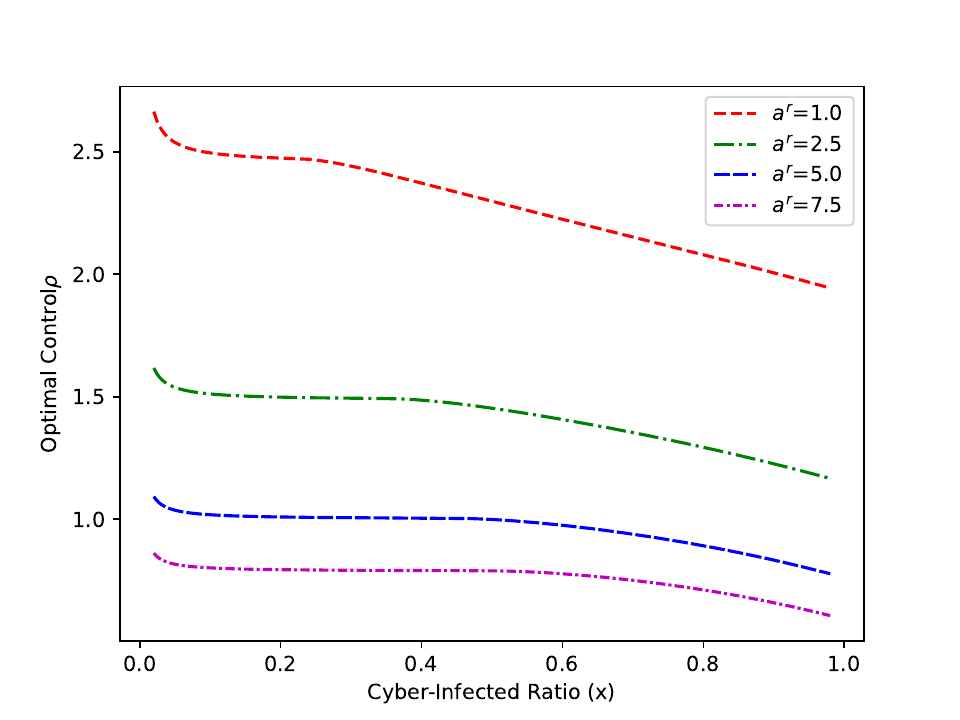}
            \caption{Optimal $\eta^*$ for different $a_r$}
            \label{fig:compare_rho_ar}
        \end{subfigure}
        \caption{Comparison analysis for $a_I$(left) and $a_r$(right).}
        \label{fig:comparison-aIr}
    \end{figure}

When the marginal costs $a_I$ increase, Figures \ref{fig:compare_eta_aI} and \ref{fig:compare_rho_aI} show a similar (but in the opposite way) result of the optimal management control $\eta^*$ as we observed in Figures 
\ref{fig:compare_eta_aIm} and \ref{fig:compare_rho_aIm}, 
where a stronger prevention measure is achieved by expanding the zero-valued plateau (representing maximal prevention) for a larger range of cyber-infected ratios. 
This expansion occurs in an almost uniform, additive manner, where as $a_I$ rises, the switching threshold at which $\eta^*$ departs from zero shifts rightward by roughly the same increment. 
In addition, with a small value in $a_I$ (e.g., $a_I=1$), the optimal proactive management control $\eta^*(x)$, as a function of cyber-infected ratio $x$, exhibits a concave property; 
and the concavity diminishes when $a_I$ increases and eventually exceeds the value of $a_m^I$. 
However, the optimal reactive mitigation $\rho^*$ adjusts in a significantly different way compared with what we obtained in Figure \ref{fig:compare_rho_aIm}, 
especially when the cyber-infected ratio ($x$) is close to one.
To be specific, when the system is heavily compromised, for a large value of $a_I$ (compared to $a_m^I$),
it is optimal to increase the reactive mitigation control to reduce the number of infected nodes so that the expected discounted costs can be reduced significantly.


On the other hand, Figures \ref{fig:compare_eta_ar} and \ref{fig:compare_rho_ar} show that increasing the mitigation cost $a_r$ from $1$ to $2.5$, $5$, and $7.5$ leads to a systematic weakening of the mitigation strategy $\rho^*$, while the management strategy $\eta^*$ is strengthened by expanding its zero-valued plateau (maximal prevention) so that strong prevention is applied earlier and more widely. The movement pattern of $\rho^*$ is similar to that observed in Figures~\ref{fig:compare_eta_aI} and \ref{fig:compare_rho_aI}, in the sense that the adjustment is roughly uniform across the state space, but the direction is opposite: a higher value of $a_I$ lifts $\rho$ proportionally, whereas a higher mitigation cost $a_r$ pushes $\rho^*$ downward. This observation highlights that when the operating costs associated with cyber-infected nodes become more costly, it is optimal to reinforce both risk management and mitigation strategies. However, when mitigation becomes expensive, the optimal strategy reallocates effort towards risk management control with a reduced reliance on expensive reactive mitigation controls.

\section{Conclusion and future outlook}\label{sec:5}
In this paper, we model cyber risk management and mitigation as a stochastic optimal control problem within a stochastic Susceptible-Infected-Susceptible (SIS) epidemic framework. We introduce two dynamic controls to capture real-time risk management and mitigation strategies: 1) a proactive control ($\eta$) that reduces external cyber attacks and internal contagion effects; 2) a reactive control ($\rho$) that speeds up the recovery of infected nodes.
We formulate this as a dual stochastic control problem governed by a general diffusion process. Theoretically, we establish the well-posedness of the controlled SIS model under these dual controls and prove that the associated value function is the unique increasing and Lipschitz-continuous viscosity solution of the Hamilton-Jacobi-Bellman (HJB) equation derived from the control problem.

For numerical solutions, we propose a Policy Improvement Algorithm (PIA) and demonstrate its convergence using Backward Stochastic Differential Equations (BSDEs). Our convergence result extends existing finite-horizon analyses to the infinite-horizon case. Then, we present a benchmark example that illustrates the optimal risk management and mitigation strategy, along with the corresponding value function, for a given model parameter set. We further examine suboptimal performance and sensitivity by: 1) removing one control entirely in the benchmark scenario; 2) introducing small perturbations to each optimal control; 3) conducting a comprehensive comparative statics analysis across all model parameters. The sensitivity and suboptimal control analyses reveal a fundamental asymmetry between the two control strategies, where proactive risk management control demonstrates consistent system-wide benefits and exhibits partial substitutability for reactive mitigation when absent; however, reactive risk mitigation control shows value only during high-infection scenarios and cannot compensate for missing proactive measures. Furthermore, some interesting observations are drawn from the comparative statics, including: the asymmetric impact of external attack frequency versus internal contagion rates on optimal control strategies underscores a possible critical policy implication; effective cyber defense requires prioritizing resource allocation toward external threat management; the optimal control strategy, particularly proactive risk management, exhibits significantly different behavioral patterns depending on the current infection ratio. This variation stems from the interaction between the operational costs of maintaining all infected nodes in the system and the marginal costs of implementing risk management controls on these compromised nodes. 
\par
Finally, we remark that our work lays a foundation for several natural extensions in the field. One direction is incorporating jump processes to model sudden, large-scale cyber attacks or system failures, which could better capture extreme events beyond the diffusion approximation. Another extension involves regime-switching dynamics, where the network environment or external threat landscape changes over time, influencing both infection propagation and optimal control strategies. Further research may also explore multi-layered or networked SIS models with heterogeneous nodes, time delays, and partial observation, enabling more realistic and granular cyber risk management strategies. These extensions could provide a richer theoretical framework and more practical insights for robust cyber defense policies under uncertainty and complex operational conditions. We left them for future research.

\section*{Statements and Declarations}
No competing interests.
\section*{Acknowledgments}
Zhuo Jin and Hailiang Yang were supported by the National Natural Science Foundation of China Grant [Grant 12471452]. Ran Xu was supported by the National Natural Science Foundation of China [Grants 12201506 and 12371468]. 

\begin{appendices}
\section{Proof of Proposition \ref{prop:2.1}}\label{App:A}
(i) To prove the assertion, we first show that the cost functional $J$ defined in \eqref{eq:costf} is non-decreasing in $x$ if $f(x, \cdot, \cdot)$ is non-decreasing in $x$ for each pair of admissible controls. This can be proved using the density argument and It\^o's formula. To be specific, let's firstly argue the Yamada \& Watanabe's comparison principle of It\^o's diffusion (see, e.g., \cite{KaratzasShreve1991}).
   The diffusion term $\sigma(x)$ holds the locally Lipschitz property, and further observe that $|\sigma(x) - \sigma(y)|\leq h(|x - y|)$ by simply taking $h(x):=3\sigma x := ax$. There exists a strictly decreasing sequence $\{a_n\}_{n\in \mathbb N}\downarrow 0$ with $a_0 := 1$ and $\int_{[a_{n - 1}, a_n]} h^{-2}(u)du = n$ for every $n\in \mathbb N$. To see this, one explicitly has $\int_{[a_{n - 1}, a_n]} h^{-2}(u)du = \frac{1}{a^2}\cdot\left(\frac{1}{a_n} - \frac{1}{a_{n+1}}\right) = n$, which gives $a_n = \frac{2}{2 + a^2n(n+1)}$ satisfying such properties as required for each $n\in \mathbb N$. Moreover, we would like to take a nonnegative continuous function $p(x)$ dominated by $2/(nh^{2}(x)0$, such that $1\leq \int_{[a_{n - 1}, a_n]} p(u)\dif u \leq \int_{[a_{n - 1}, a_n]} 2/(nh^{2}(u))\dif u = 2$, so that we get a normalized function $\rho_n(x): (a_{n-1}, a_n)\to \mathbb R$ continuous in $x$ taking the form of 
   \[0\leq \rho_n(x) := \frac{p(x)}{\int_{(a_{n - 1}, a_n)} p(u)\dif u} \cdot I_{(a_{n - 1}, a_n)}(x) \leq p(x) \leq \frac{2}{nh^{2}(x)}. 
   \]
   For example, one can take $p(x) = 1/(nh^2(x))$. Notice also that a property 
    $$\int_0^{|x|} \rho_n(x)\dif  x \leq \int_0^\infty \frac{p(x)}{\int_{(a_{n - 1}, a_n)} p(u)\dif u} \cdot I_{(a_{n - 1}, a_n)}(x) \dif x = \int_{(a_{n - 1}, a_n)} \frac{p(x)}{\int_{(a_{n - 1}, a_n)} p(u)du}\dif x = 1 $$
    holds for some real number $x$. Next, assign the function 
    
    $$\psi_n(x):= \int_0^{|x|}\int_0^y \rho_n(u) \dif u \dif y, \;\text{ on }\;\mathbb R$$ 
    
    which is even and twice continuously differentiable: $$|D_x\psi_n(x)| = \left|\frac{d}{dx}\left(\int_0^{|x|}\int_0^y \rho_n(u)du dy\right)\right| = \int_0^{|x|}\rho_n(u)du \leq 1$$ by fundamental theorem of calculus, and 
    \begin{equation*}
        \begin{split}
            \lim_{n\nearrow \infty} \psi_n(x) &= \int_0^{|x|} \lim_{n\nearrow \infty}\int_{[0, y]} \rho_n(u)\dif u\dif y \\
            &= \int_0^{|x|}\lim_{n\nearrow \infty}\int_{(a_{n - 1}, a_n)\cap (0, y)} \frac{p(u)}{\int_{(a_{n - 1}, a_n)} p(v)dv} \cdot \dif u \dif y\\
            &= \int_0^{|x|}\lim_{n\nearrow \infty}\int_{(a_{n - 1}, a_n)} \frac{p(u)}{\int_{(a_{n - 1}, a_n)} p(v)\dif v} \cdot \dif u \dif y\\
            &= \int_0^{|x|} 1 \cdot \dif y = |x|, 
        \end{split}
    \end{equation*} 
    where observe that $\{\int_0^y \rho_n(u)\dif u\}_{n\in \mathbb N} \uparrow 1$ is at least non-decreasing sequence for each $y\geq 0$, hence allowing us to apply the monotone convergence theorem. Consider two random processes (namely the solution to the corresponding controlled SDE \eqref{eq:SDE-X}) $X^{x, \eta, \rho}_t$ and $X^{y, \eta, \rho}_t$ for different initial data $x\leq y$, each of that has continuous trajectory for every individual $\omega\in \Omega$. Now, apply the It\^o's formula for the random process $\varphi_n(X^{x, \eta, \rho}_s - X^{y, \eta, \rho}_s):= \psi_n(X^{x, \eta, \rho}_s - X^{y, \eta, \rho}_s)\cdot \mathbf{1}_{(0, \infty)}(X^{x, \eta, \rho}_s - X^{y, \eta, \rho}_s)$: 
    \begin{equation}
    \label{eq:A1}
        \begin{split}
            \mathbb{E}\Big[\varphi_n(X^{x, \eta, \rho}_t - X^{y, \eta, \rho}_t)\Big] \leq& \mathbb{E}\Big[\varphi_n(X^{x, \eta, \rho}_t - X^{y, \eta, \rho}_t)\Big] - (x - y)\\
            =& \mathbb{E}\Bigg[\int_0^t D_x\varphi_n(X^{x, \eta, \rho}_s - X^{y, \eta, \rho}_s) \cdot \left[b(X^{x, \eta, \rho}_s) - b(X^{y, \eta, \rho}_s)\right] \dif s \Bigg]\\
            +& \frac{1}{2}\mathbb{E}\Bigg[\int_0^t D_{xx}\varphi_n(X^{x, \eta, \rho}_s - X^{y, \eta, \rho}_s) \cdot \left[\sigma(X^{x, \eta, \rho}_s) - \sigma(X^{y, \eta, \rho}_s)\right]^2 \dif s\Bigg]\\
            \leq & \mathbb{E}\Bigg[\int_0^t D_x\varphi_n(X^{x, \eta, \rho}_s - X^{y, \eta, \rho}_s) \cdot \left[b(X^{x, \eta, \rho}_s) - b(X^{y, \eta, \rho}_s)\right] \dif s\Bigg] + \frac{t}{n}\\
            \leq & K\cdot \mathbb E\Bigg[\int_0^t [X^{x, \eta, \rho}_s - X^{y, \eta, \rho}_s]^+ \dif s\Bigg] + \frac{t}{n},
        \end{split}
    \end{equation}
    where the stochastic integral vanished due to the fact $\mathbb{E}\Big(\int_0^t |\sigma(X_s)|^2\dif s\Big)< \infty$ for each $t\geq 0$; the second inequality holds by the established property of function $h$ and $\psi_n$; and further by the Lipschitz continuity of $b(\cdot)$  to  reach the last line in \eqref{eq:A1}. By sending $n\nearrow \infty$ on both sides of \eqref{eq:A1} with Lebesgue dominated convergence theorem yields
    \begin{align*}
        \lim_{n\nearrow\infty}\mathbb{E}\Big[\varphi_n(X^{x, \eta, \rho}_t - X^{y, \eta, \rho}_t)\Big]&  = \mathbb{E} \Big[\lim_{n\nearrow\infty}\varphi_n(X^{x, \eta, \rho}_t - X^{y, \eta, \rho}_t)\Big] \\
        &= \mathbb{E}\left[X^{x, \eta, \rho}_t - X^{y, \eta, \rho}_t\right]^+ \le K\cdot \mathbb E\Bigg[\int_0^t [X^{x, \eta, \rho}_s - X^{y, \eta, \rho}_s]^+ \dif s\Bigg]
    \end{align*}
    As a consequence, the desired comparison principle follows by Gr\"onwall's inequality, i.e. $$\mathbb E [X^{x, \eta, \rho}_t - X^{y, \eta, \rho}_t]^+ = 0, \;\forall t \geq 0.$$ 
    
    Secondly, let's assume that $\mathbb E\left[f(X^{x, \eta, \rho}_t, \eta_t, \rho_t) - f(X^{y, \eta, \rho}_t, \eta_t, \rho_t)\right]^+ > 0$. However, this immediately turns out that there exists some positive measure set $\Omega_+\subset\Omega$ such that $$\int_{\Omega_+} \left[X^{x, \eta, \rho}_t - X^{y, \eta, \rho}_t\right]^+  \mathbb P(\dif \omega) > 0$$ by the non-decreasing property of cost functional. This apparently contradicts the result we just obtained.
    
    Thirdly, we can thus claim that the comparison principle of two objective functions: $J(x, \eta, \rho)\leq J(y, \eta, \rho)$ if $x\leq y$ under each control $(\eta, \rho)\in \mathcal U_0$. To show this, observe that
    \begin{align*}
        & \mathbb E\left[\int_0^\infty e^{-\delta t}[f(X^{x, \eta, \rho}_t, \eta_t, \rho_t) - f(X^{y, \eta, \rho}_t, \eta_t, \rho_t)]^+ \dif t\right]\\
        &\quad  \leq 2\sup_{z\in\{x, y\}} \mathbb E\left[\int_0^\infty e^{-\delta t}|f(X^{z, \eta, \rho}_t, \eta_t, \rho_t)| \dif t\right]< \infty
    \end{align*}
    by Assumption~\ref{a2}. That, together with the joint measurability of process $f(X^{z, \eta, \rho}_t, \eta_t, \rho_t)$ on product $\sigma$-field $\mathcal B([0, t])  \otimes \mathcal F_t$ for every $t\geq 0$ and $z\in\{x, y\}$, ensures the applicability of Fubini's theorem. Therefore, interchange the order of the two integrals 
    \begin{equation*}
        \begin{split}
            &\mathbb E\left[\int_0^\infty e^{-\delta t}[f(X^{x, \eta, \rho}_t, \eta_t, \rho_t) - f(X^{y, \eta, \rho}_t, \eta_t, \rho_t)]^+ \dif t\right] \\
            & \quad = \int_0^\infty e^{-\delta t}\mathbb E_x[f(X^{x, \eta, \rho}_t, \eta_t, \rho_t) - f(X^{y, \eta, \rho}_t, \eta_t, \rho_t)]^+ \dif t = 0.
        \end{split}
    \end{equation*}

Hence, the non-decreasing property of $V$  is a straightforward argument. To be specific, we consider $0<x\le y<1$, and for any $\varepsilon>0$, let $(\eta^\varepsilon(y),\rho^\varepsilon(y))\in\mathcal{U}_0$ be an $\varepsilon$-optimal control for $X_0=y$ such that $J(y;\eta^\varepsilon(y),\rho^\varepsilon(y)) \le V(y)+\varepsilon$. Then, since $J(x;\cdot,\cdot)$ is non-decreasing in $x$ for any given admissible control, we have
    \[
    V(x)\le J(x;\eta^\varepsilon(y),\rho^\varepsilon(y)) \le J(y;\eta^\varepsilon(y),\rho^\varepsilon(y)) \le V(y)+\varepsilon,
    \]
    and by sending $\varepsilon \to 0$, we complete the proof.

 (ii) We first show a similar result as stated in \eqref{ineq:V}  holds for the cost functional $J$ for any given admissible control.  For fixed $(\eta, \rho)\in \mathcal U_0$,  and take any $x, y\in (0, 1)$, and a time $0<T<\infty$, we have
    \begin{align}\label{eqj}
        &|J(x; \eta, \rho) - J(y; \eta, \rho)| \nonumber\\
        &\leq \mathbb E\left[\int_0^\infty e^{-\delta t}\left|f(X_t^{x,\eta, \rho}, \eta_t, \rho_t) - f(X_t^{y,\eta, \rho}, \eta_t, \rho_t)\right|\dif t\right] \nonumber\\
        &\leq K\mathbb E\left[\int_0^{T} e^{-\delta t}(1 + |X_t^{x,\eta, \rho}|^m + |X_t^{y,\eta, \rho}|^m)|X_t^{x,\eta, \rho} - X_t^{y,\eta, \rho}|\dif t\right] + 2C_f\int_{T}^\infty e^{-\delta t}\dif t \nonumber \\
        &\leq C_{K, m} \left(\mathbb{E}\Bigg[\int_0^{T} e^{-\frac{(m+1)\delta t}{2m}}(1 + |X_t^{x,\eta, \rho}|^{m+1} + |X_t^{y,\eta, \rho}|^{m+1})\dif t\Bigg]\right)^{\frac{m}{m+1}}\nonumber \\
        &\qquad \times \left(\mathbb{E}\Bigg[\int_0^{T} e^{-\frac{(m+1)\delta t}{2}}(|X_t^{x,\eta, \rho} - X_t^{y,\eta, \rho}|^{m+1})\dif t\Bigg]\right)^{\frac{1}{m+1}} + 2C_fe^{-\delta T}. 
    \end{align}
    where $X^{x,\eta,\rho}_{\cdot}$ denotes the controlled process under fixed control $(\eta,\rho)\in\mathcal{U}_0$ with initial value $x$.
    Next, we estimate the following two terms 
    $$I_1:=\mathbb{E}\Bigg[\int_0^{T} e^{-\frac{(m+1)\delta t}{2m}}(1 + |X_t^{x,\eta, \rho}|^{m+1} + |X_t^{y,\eta, \rho}|^{m+1})\dif t\Bigg]$$ 
    and 
    $$I_2:=\mathbb{E}\Bigg[\int_0^{T} e^{-\frac{(m+1)\delta t}{2}}(|X_t^{x,\eta, \rho} - X_t^{y,\eta, \rho}|^{m+1})\dif t\Bigg].$$ 
    Therefore, we can apply boundedness property of the moment of the controlled SDE within $[0, T]$ for any $T<\infty$, there exists a constant $N = N(m, K)$ such that
    \begin{align*}
        I_1 &\leq \mathbb{E}\Bigg[\int_0^{T} e^{-\frac{(m+1)\delta t}{2m}}(1 + \sup_{0\leq t\leq T}|X_t^{x,\eta, \rho}|^{m+1} + \sup_{0\leq t\leq T}|X_t^{y,\eta, \rho}|^{m+1})\dif t\Bigg] \\
        &\leq \left(1 + Ne^{NT}(1 + |x|^{m+1})+ Ne^{NT}(1 + |y|^{m+1})\right) 
    \end{align*}
    and 
    \begin{align*}
        I_2 \leq Ne^{NT} |x - y|^{m+1}, 
    \end{align*}
    which can be seen, for example, in Corollary 2.5.12 in \cite{krylov1980controlled}. Thus, by 
    combining $I_1, I_2$ and equation (\ref{eqj}) and applying Minkowski's inequality, we attain that
    $$|J(x; \eta, \rho) - J(y; \eta, \rho)| \leq C(T, m, \delta)(1 + |x|^{m} + |y|^{m})|x - y|.$$ 
 Then, without loss of generality, we consider  $x \ge  y$ such that $V(x) \ge V(y)$. Take an $\varepsilon$-optimal control $(\eta^\varepsilon, \rho^\varepsilon)\in \mathcal U_0$ such that $V(y) \geq J(y; \eta^\epsilon, \rho^\epsilon) - \varepsilon$. Thus, we obtain the following inequality 
 $$|V(x) - V(y)| = V(x) - V(y) \leq J(x;\eta^\epsilon, \rho^\epsilon) - J(y;\eta^\epsilon, \rho^\epsilon) + \varepsilon \leq  C(1 + |x|^m + |y|^m)|x - y| + \varepsilon.$$
 The desired result will be achieved by sending $\varepsilon$ to zero. Furthermore, $V(x)$ is also continuous uniformly in $x\in (0, 1)$. The Lipschitz and uniform continuity will follow again due to $|x|^m, |y|^m\in (0, 1)$. 

\section{Proof of Proposition \ref{prop:existence}}\label{app:existence}
(i) \textit{Viscosity subsolution}.  We consider any test function $\phi \in C^2((0,1))$ with $\phi \ge V$ such that $\phi(x)=V(x)$ for any given $x\in(0,1)$, we just need to show that
\[
\inf_{(\rho,\eta)\in U}\Big\{(\mathcal{L}^{\eta,\rho} - \delta )\phi(x) + f(x, \eta, \rho)\Big\} \ge 0,
\]
where 
\[
\mathcal{L}^{\eta,\rho}\phi(x) = b(x,\eta,\rho)\phi_x(x) + \frac{1}{2}\sigma^2(x) \phi_{xx}(x)
\]
is the infinitesimal generator of the controlled diffusion given by \eqref{eq:SDE-X} under a pair of fixed control $(\eta,\rho)\in U$. Fix any  $(\eta,\rho)\in U$, consider the control $\eta_t\equiv \eta,\rho_t\equiv \rho$ and the corresponding controlled diffusion $X$. For a sufficiently small $r>0$ such that $(x-r,x+r)\subset (0,1)$, consider the stopping time $\tau:=\inf\{s>0: X_s\notin (x-r,x+r)\}$. Then, by applying It\^o's formula and the dynamic programming principle, we have
\begin{align}\label{eq:Prop3-1}
 0 &\le  \mathbb{E}_x\left[ e^{-\delta\tau} V(X_\tau) +\int_0^\tau e^{-\delta t} f(X_t,\eta_t, \rho_t)\dif t-V(x)\right] \nonumber\\
 & \le  \mathbb{E}_x\left[ e^{-\delta\tau} \phi(X_\tau) +\int_0^\tau e^{-\delta t} f(X_t,\eta_t, \rho_t)\dif t-\phi(x)\right]\nonumber\\
 & = \mathbb{E}_x\left[\int_{0}^{\tau} e^{-\delta t} (\mathcal{L}^{\eta,\rho}-\delta)\phi(X_t)\dif t +\int_0^\tau e^{-\delta t} f(X_t,\eta_t, \rho_t)\dif t\right].
 \end{align}
Now, we assume that $L(x):=(\mathcal{L}^{\eta,\rho} - \delta )\phi(x) + f(x, \eta, \rho)\le - \epsilon<0$ for any given $(\eta,\rho)\in U$. By the continuity of the function $L$, there exists a $h>0$ such that $L(y)< -\epsilon/2$ for all $y\in(x-h,x+h)$. Then, let $\tilde{\tau}:=\inf\{s> 0: X_s\notin(x-h,x+h)\}$, we have
\[
\mathbb{E}_x\Big[\int_{0}^{\tilde{\tau}} e^{-\delta t} L(X_t)\dif t\Big] < -\mathbb{E}_x\Big[\int_{0}^{\tilde{\tau}}e^{-\delta t}\frac{\epsilon}{2}\dif t\Big] = -\mathbb{E}_x\Big[\frac{\epsilon(1-e^{-\delta \tilde{\tau}})}{2\delta}\Big]<0,
\]
which is a contradiction to \eqref{eq:Prop3-1} if we choose $r=h$ above. Hence, by the arbitrariness of the control, we must have
\[
\inf_{(\rho,\eta)\in U}\Big\{(\mathcal{L}^{\eta,\rho} - \delta )\phi(x) + f(x, \eta, \rho)\Big\} \ge 0.
\]
(ii) \textit{Viscosity supersolution}. For any $x\in(0,1)$, let $\varphi\in C^2((0,1))$ be any test function such that $\varphi-V$ attains a maximum value of zero at $x$, we show that 
\[
\inf_{(\rho,\eta)\in U}\left\{(\mathcal{L}^{\eta,\rho}- \delta) \varphi(x) + f(x, \eta, \rho)\right\} \le 0.
\]
We prove this by contradiction. Assume that $\inf_{(\rho,\eta)\in U}\left\{(\mathcal{L}^{\eta,\rho}- \delta) \varphi(x) + f(x, \eta, \rho)\right\} >0$, then by the continuity of the function $(\mathcal{L}^{\eta,\rho}- \delta) \varphi(\cdot)$ and $f(\cdot,\eta,\rho)$ uniformly in the control, there exist $\epsilon>0$ and $\varpi>0$ such that
\[
\inf_{(\rho,\eta)\in U}\left\{(\mathcal{L}^{\eta,\rho}- \delta) \varphi(y) + f(y, \eta, \rho)\right\} >\epsilon,\quad \text{for all } y\in(x-\varpi,x+\varpi)\subset (0,1).
\]
Then, for any fixed control $(\eta_t,\rho_t)_{t\ge0} \equiv (\eta,\rho)\in U$, let $X_t$ be the controlled process with $X_0=x$. We define the first exit time of the interval $(x-\varpi,x+\varpi)$ as $\tau^\varpi:= \inf\{t>0: X_t\notin (x-\varpi,x+\varpi)\}$. Then, by applying It\^o's formula, we have
\begin{align*}
    \varphi(x) &= \mathbb{E}_x\Bigg[ e^{-\delta (t\wedge \tau^\varpi)}\varphi(X_{t\wedge \tau^\varpi}) - \int_{0}^{t\wedge \tau^\varpi}e^{-\delta s}(\mathcal{L}^{\eta,\rho}- \delta)\varphi(X_s)\dif s\Bigg]\\
    & \le \mathbb{E}_x\Bigg[ e^{-\delta (t\wedge \tau^\varpi)}V(X_{t\wedge \tau^\varpi}) +\int_{0}^{t\wedge \tau^\varpi}e^{-\delta s}f(X_t,\eta_t,\rho_t)\dif s\Bigg]\\
    &\quad - \mathbb{E}_x\Bigg[\Big(\int_{0}^{t\wedge \tau^\varpi}e^{-\delta s}(\mathcal{L}^{\eta,\rho}- \delta)\varphi(X_s)\dif s + \int_{0}^{t\wedge \tau^\varpi}e^{-\delta s}f(X_t,\eta_t,\rho_t)\dif s\Big)\Bigg]\\
    & \le \mathbb{E}_x\Bigg[ e^{-\delta (t\wedge \tau^\varpi)}V(X_{t\wedge \tau^\varpi}) +\int_{0}^{t\wedge \tau^\varpi}e^{-\delta s}f(X_t,\eta_t,\rho_t)\dif s\Bigg] - \epsilon\mathbb{E}_x\Bigg[\frac{1-e^{-t\wedge\tau^{\varpi}}}{\delta}\Bigg].
    \end{align*}
    Then, by the arbitrariness of control $(\eta,\rho)\in U$, the dynamic programming principle and the fact that $\tau^\varpi>0$ for all $\omega\in\Omega$, we obtain that $ \varphi(x) < V(x)$, which is a contradiction since $\varphi(x) =V(x)$.
\section{Proof of Proposition \ref{prop:uniqueVS}}\label{app:unique}
      We prove the comparison principle by the usual arguments of contradiction (see e.g. \cite{touzi2012optimal}, \cite{albrecher2020optimalBW}). Assume that
    \[
    0< M:= \sup_{x\in \mathcal{O}}\{v(x) - u(x)\},
    \]
    and $x^* :=\argmax_{x\in\mathcal{O}} \{v(x) - u(x)\}$.  Since both $u$ and $v$ are Lipschitz continuous on $\mathcal{O}\subset (0,1)$, there exists a constant $K>0$ such that
    \[
|u(x)-u(y)| \le K|x-y|,\quad |v(x)-v(y)|\le K|x-y|,\quad \text{for } x,y\in\mathcal{O}.
    \]
    To proceed, let us consider the following set
    \[
    \mathcal{S}:=\Big\{(x,y)\in \mathcal{O}\times\mathcal{O}: x\le y\Big\}. 
    \]
     For all $\alpha >0$, define two auxiliary functions,
    \begin{align*}
    \Psi_\alpha(x,y) &:= \frac{\alpha}{2}(x-y)^2 +\frac{2K}{\alpha^2(y-x)+ \alpha},\\
     H_{\alpha}(x,y)&:= v(x) - u(y) -\Psi_\alpha(x,y).
    \end{align*}
Let $M_\alpha := \max_{(x,y)\in\mathcal{\mathcal{S}}}H_\alpha(x,y)$, and $(x_\alpha,y_\alpha):= \argmax_{(x,y)\in\mathcal{S}}H_\alpha(x,y)$. Then, we directly have $M_\alpha \ge H_\alpha(x^*,x^*) = M-\frac{2K}{\alpha}$, hence
\begin{equation*}
    \liminf_{\alpha\to \infty}M_\alpha\ge M>0. 
\end{equation*}
Note that, by using the increasing and Lipschitz continuous property of $u$ and $v$ and the fact that $u\ge v$ on $\partial\mathcal{O}$, we can show that there exists a sufficiently large $\overline{\alpha}$  such that for all $\alpha\ge \overline{\alpha}$, $(x_\alpha,y_\alpha)$ is an interior point of $\mathcal{O}$.
Then by the following inequality, 
\[
H_\alpha(x_\alpha,x_\alpha) +H_\alpha(y_\alpha,y_\alpha) \le 2H_\alpha(x_\alpha,y_\alpha), 
\]
we arrive at
\[
\alpha|x_\alpha-y_\alpha|^2\le |u(x_\alpha)-u(y_\alpha)| + |v(x_\alpha)-v(y_\alpha)| + 4K(y_\alpha-x_\alpha) \le 6K|x_\alpha-y_\alpha|.
\]
Therefore, by considering a sequence $\alpha_n \to \infty $ as $n\to \infty$ such that $(x_{\alpha_n},y_{\alpha_n}) \to (\hat{x},\hat{y})\in\mathcal{O}$, we have
\[
|x_{\alpha_n}-y_{\alpha_n}|\le \frac{6K}{\alpha_n}
\]
which yields $\hat{x}=\hat{y}$. 

Then, we construct two twice continuously differentiable functions
\begin{align*}
   \varphi(x) &=  \Psi_\alpha(x,y_\alpha) - \Psi_\alpha(x_\alpha,y_\alpha) + v(x_\alpha),\\
   \psi(y)& =\Psi_\alpha(x_\alpha,y_\alpha) - \Psi_\alpha(x_\alpha,y) + u(y_\alpha), 
\end{align*}
which are essentially test functions for the subsolution and supersolution of \eqref{eq:HJB} at the points $x_\alpha$ and $y_\alpha$, respectively.  To simplify our proof here, we first assume that $v(x)$ and $u(y)$ are twice continuously differentiable at $x_\alpha$ and $y_\alpha$, respectively; one can resort to a more general theorem to get a similar result when $u,v$ are not twice continuously differentiable at the point (see e.g. \cite{crandall1992user}).  Since $H_\alpha$ reaches a local maximum at $(x_\alpha,y_\alpha)$ which is an interior point in $\mathcal{O}$, hence we have
\[
\frac{\partial}{\partial x} H_\alpha(x_\alpha,y_\alpha) = \frac{\partial}{\partial y} H_\alpha(x_\alpha,y_\alpha)=0.
\]
Therefore, we arrive at
\[
v_x(x_\alpha) = \frac{\partial}{\partial x} \Psi_\alpha(x_\alpha,y_\alpha) = -\frac{\partial}{\partial y} \Psi_\alpha(x_\alpha,y_\alpha)=u_y(y_\alpha).
\]
In addition, the Hessian matrix is negative semi-definite,
\begin{equation}\label{Hessian}
\begin{pmatrix}
 \frac{\partial^2}{\partial x^2} H_\alpha(x_\alpha,y_\alpha) &\frac{\partial^2}{\partial x \partial y} H_\alpha(x_\alpha,y_\alpha)  \\
 \frac{\partial^2}{\partial y \partial x} H_\alpha(x_\alpha,y_\alpha) & \frac{\partial^2}{\partial y^2} H_\alpha(x_\alpha,y_\alpha)
\end{pmatrix}\le 0
\end{equation}
Let $A = v_{xx}(x_\alpha), B= u_{yy}(y_\alpha)$, and 
\[
M =  \begin{pmatrix}
\frac{\partial^2}{\partial x^2} \Psi_\alpha(x_\alpha,y_\alpha) &\frac{\partial^2}{\partial x \partial y} \Psi_\alpha(x_\alpha,y_\alpha)  \\
 \frac{\partial^2}{\partial y \partial x} \Psi_\alpha(x_\alpha,y_\alpha) & \frac{\partial^2}{\partial y^2} \Psi_\alpha(x_\alpha,y_\alpha),
\end{pmatrix}
\]
we can rewrite \eqref{Hessian}  as
\[
\begin{pmatrix}
A- \frac{\partial^2}{\partial x^2} \Psi_\alpha(x_\alpha,y_\alpha) &-\frac{\partial^2}{\partial x \partial y} \Psi_\alpha(x_\alpha,y_\alpha)  \\
 -\frac{\partial^2}{\partial y \partial x} \Psi_\alpha(x_\alpha,y_\alpha) & -B-\frac{\partial^2}{\partial y^2} \Psi_\alpha(x_\alpha,y_\alpha)
\end{pmatrix} = \begin{pmatrix}
    A & 0\\
    0& -B
\end{pmatrix} - M\le 0.
\]
Then, according to \cite[Theorem 3.2]{crandall1992user}, for any $\varepsilon>0$, there exists $A_\varepsilon, B_\varepsilon\in\mathbb{R}$ such that, 
\begin{equation}\label{eq:ABineq}
 \begin{pmatrix}
     A_\varepsilon & 0\\
     0 & -B_\varepsilon
 \end{pmatrix} \le M + \varepsilon M^2,  
\end{equation}
and $(\varphi_x(x_\alpha),A_\varepsilon)\in D^{2,+}_{x_\alpha}v(x_\alpha)$ and $(\psi_x(x_\alpha),B_\varepsilon)\in D^{2,-}_{y_\alpha}u(y_\alpha)$, hence, we have
\begin{equation}\label{ineq:F}
    F(x_\alpha, v(x_\alpha), \varphi_x(x_\alpha), A_\varepsilon) \ge 0,\quad \text{ and }\quad  F(y_\alpha, u(x_\alpha), \psi_y(y_\alpha), B_\varepsilon) \le 0.
\end{equation}
In addition, by noting that $M =\frac{\partial^2}{\partial x^2} \Psi_\alpha(x_\alpha,y_\alpha)\begin{pmatrix}
    1&-1\\
    -1&1
\end{pmatrix} $, we obtain from \eqref{eq:ABineq} that 
\begin{align*}
   A_\varepsilon -B_\varepsilon &= \begin{pmatrix}
       1&1
   \end{pmatrix} \begin{pmatrix}
       A_\varepsilon & 0\\
       0 & -B_\varepsilon
   \end{pmatrix} \begin{pmatrix}
       1\\
       1
   \end{pmatrix} \le \begin{pmatrix}
       1&1
   \end{pmatrix} (M + \varepsilon M^2) \begin{pmatrix}
       1\\
       1
   \end{pmatrix}\\
   & = \begin{pmatrix}
       1&1
   \end{pmatrix} \Bigg[ \frac{\partial^2}{\partial x^2} \Psi_\alpha(x_\alpha,y_\alpha)\begin{pmatrix}
    1&-1\\
    -1&1
\end{pmatrix}+ \varepsilon \Big(\frac{\partial^2}{\partial x^2} \Psi_\alpha(x_\alpha,y_\alpha)\Big)^2\begin{pmatrix}
    1&-1\\
    -1&1
\end{pmatrix}\Bigg] \begin{pmatrix}
       1\\
       1
   \end{pmatrix}^2\\
   & =\Big(\frac{\partial^2}{\partial x^2}\Psi_\alpha(x_\alpha,y_\alpha)+ 2\varepsilon \Big(\frac{\partial^2}{\partial x^2} \Psi_\alpha(x_\alpha,y_\alpha)\Big)^2 \Big)\begin{pmatrix}
    1&1
\end{pmatrix}\begin{pmatrix}
    1&-1\\
    -1&1
\end{pmatrix}\begin{pmatrix}
    1\\
    1
\end{pmatrix}=0.
\end{align*}
Therefore, we can derive from \eqref{ineq:F},
\begin{align}
\label{ineq:F-F}
0 & \le F(x_\alpha, v(x_\alpha), \varphi_x(x_\alpha), A_\varepsilon) - F(y_\alpha, u(x_\alpha), \psi_y(y_\alpha), B_\varepsilon)\nonumber\\
&\le \inf_{(\rho,\eta)\in U}\Big\{b(x_\alpha,\eta,\rho)\varphi_x(x_\alpha) - b(y_\alpha,\eta,\rho)\psi_y(y_\alpha) + f(x_\alpha,\eta,\rho) - f(y_\alpha,\eta,\rho)\Big\}\nonumber\\
\quad & + \frac{1}{2}\Big(\sigma^2(x_\alpha) A_\varepsilon - \sigma^2(y_\alpha) B_\varepsilon\Big) - \delta (v(x_\alpha)-u(y_\alpha)).
\end{align}
Hence, by \eqref{ineq:F-F}, we can obtain
\begin{align*}
    0<& M \le \liminf_{\alpha\to \infty}M_{\alpha} \le \lim_{n\to \infty}M_{\alpha_n}  = \lim_{n\to \infty} (v(x_{\alpha_n}) - u(y_{\alpha_n}))\\
    & \le \lim_{n\to \infty}\frac{1}{\delta}\Bigg( \inf_{(\rho,\eta)\in U}\Big\{b(x_\alpha,\eta,\rho)\varphi_x(x_\alpha) - b(y_\alpha,\eta,\rho)\psi_y(y_\alpha) + f(x_\alpha,\eta,\rho) - f(y_\alpha,\eta,\rho)\Big\}\\
    &\qquad \qquad  + \frac{1}{2}\Big(\sigma^2(x_\alpha) A_\varepsilon - \sigma^2(y_\alpha) B_\varepsilon\Big)\Bigg) \\
    & = \inf_{(\rho,\eta)\in U}\{b(\hat{x},\eta,\rho)\}\frac{1}{\delta}(\varphi_x(\hat{x}) - \psi_y(\hat{x})) + \frac{\sigma^2(\hat{x})}{2\delta}(A_\varepsilon - B_\varepsilon)\\
    & \le \inf_{(\rho,\eta)\in U}\{b(\hat{x},\eta,\rho)\}\frac{1}{\delta}(\varphi_x(\hat{x}) - \psi_y(\hat{x}))=0,
\end{align*}
which is a contradiction, then we complete the proof. 
\section{Infinite-Horizon BSDE and Convergence of PIA}\label{App:B}
 Fix a filtered probability space $(\Omega, \mathcal{F}_t, \mathbb F:=( \mathcal{F}_t)_{0\leq t\leq\infty}, \mathbb P)$ and let $W=(W_t)_{0\leq t<\infty}$ be a one-dimensional Wiener process on this space. For self-containment, we introduce the following notations: 
\begin{itemize}
    \item[(i)] We first introduce the notations of some spaces that are involved in the later analysis: 
For any $\gamma>0$, let $\mathbb{L}_{\mathcal{F}}^{2,\gamma}(0,\infty;\mathbb{R})$ be the set of all $\mathbb{R}$-valued $\mathbb{F}$-adapted process $\phi_\cdot$ such that 
\[
\mathbb{E}\Bigg[ \int_{0}^{\infty}e^{2\gamma t}|\phi_t|^2\dif t\Bigg] <\infty.
\]
Let $\mathbb{L}_{\mathcal{F}}^{2,\gamma}(\Omega; C[0,\infty];\mathbb{R})$ be the set of all $\mathbb{R}$-valued $\mathbb{F}$-adapted continuous process $\phi_\cdot$ such that
\[
\mathbb{E}\Big[\sup_{t\in[0,\infty]}e^{2\gamma t}|\phi_t|^2\Big] <\infty.
\]
Finally, let $\mathbb{L}_{\mathcal{F}_{\tau}}^{2,\gamma}(\Omega;\mathbb{R})$ be the set of $\mathbb{R}$-valued $\mathcal{F}_\tau$-measurable random variables $\xi$ such that $\mathbb{E}\Big[e^{2\gamma \tau}|\xi|^2\Big]<\infty$, where $\tau$ is any $\mathbb{F}$-stopping time taking values in $[0,\infty]$. 

Then, we define the space
\begin{equation}
\mathbb{B}_{\gamma}[0,\infty] := \left\{\mathbb{L}_{\mathcal{F}}^{2,\gamma}(\Omega; C[0,\infty];\mathbb{R}) \bigcap \mathbb{L}_{\mathcal{F}}^{2,\gamma}(0,\infty;\mathbb{R})\right\} \times \mathbb{L}_{\mathcal{F}}^{2,\gamma}(0,\infty;\mathbb{R})
\end{equation}
with the norm
\begin{align*}
\|(Y_\cdot,Z_\cdot) \|_{\mathbb{B}_{\gamma}[0,\infty]} := \Bigg(\mathbb{E}\Bigg[\sup_{t\in[0,\infty]}e^{2\gamma t}|Y_t|^2   + \int_{0}^{\infty}e^{2\gamma t}|Y_t|^2\dif t + \int_{0}^{\infty}e^{2\gamma t}|Z_t|^2\dif t\Bigg]\Bigg)^{1/2},
\end{align*}
for any $(Y_\cdot,Z_\cdot)\in \mathbb{B}_{\gamma}[0,\infty]$. Obviously, $\mathbb{B}_{\gamma}[0,\infty]$ with the norm is a Banach space.
    \item[(ii)]  For a constant $\gamma>0$, and predictable process $(\phi_t)_{t\ge0}$, we introduce the $\mathbb{L}_\gamma^2$-norm: 
    $$\|\phi\|_{\mathbb{L}^2_\gamma}:= \left(\mathbb E\int_0^\infty e^{2\gamma t}|\phi_t|^2\dif t\right)^{1/2}.$$
    \item[(iii)] For adapted process $\phi_\cdot$ such that $\mathbb E\left[\int_0^\infty |\phi_t|^2 \dif t\right]<\infty$, we define $$(\phi\bullet W)_t:= \int_0^t \phi_s \dif W_s, \quad \text{for}\;  t\geq 0.$$
    \item[(iv)] For any continuous local martingale $M$ let  $(\langle M \rangle_t)_{t\geq 0}$ denote the quadratic variation process and let 
    $$\mathcal E(M)_t:= \exp\left\{M_t - \frac{1}{2} \langle M \rangle_t\right\}, \quad \text{for}\;  t\geq 0$$ denotes the Dol\'eans-Dade exponential of $M_t$ given the initial data $M_0 = 0$. 
\end{itemize}



\subsection{Some preliminary lemmas}
The following lemma is a straightforward result of Girsanov's theorem under the random process $\Big(\big(b^{u}\sigma^{-1}\big)(X_s)\Big)_{s\ge0}$, which is bounded (according to Assumption~\ref{a4}). This result will be helpful in the construction of a contraction mapping when proving Theorem \ref{them:convegencePIA} under a new measure $\widehat{\mathbb{P}}$ on the probability space. 
\begin{lemma}\label{girsanov}
    Let $0\leq t \leq \infty$, $x\in (0, 1)$, and $X:= X^{x, u^*}$ be the unique solution to the SDE~\eqref{eq:SDE-X}, started from time $0$ with initial data $X_0 = x$, and controlled by the optimal control process $u^* =(u^*_t)_{t\ge 0}$ (with a abuse of notation, we use $u = (\eta_{.},\rho_{.})\in\mathcal{U}_0$ to denote the control process as well). Then, $\dif \widehat{\mathbb P} :=\mathcal E\Big(\big(b^{u^*}\sigma^{-1}\big(X)\bullet W\Big)_\infty \dif \mathbb P$ is equivalent to the probability measure $\mathbb P$, and the process 
    $$\widehat W_t:= W_t + \int_0^t b^{u^*_s}(X_s)\sigma^{-1}(X_s)\dif s$$ is a $\widehat {\mathbb P}$-Wiener process. 
\end{lemma}
\begin{proof}
  See Theorem 6.8.8 of \cite{krylov2002introduction}. 
\end{proof}

\begin{assumption}\label{assumption-F}
For all $(z,Z)\in \mathbb R^2$, $F_s(z, Z)$ for $s\in [0, \infty]$ is progressively measurable. And, there exist constants $L_1, L_2,\gamma>0$ such that $F_\cdot(0,0)\in \mathbb{L}_{\mathcal{F}}^{2,\gamma}(0,\infty;\mathbb{R})$, and for any $z,Z,\bar{z},\bar{Z}\in \mathbb{R}$, $s\in[0,\infty)$,
\begin{equation}
|F_s(z,Z) - F_s(\bar{z},\bar{Z})|\le L_1|z-\bar{z}|+ L_2|Z-\bar{Z}|,\quad \mathbb{P}\text{-}a.s.,
\end{equation}
and 
\begin{equation}
    \gamma> \frac{L^2_1+L_2^2}{2}.
\end{equation}
\end{assumption}
The following lemma states the unique solution of an infinite-horizon BSDE associated with our infinite-horizon stochastic control problem.
\begin{lemma}
  Let Assumption \ref{assumption-F} hold. For any given $z_\cdot \in \mathbb{L}_{\mathcal{F}}^{2,\gamma}(0,\infty;\mathbb{R})$, the following BSDE
  \begin{equation}\label{eq:bsde-uniquelemma}
      Y_t =  \int_{t}^{\infty}F_s(z_s, Z_s)\dif s - \int_{t}^{\infty}Z_s\dif W_s,\quad t\in[0,\infty),
  \end{equation}
  admits a unique solution $(Y_\cdot,Z_\cdot)\in \mathbb{B}_{\gamma}[0,\infty]$. Further, there exists a constant $C>0$ such that if $(\overline{Y}_\cdot, \overline{Z}_{\cdot})\in\mathbb{B}_{\gamma}[0,\infty]$ is the solution of \eqref{eq:bsde-uniquelemma} with $F$  replaced by $\overline{F}$, where $\overline{F}$ satisfies Assumption \ref{assumption-F}, then
  \begin{align}
  \label{ineq:continuityY}
&\|(Y,Z) - (\overline{Y},\overline{Z})\|_{\mathbb{B}_{\gamma}[0,\infty]} \nonumber \\
\le&  C \Bigg(\int_{0}^{\infty}e^{2\gamma s}|F_s(z_s,Z_s)| - \overline{F}_s(z_s,Z_s)|^2 \dif s\Bigg)^{1/2}
  \end{align}
\end{lemma}
\begin{proof}
    See, for example, \cite[Theorem 7.3.6]{yong1999stochastic}.
\end{proof}

Moreover, the comparison principle for infinite-horizon BSDE was also well established, see, e.g., \cite[Theorem 2.2 ]{hamadene1999bsde}. 
\begin{lemma}\label{cmpbsde}
    Let $(Y^i, Z^i)\in\mathbb{B}_{\gamma}[0,\infty]$ be the solution to the following BSDE 
    $$Y^i_t = \xi^i + \int_t^\infty F^i_s(Z_s^i)\dif s - \int_t^\infty Z^i_s \dif W_s,~ t\in [0, \infty],~ i=1,2,$$ 
    where $\xi^i\in\mathbb{L}_{\mathcal{F}_{\infty}}^{2,\gamma}(\Omega;\mathbb{R})$ and $F^i$ satisfy the Assumption \ref{assumption-F} (with $z$ removed) for $i=1,2$, respectively. If further $\xi^1\geq \xi^2~ a.s.$,  and $F^1_s(Z^2_s) \geq F^2_s(Z^2_s)~ a.s.$ for all $s\geq 0$, then $Y^1_t\geq Y^2_t ~ \mathbb{P}\text{-}a.s.$ for all $t\geq 0$. 
\end{lemma}


Now, we are ready to present the following Lemma \ref{lemmaa5}, which plays a key role in the proof of Theorem \ref{them:convegencePIA}.  Lemma \ref{lemmaa5} follows a similar idea to Lemma A.5 of \cite{kerimkulov2020exponential}, see also Lemma 3.2 of \cite{Fuhrman}. 
\begin{lemma}\label{lemmaa5}
    Let $F:\Omega\times [0, \infty)\times \mathbb R\times \mathbb R \to \mathbb R$ be a measurable function that satisfies Assumption \ref{assumption-F}. Fix  $\delta> (L_1+L_2)^2/2$. Let $(Y_\cdot, Z_\cdot)\in \mathbb{B}_{\delta}[0,\infty]$ be the unique solution to \eqref{eq:bsde-uniquelemma} for any given $z_\cdot\in\mathbb{L}_{\mathcal{F}}^{2,\delta}(0,\infty;\mathbb{R})$.
    Moreover, assume that for $z^1_\cdot, z^2_\cdot \in \mathbb{L}_{\mathcal{F}}^{2,\delta}(0,\infty;\mathbb{R}) $ the following condition is satisfied:
    $$|F_t(z^1_t, Z^1_t) - F_t(z^2_t, Z^2_t)|\leq L_1|z^1_t - z^2_t| + L_2|Z^1_t - Z^2_t|,\quad t\in[0,\infty],~\mathbb{P}\textbf{-}a.s.$$
    Then there is $0<\gamma <\delta$ and $q\in (0, 1)$ such that for any $t\ge0$ we have 
    \begin{equation}\label{ineq:LemmaD4}
       \mathbb{E}\Big[e^{2\gamma t}|Y^1_t - Y^2_t|^2\Big] + \|Z^1 - Z^2\|^2_{\mathbb{L}^2_\gamma} \leq q\|z^1 - z^2\|^2_{\mathbb{L}^2_\gamma}, 
    \end{equation}
 where $(Y^i_\cdot,Z^i_\cdot)\in\mathbb{B}_{\delta}[0,\infty]$ is the unique solution to \eqref{eq:bsde-uniquelemma} corresponding to $z_i\in \mathbb{L}_{\mathcal{F}}^{2,\delta}(0,\infty;\mathbb{R}) $ for $i=1,2$, respectively.
 Furthermore, one can choose  $\gamma$ sufficiently large and $\delta-\gamma>0$ sufficiently small such that $q\in(0,1/2)$, and the above results hold as well.
\end{lemma}

\begin{proof}
 Denote $\Delta Y:=Y^1-Y^2, \Delta Z:=Z^1 -Z^1,$ and $\Delta z:=z^1-z^2$. Then, apply It\^o's formula to $e^{2\gamma s}|\Delta Y_s|^2$: 
    \begin{equation*}
        \begin{split}
            e^{2\gamma t}|\Delta Y_t|^2 + \int_t^\infty e^{2 \gamma s} |\Delta Z_s|^2 \dif s &= \int_t^\infty e^{2 \gamma s}[2\Delta Y_s \big(F_s(z^1_s, Z^1_s) - F_s(z^2_s, Z^2_s)\big) - 2\gamma |\Delta Y_s|^2 ] \dif s\\
            &- 2\int_t^\infty e^{ 2 \gamma s} |\Delta Y_s||\Delta Z_s| \dif W_s. 
        \end{split}
    \end{equation*}
    By taking the expectation of the equation above, we get 
    \begin{equation}\label{eq:LD4-1}
    \begin{split}
       & \mathbb E\left[e^{2\gamma t}|\Delta Y_t|^2  + \int_t^\infty e^{2\gamma s} |\Delta Z_s|^2 \dif s \right]\\ 
       & \qquad  = \mathbb E\left[\int_t^\infty e^{ 2\gamma s}[2\Delta Y_s (F_s(z^1_s, Z^1_s) - F_s(z^2_s, Z^2_s)) - 2\gamma |\Delta Y_s|^2 ] \dif s \right].   
    \end{split} 
    \end{equation} 
    By the Lipschitz property of the generator and by Young's inequality, we observe that, for any $\epsilon > 0$, 
    \begin{align}\label{ineq:LD4-2}
            &\mathbb E\left[\int_t^\infty e^{2\gamma s}\big(2\Delta Y_s (F_s(z^1_s, Z^1_s) - F_s(z^2_s, Z^2_s)) - 2\gamma |\Delta Y_t|^2 \big) \dif s \right] \nonumber\\
            \leq & \mathbb E \left[\int_t^\infty e^{ 2\gamma s} \big(2\Delta Y_s (L_1|\Delta z_s| + L_2 |\Delta Z_s|) - 2\gamma |\Delta Y_t|^2\big)\dif s\right]\nonumber\\
            \leq & \mathbb E \left[\int_t^\infty e^{2\gamma s} \big((L_1 + L_2)\epsilon\Delta |Y_s|^2 + L_1\epsilon^{-1} |\Delta z_s|^2 + L_2\epsilon^{-1} |\Delta Z_s|^2 - 2\gamma |\Delta Y_t|^2\big)\dif s\right]\nonumber \\
            \leq &  \max\left\{\frac{(L_1 + L_2)L_1}{2\gamma}, \frac{(L_1+ L_2)L_2}{2\gamma}\right\} \cdot \mathbb E \left[\int_t^\infty e^{2\gamma s} \Big( |\Delta z_s|^2 + |\Delta Z_s|^2 \Big)\dif s\right],
    \end{align}
    where the last inequality holds by choosing $\epsilon = 2\gamma/(L_1+L_2)$. Then, take $\gamma$ sufficiently large such that $\delta-\gamma>0$ sufficiently small, we have
    $$\tilde{q}:= \max\left\{\frac{(L_1 + L_2)L_1}{2\gamma}, \frac{(L_1 + L_2)L_2}{2\gamma}\right\} \in \left(0, \frac{1}{2}\right),$$ 
    and subtracting $\tilde q\cdot\mathbb E\big[\int_t^\infty e^{2\gamma s} |\Delta Z_s|^2 \dif s\big]$ on both sides of \eqref{eq:LD4-1} together with \eqref{ineq:LD4-2}, we have
    \begin{equation}\label{ineq:LD4-3}
        \mathbb E\left[e^{ 2\gamma t}|\Delta Y_t|^2 + (1 - \tilde q)\cdot \int_t^\infty e^{2\gamma s} |\Delta Z_s|^2 \dif s \right] \leq \tilde q\cdot  \mathbb E \Big[\int_t^\infty e^{2\gamma s} |\Delta z_s|^2\dif s\Big].
    \end{equation}
  Dividing by $1 - \tilde q$ on both sides of \eqref{ineq:LD4-3}, we have 
    \begin{equation*}
        \begin{split}
            \mathbb E\left[e^{2\gamma t}|\Delta Y_t|^2 + \int_t^\infty e^{2\gamma s} |\Delta Z_s|^2 \dif s \right] &\leq \mathbb E\left[\frac{1}{1 - \tilde q}\cdot e^{2\gamma t}|\Delta Y_t|^2 + \int_t^\infty e^{2\gamma s} |\Delta Z_s|^2 \dif s \right]\\
            &\leq q\cdot \mathbb E \Big[\int_t^\infty e^{2\gamma s}  \left(|\Delta z_s|^2 \right)\dif s\Big], 
        \end{split}
    \end{equation*}
    where the first inequality holds due to $\frac{1}{1 - \tilde q} > 2$, and the second holds by setting $q:= \frac{\tilde q}{1 - \tilde q}\in (0, 1)$. If one further chooses $\gamma$ large such that $\tilde{q}\in(0,1/3)$, hence one get $q\in(0,1/2)$, and complete the proof.
\end{proof}

\begin{lemma}\label{lemmad3}
    The solution $Y_0^x$ to the following BSDE
    \begin{equation}\label{eq:Yx0}
        Y_0^x = \int_0^\infty F_s(Z_s, Z_s) \dif s - \int_0^\infty Z_s \dif\widehat W_s,
    \end{equation}
with $F$ given by \eqref{eq:tilde-Y-Z-F} which satisfies Assumption \ref{assumption-F}, is deterministic and continuous function on $(0, 1)$, and is a viscosity solution to to the HJB equation~\eqref{eq:HJB}. 
\end{lemma}
\begin{proof}
We change the measure back to $\mathbb{P}$ and rewrite the BSDE \eqref{eq:Yx0} as 
\begin{equation*}
        Y_0^x = \int_0^\infty  e^{-\delta s }f^{u(X_s,Z_s)}(X_s) \dif s - \int_0^\infty Z_s \dif W_s.
\end{equation*}
By the definition of infinite-horizon BSDE, $Y_t$ is a $\mathcal F_t$-adapted process for $t\geq 0$. Obviously, $Y^x_0$ is measurable w.r.t. trivial $\sigma$-field $\mathcal F_0$, hence deterministic. 
The continuity of $Y^x_0$ for $x\in (0, 1)$ follows the inequality \eqref{ineq:continuityY} combined with the continuity of $F$ and $X^{x,u}_\cdot$ in $x$.
Let's then prove that $Y_0^x$ is the viscosity solution to the HJB equation \eqref{eq:HJB}. We only show the case for viscosity supersolution, and the subsolution property will follow from the same idea. As we have proved $Y^x_0$ is deterministic, consider the discounted process $Y^x_s = e^{-\delta s}v(X_s)$  with  $v(x) = Y_0^x$. Take $\varphi\in C^2(0, 1)$ being a test function such that $v - \varphi$ attains its (local) minimum value of zero at any $x\in (0, 1)$, we shall have the supersolution inequality holds. We prove the assertion by contradiction. Assume that 
\begin{equation*}
    \inf_{u\in U}\Big\{\big(b^{u}D_x\varphi\big)(x) + \frac{1}{2}\big(\sigma^2D_{xx}\varphi\big)(x) -\delta \varphi(x) + f^{u}(x)\Big\}>0.
    \end{equation*}
 Since $\varphi$ is smooth enough and $f$ is continuous, there exists $\epsilon>0$, such that for any $u\in U$, $y\in (x-\epsilon,x+\epsilon)$, we have $v(y)\ge \varphi(y)$ and 
\begin{equation}\label{ineq:proofD5}
    (b^{u}D_x\varphi\big)(y) + \frac{1}{2}\big(\sigma^2D_{xx}\varphi\big)(y) -\delta \varphi(y) + f^{u}(y)>0.
\end{equation}
Let $\tau:=\inf \{s\geq 0: X_s\notin (x-\epsilon,x+\epsilon)\}\wedge t$ for any $t\ge 0$, and consider the pair of stopped processes 
$$(Y^x_{s\wedge\tau},  Z_s\mathbf 1_{[0, \tau]}),\quad s\in[0,t],$$ 
which solves the following finite-horizon BSDE
$$v(x) = \xi_1  + \int_0^\infty  e^{-\delta s}f^{u(X_s,Z_s)}(X_s)\mathbf{1}_{[0,\tau]}\dif s - \int_0^\infty Z_s \mathbf{1}_{[0, \tau]} \dif W_s,$$
 where $\xi_1 := e^{-\delta \tau}v(X_\tau) $. 
On the other hand, consider another pair of stopped processes 
$$\left(e^{-\delta s}\varphi(X_{s\wedge \tau}), e^{-\delta s}(\sigma D_x\varphi ) (X_s)\mathbf 1_{[0, \tau]}\right),\quad s\in[0,t],$$
Apply the It\^o's formula to $e^{-\delta s}\varphi(X_{s})$, then
$$\varphi(x) = \xi_2 - \int_0^\infty \mathbf 1_{[0, \tau]} e^ {-\delta s}\Big(b^{u'} D_x\varphi + \frac{1}{2}\sigma^2 D_{xx} \varphi -\delta \varphi\Big)(X_{s}) \dif s - \int_0^\infty  e^{-\delta s}(\sigma D_x\varphi ) (X_s) \mathbf{1}_{[0, \tau]} \dif W_s, $$
where $\xi_2 := e^{-\delta \tau}\varphi(X_\tau)$. Then, since $v(X_s)\ge \varphi(X_s)$ for all $x\in[0,\tau]$,  we have $\xi_1\ge \xi_2$. In addition,  by \eqref{ineq:proofD5}, one has
\begin{equation*}
  f^{u\big(X_s, (\sigma D_x\varphi)(X_s)\big)}(X_s) > - \Big(b^{u} D_x\varphi + \frac{1}{2}\sigma^2 D_{xx} \varphi -\delta \varphi\Big)(X_{s}),\quad s\in[0,\tau).  
\end{equation*}
Hence, by Theorem~\ref{cmpbsde} (with a further argument of strict comparison principle, see e.g. \cite[Theorem 6.2.2]{pham2009continuous}), we have $v(x) > \varphi(x)$, which is a contradiction. Hence, we have the supersolution inequality.
\end{proof}

\subsection{Proof of Theorem~\ref{them:convegencePIA}}\label{proofthmPIA}

Let $v^n$ be the smooth solution to the Bellman-type ODE~\eqref{pia}, and recall the updated control at $n$th iteration
$$u^n(x) := \argmin_{u\in U}\big\{(b^u(x)D_xv^{n-1}(x) + f^u(x)\big\} = u\big(x, \sigma(x)D_xv^{n-1}(x)\big).$$
Define $X:= X^{x, u^*}$ as the solution to the SDE~\eqref{eq:SDE-X} started from $X_0= x$ and controlled by the optimal control policy $u^*$. 
Apply the It\^o's formula to $v^n(X_s)$, we have
\begin{equation}
\label{eq:dvnXs}
    \begin{split}
         \dif v^n(X_s)&= \left(\frac{1}{2}(\sigma^2D_{xx}v^n)(X_s) + (b^{u^*}D_xv^{n})(X_s) \right)\dif s +(\sigma D_xv^n)(X_s)\dif W_s\\
        &= \left((b^{u^*}D_xv^{n})(X_s) - (b^{u^n}D_xv^{n})(X_s) - f^{u^n}(X_s) + \delta v^n(X_s)\right)\dif s + (\sigma D_xv^n)(X_s)\dif W_s,  
    \end{split}
\end{equation}
where the second equality holds since $v^n$ is a solution to the ODE \eqref{pia}. Let's define for,
\begin{align*}
Y^n_s&:= v^n(X_s)\\
   Z_s^n &:= \big(\sigma D_xv^n\big)(X_s), \\
   F_s(z, Z)&:=  ( b^{u(X_s, z)}\sigma^{-1})(X_s) Z + f^{u(X_s, z)}(X_s).
\end{align*}
Hence, we can write \eqref{eq:dvnXs} as
\begin{equation}\label{eq:bsded-9}
    \dif Y^n_s =\Big((b^{u^*}D_xv^{n})(X_s) - F_s(Z^{n-1}_s,Z^n_{s})+\delta Y^n_s\Big)\dif s + Z^n_s\dif W_s.
\end{equation}
Let $\delta>0$ be a discounting rate, and define
\begin{equation}\label{eq:tilde-Y-Z-F}
   \widetilde{Y}^n_s := e^{-\delta s}Y^n_s,\quad \widetilde{Z}^{\cdot}_s:= e^{-\delta s}Z^{\cdot}_s,\quad \widetilde{F}_s(z,Z):= e^{-\delta s}F_s(e^{\delta s}z, e^{\delta s}Z).
\end{equation}
Then, we can rewrite the above BSDE \eqref{eq:bsded-9} as
\begin{equation*}
    \dif \widetilde{Y}^n_s = \Big(e^{-\delta s}(b^{u^*}D_xv^{n})(X_s) - \widetilde{F}_s(\widetilde{Z}^{n-1}_s,\widetilde{Z}^n_{s}) \Big) \dif s + \widetilde{Z}^n_s \dif W_s.
\end{equation*}
    Moreover, one observes that
    $$\mathbb{P}\Big(\lim_{t\to\infty}\widetilde{Y}^n_t = 0\Big) =  \mathbb P\Big(\lim_{t\to\infty}e^{-\delta t}v^n(X_t) = 0\Big) = 1, $$ 
    since the value function is finite and we have verified that  $X_t \in (0, 1)$ almost surely for any $t\geq 0$ (see, Proposition \ref{prop:2.1}).  Hence, we have
    \begin{equation}\label{eq:eq:tildeYnt}
    \widetilde{Y}^n_t = \int_{t}^{\infty} \Big(\widetilde{F}_s(\widetilde{Z}^{n-1}_s,\widetilde{Z}^n_{s}) - (b^{u^*}\sigma^{-1})(X_s)\widetilde{Z}^n_s\Big)\dif s - \int_{t}^{\infty}\widetilde{Z}^n_s \dif W_s.
    \end{equation}
 By Lemma~\ref{girsanov}, we change the measure to $\widehat{\mathbb{P}}$, and \eqref{eq:eq:tildeYnt} can be rewritten as
\begin{equation}\label{eq:bsde-tildeYn}
    \widetilde{Y}^n_t = \int_{t}^{\infty} \widetilde{F}_s(\widetilde{Z}^{n-1}_s,\widetilde{Z}^n_{s})\dif s - \int_{t}^{\infty}\widetilde{Z}^n_s \dif \widehat{W}_s.
    \end{equation}
Similarly, consider the following BSDE with $v^n$ and $v^{n-1}$ replaced by the value function of our stochastic control problem $v$ \eqref{problem} in \eqref{eq:bsde-tildeYn}:
\begin{equation}\label{eq:Yt}
  \widetilde{Y}_t = \int_t^\infty \widetilde{F}_s(\widetilde{Z}_s, \widetilde{Z}_s)\dif s - \int_t^\infty \widetilde{Z}_s \dif \widehat{W}_s.  
\end{equation}
Note that  for $\gamma<\delta$, we have $(\widetilde{Y}^n_\cdot, \widetilde{Z}^n_\cdot),\, (\widetilde{Y}_\cdot, \widetilde{Z}_\cdot) \in \mathbb{B}_{\gamma}[0,\infty]$, and by Assumption \ref{a4},
{\small
\begin{align*}
 | \widetilde{F}_s(Z_s, Z_s)  - & \widetilde{F}_s(Z_s^{n - 1}, Z_s^n)| \\
  & = \Big\vert( b^{u(X_s, e^{\delta s}Z_s)}\sigma^{-1})(X_s) \cdot Z_s + e^{-\delta s} f^{u(X_s, e^{\delta s}Z_s)}(X_s)\\
  &\qquad -( b^{u(X_s, e^{\delta s}Z^{n-1}_s)}\sigma^{-1})(X_s) \cdot Z^n_s +e^{-\delta s}  f^{u(X_s, e^{\delta s}Z^{n-1}_s)}(X_s) \Big\vert\\
  &\le \Big|\sigma^{-1}(X_s)Z_s\Big| \Big| b^{u(X_s, e^{\delta s}Z_s)}(X_s) - b^{u(X_s, e^{\delta s}Z^{n-1}_s)}(X_s) \Big| + \Big| b^{u(X_s, e^{\delta s}Z^{n-1}_s)}(X_s)\Big|\Big|Z_s-Z^n_{s}\Big|\\
  & \qquad +  e^{-\delta s} \Big|f^{u(X_s, e^{\delta s}Z_s)}(X_s) - f^{u(X_s, e^{\delta s}Z^{n-1}_s)}(X_s)\Big|\\
  & \le  e^{-\delta s}C \theta |e^{\delta s}Z_s -e^{\delta s}Z_s^{n-1}| + K|Z_s - Z_s^n| + e^{-\delta s} \theta|e^{\delta s}Z_s -e^{\delta s}Z_s^{n-1}|\\
  & = \theta (C+1) |Z_s - Z_s^{n - 1}| +  K|Z_s - Z_s^n|,
\end{align*}
}
where we have applied the fact that there is a constant $C>0$, $ |\sigma^{-1}(X_s)Z_s| = e^{-\delta s}|D_xv(X_s)| \le e^{-\delta s}C$ (given the derivative exists) by the property of the value function given in Proposition \ref{pp1}.  Hence, one may choose $\delta$ large enough such that Assumption \ref{assumption-F} holds for $\widetilde{F}$. Then, by applying Lemmas ~\ref{lemmaa5} and  \ref{lemmad3}, there is $0<\gamma<\delta$,
\begin{equation}\label{ineq:Y-Yn}
 \widehat{\mathbb E}\Big[e^{2\gamma t}|Y_t - Y^n_t|^2\Big] + \|Z - Z^n\|^2_{\widehat{\mathbb{L}}^2_\gamma} \leq q\|Z - Z^{n-1}\|^2_{\widehat{\mathbb{L}}^2_\gamma},  
\end{equation}
for any $t\ge 0$, where $\widehat{\mathbb{E}}$  and $\widehat{\mathbb{L}}^2_{\gamma}$ denote the expectation and $\mathbb{L}^2_{\gamma}$-norm under measure $\widehat{\mathbb{P}}$.
 In addition, by noting that the solution $Y^x_0$ and $Y^{x, n}_0$ are the value function $v(x)$ and approximation sequence at $n$th iteration $v^n(x)$ in the Algorithm \ref{PIA} (with a recursive argument, see e.g. \cite{kerimkulov2020exponential}), we have 
$$|v(x) - v^n(x)|^2 \leq q^n\widehat{\mathbb E}\Bigg[\int_0^\infty e^{2(\gamma-\delta) s}\Big|\sigma(X_s)\Big|^2\,\Big|D_xv(X_s) - D_xv^0(X_s)\Big |^2\dif s\Bigg].$$
Hence, we complete the proof.

\end{appendices}

\bibliographystyle{apalike}
\bibliography{Ref}

\end{document}